\documentclass[11pt,reqno]{amsart}
\usepackage{graphicx,epsfig}
\usepackage{epstopdf}
\usepackage{pdfpages}
\usepackage{amssymb,cite}
\usepackage{mathrsfs}
\usepackage{empheq}
\usepackage{cases}
\usepackage{amsthm,amsmath}
\usepackage{caption,lipsum}
\usepackage{stmaryrd}
\usepackage{tabularx}
\usepackage{color}
\usepackage{empheq,float}
\usepackage[margin=1.1in]{geometry}
\usepackage[colorlinks,linktocpage,linkcolor=blue]{hyperref}
\usepackage{backref}
\usepackage{enumitem}

\allowdisplaybreaks

\def\R{\mathbb{R}}
\def\Z{\mathbb{Z}}
\def\T{\mathbb{T}}

\def\P{\mathbb{P}}
\def\d{\mathrm{d}}

\def\F{\mathcal{F}}

\newcommand{\fe}{\mathrm{e}}

\newcommand{\bT}{{\mathbb T}}

\numberwithin{equation}{section}
\theoremstyle{plain}
\newtheorem{theorem}{Theorem}[section]
\newtheorem{lemma}[theorem]{Lemma}
\newtheorem{proposition}[theorem]{Proposition}
\newtheorem{corollary}[theorem]{Corollary}

\newtheorem{remark}{Remark}[section]
\numberwithin{equation}{section}
\numberwithin{table}{section}
\numberwithin{figure}{section}

\usepackage{titlesec} 
\titleformat{\section}{\vskip10pt\large\bfseries}{\thesection.}{0.5em}{\centering\vspace{5pt}}
\titleformat{\subsection}{\vskip10pt\normalsize\bfseries}{\thesubsection.}{0.5em}{}
\titleformat{\subsubsection}{\vskip10pt\normalsize\bfseries}{\thesubsection.}{0.5em}{}

\usepackage{lipsum,titletoc}

\titlecontents{section} 
[2.5em]                 
{\rmfamily}
{\contentslabel{2.3em}} 
{\hspace*{-2.3em}} {\titlerule*[1pc]{.}\contentspage}

\titlecontents{subsection} 
[4.8em]                 
{\rmfamily}             
{\contentslabel{2.3em}} 
{\hspace*{-2.3em}} {\titlerule*[1pc]{.}\contentspage}

\begin{document}
\title[]{An unfiltered low-regularity integrator \\
for the KdV equation with solutions below ${\bf H^1}$}

\author[]{Buyang Li\,\,}
\address{\hspace*{-12pt}Buyang Li: 
Department of Applied Mathematics, The Hong Kong Polytechnic University,
Hung Hom, Hong Kong. {\it Email address: \tt buyang.li@polyu.edu.hk}}

\author[]{\,\,Yifei Wu}
\address{\hspace*{-12pt}Y.~Wu: Center for Applied Mathematics, Tianjin University, 300072, Tianjin, China.
\newline 
{\it Email address: \tt yifei@tju.edu.cn}}

\date{}

\dedicatory{}

\begin{abstract}

This article is concerned with the construction and analysis of new time discretizations for the KdV equation on a torus for low-regularity solutions below $H^1$. New harmonic analysis tools, including new averaging approximations to the exponential phase functions, new frequency decomposition techniques, and new trilinear estimates of the KdV operator, are established for the construction and analysis of time discretizations with higher convergence orders under low-regularity conditions. In addition, new techniques are introduced to establish stability estimates of time discretizations under low-regularity conditions without using filters when the energy techniques fail. The proposed method is proved to be convergent with order $\gamma$ (up to a logarithmic factor) in $L^2$ under the regularity condition $u\in C([0,T];H^\gamma)$ for $\gamma\in(0,1]$.\\[10pt]  
 {\bf Keywords:} KdV equation, rough data, low-regularity integrator, stability and error estimates \\[10pt]
{\bf AMS Subject Classification:} 35Q53, 35B99, 65M12, 65M15. 
\end{abstract}

\maketitle

\tableofcontents

\section{Introduction}

The Korteweg--De Vries (KdV) equation is a nonlinear dispersive partial differential equation that describes many physical phenomena, including shallow water waves, ion acoustic waves in plasmas, acoustic waves on crystal lattices, and so on. The development of computational methods for the KdV equation has attracted much attention.

It is known that the KdV equation, either on a torus or on the whole space, is globally well-posed in $H^s$ with $s\ge 0$, i.e., there exists a unique solution in $C ([0,T];H^s)$ for any initial value in $H^s$; see \cite{BaLlTi-2011-CPAM, KaTo, KiVi-2019,Zhou-1997-IMRN}. However, classical time discretizations for the KdV equation, including finite difference methods, splitting methods, discontinuous Galerkin methods, and classical exponential integrators, generally require much higher regularity for the numerical solutions to converge with certain orders, i.e., these methods typically require $u\in C([0,T];H^{3})$ and $u\in C([0,T];H^{6})$ to have first- and second-order convergence in $L^2$, respectively. The error estimates under these regularity conditions (or stronger conditions) for the above-mentioned classical time discretizations have been established, for example, in \cite{splittingJCP,Spectralkdv1,Fourierkdv,Spectralkdv2,splitting0,splitting1,splitting2,DG2,Ostermann-Su-2020}. Such regularity conditions are not mere technical conditions required in the error analysis. When the solution of the KdV equation does not have the required regularity, its numerical approximations by the classical time discretizations generally have reduced order of convergence.

In practice, the solutions of the KdV equation may be rough due to measurement or randomness of the data \cite{BD2009,Gubinelli-2012}. To address the numerical approximation to nonsmooth solutions, some low-regularity exponential integrators based on resonance analysis were recently developed to relax the regularity requirement in solving nonlinear dispersive equations. Such low-regularity integrators based on resonance analysis were initially introduced by Hofmanov\'{a} \& Schratz \cite{Hofmanova-Schratz-2017} and Ostermann \& Schratz \cite{Ostermann-Schratz-FoCM,Ostermann-Schratz-FoCM3} for solving the KdV equation and the nonlinear Schr\"odinger (NLS) equation, respectively. In particular, for the KdV equation, the low-regularity integrator proposed in \cite{Hofmanova-Schratz-2017} can have first-order convergence in $H^{1}$ for $u\in C([0,T];H^{3})$. Wu \& Zhao \cite{Wu-Zhao-IMA} showed that another method outlined in \cite{Hofmanova-Schratz-2017} can have second-order convergence in $H^\gamma$ for $u\in C([0,T];H^{\gamma+4})$ with $\gamma\ge 0$. 
In a more recent article, Wu \& Zhao \cite{WuZhao-BIT} proposed two embedded low-regularity integrators for the KdV equation, which have first-order convergence in $H^\gamma$ for $u\in C([0,T];H^{\gamma+1})$ with $\gamma>\frac12$ and second-order convergence $H^\gamma$ for $u\in C([0,T];H^{\gamma+3})$ with $\gamma\ge0$, respectively. The minimal regularity requirement for the convergence analysis of these unfiltered algorithms for the KdV equation is $u\in C([0,T];H^\gamma)$ for $\gamma> 3/2$. This condition naturally arises in the energy type of stability analyses.

The convergence of a fully discrete finite difference method was proved in \cite{CLR-2020} for $u\in C([0,T];H^\gamma)$ with $\gamma \ge 3/4$ under the CFL condition $\Delta t\le \Delta x^3$, where $\Delta t$ and $\Delta x$ denote the stepsize and mesh size in the temporal and spatial discretizations, respectively. The CFL condition in a finite difference method plays a similar role as the filters in a spectral method, i.e., to improve the stability of the method under low-regularity conditions.  In the case of $\gamma=3/4$, the method is proved convergent with order $1/42$. Since the convergence analysis relies on the smoothing effect on $\R$, the proof cannot be extended to the torus $\T$. 

Similarly, the development of low-regularity integrators for the NLS equations can be found in \cite{Knoller-Ostermann-Schratz-2019,Wu-Yao-2021,Li-Wu-2021,Ostermann-Yao-2022}. The minimal regularity requirement for the convergence analysis of the unfiltered algorithms for the NLS equation is $u\in C([0,T];H^\gamma)$ for $\gamma> d/2$, where $d$ is the dimension of space. This condition also arises in the energy stability analyses, which require using the Kato--Ponce inequality $\|fg\|_{H^\gamma}\lesssim \|f\|_{H^\gamma}\|g\|_{H^\gamma}$ with $\gamma>\frac d2$. The question of whether any convergence rates can be achieved for rough solutions $u\in C([0,T];H^\gamma)$, with an arbitrary small $\gamma>0$, remained open for a long time for both the KdV equation and the NLS equation. 

The convergence of numerical solutions for rough solutions in $C([0,T];H^\gamma)$ with an arbitrary small $\gamma>0$ was addressed by Ostermann, Rousset \& Schratz \cite{ORS-JEMS} and Rousset \& Schratz \cite{RS-PAM-2022} for the NLS and KdV equations, respectively, by introducing and utilizing the discrete Bourgain spaces. In particular, for the KdV equation, Rousset \& Schratz \cite{RS-PAM-2022} proposed three filtered time discretizations for the KdV equation on the torus $\T$, including a filtered exponential integrator, a filtered Lie splitting method, and a filtered version of the resonance based scheme, and proved the convergence of order $\frac\gamma3$ for the three methods under the regularity condition $u\in  C([0,T];H^\gamma)$ with $\gamma\in(0,3]$. 
The convergence analysis in \cite{RS-PAM-2022} is based on the combination of filters in the algorithms and the discrete Bourgain spaces in the analysis. Since the filters in these algorithms truncate the numerical solutions to frequencies below $\tau^{-\frac13}$, and such frequency-truncated functions approximate the original functions in $H^\gamma$ with an error bound of $O(\tau^{\frac{\gamma}{3}})$, it follows that the convergence of such filtered algorithms is limited to order $\frac{\gamma}{3}$ for the KdV equation under the regularity condition $u\in C([0,T];H^\gamma)$. 

This article is concerned with the construction and analysis of new time discretizations for the KdV equation on a torus, 
\begin{equation}\label{model}
 \left\{\begin{aligned}
& \partial_tu(t,x)+\partial_x^3u(t,x)
 =\frac{1}{2}\partial_x(u(t,x)^2),
&&  x\in\bT =[0,2\pi] \,\,\,\mbox{and}\,\,\, t\in(0,T],\\
 &u(0,x)=u^0(x),&& x\in\bT,
 \end{aligned}\right.
\end{equation}
for low-regularity solutions below $H^1$, i.e., the initial value $u^0$ is in $H^\gamma$ with $\gamma\in(0,1]$ and therefore the solution $u$ is in $C([0,T];H^\gamma)$ with $\gamma\in(0,1]$. 
One of the main difficulties in the construction and analysis of low-regularity integrators for nonlinear dispersive equations is the approximation of exponential functions with imaginary powers, say $\fe^{is\phi}$, based on a certain decomposition of the phase function $\phi=\phi_1+\phi_2$. The approximation of such  exponential functions with imaginary powers were typically based on the following techniques: 
\begin{align}\label{exp-approx-1}
e^{is\phi}=e^{is\phi_1}+O(s|\phi_2|) \quad \mbox{or } \quad \fe^{is\phi}=e^{is\phi_1}+e^{is\phi_2}-1+O(\min\{s|\phi_1|,s|\phi_2|\}) , 
\end{align}
see \cite{Knoller-Ostermann-Schratz-2019,Hofmanova-Schratz-2017,WuZhao-BIT,Wu-Yao-2021,Li-Wu-2021,Ostermann-Yao-2022} and the references therein. The remainders in these types of approximations are still too large to obtain error estimates for rough solutions $u\in C([0,T];H^\gamma)$ with $\gamma\in(0,1]$. In this article we introduce a new averaging approximation technique: 
\begin{align}\label{exp-approx-2}
M_\tau\big(\fe^{is(\phi_1+\phi_2)}\big)=M_\tau\big(\fe^{is\phi_1}\big)M_\tau\big(\fe^{is\phi_2}\big)+O\left(\min\left\{\left|\frac{\phi_1}{\phi_2}\right|,\left|\frac{\phi_2}{\phi_1}\right|,s|\phi_1|,s|\phi_2|\right\}\right),
\end{align}
where $M_\tau(f)$ denotes the average of a function $f$ in the interval $[0,\tau]$; see Lemma \ref{lem:average2}. The remainder in this approximation is smaller than the remainders in \eqref{exp-approx-1}. In particular, the additional upper bounds $|\phi_1/\phi_2|$ and $|\phi_2/\phi_1|$ for the remainder are important for us to obtain error estimates in the rough case by using harmonic analysis techniques. 
 
Moreover, it is known that the combination of filters and discrete Bourgain spaces in \cite{RS-PAM-2022} has played an important role in establishing the stability of numerical approximations to rough solutions. In this article, we develop new techniques which can be used to establish stability estimates under such low-regularity conditions when the energy techniques fail and filters are not used. More specifically, instead of using energy techniques locally in time, we define a temporally continuous function $\mathscr{V}(t)$ which equals the numerical solution $v^n$ at the discrete time levels $t_n$, $n=0,1,\dots,L$, and satisfies an integral formulation of the continuous KdV equation globally in time up to a perturbation term, i.e.,
$$
\mathscr V(t)
=
  v^0 +\frac12 \int_{0}^t \fe^{s\partial_x^3}\partial_x\left(\fe^{-s\partial_x^3}\mathscr V(s)\right)^2\,\d s 
  + \mathcal R(t) \quad\mbox{for}\,\,\, t\in (0,T] ,
$$
where the perturbation term $\mathcal{R}(t)$ can be defined piecewisely on each subinterval $(t_n,t_{n+1}]$ according to the definition of the time discretization on this subinterval. 
The specific form of the perturbation term for the low-regularity integrator constructed in this article is given in \eqref{def-R(t)}. 
In the absence of the perturbation term $ \mathcal R(t)$, the solution of the integral equation above coincides with the solution to the KdV equation. The continuous formulation of the numerical scheme allows us to apply low- and high-frequency decomposition in estimating the stability with respect to the perturbation, which can significantly weaken the regularity conditions compared with the energy approach of stability estimates used in the literature. 

In addition, we establish some new harmonic analysis tools, including new frequency decomposition techniques (Lemma \ref{lem:a4-est}) and new trilinear estimates of the KdV operator (Proposition \ref{lem:tri-linear}), which can be used to construct and analyze low-regularity integrators without using filters and therefore significantly improves the convergence order to three times of the order with filters. More specifically, by using the averaging approximation technique, the stability analysis based on the continuous formulation of the numerical scheme, and the new harmonic analysis tools established in this article, we prove that the proposed method is convergent with order $\gamma$ (up to a logarithmic factor) under the regularity condition $u\in C([0,T];H^\gamma)$ for $\gamma\in(0,1]$. 

For the convenience of readers, we present the numerical scheme and the main theoretical result below. 
Let $t_n=n\tau$, $n=0,1,\dots,L=T/\tau$, be a partition of the time interval $[0,T]$ with stepsize $\tau=T/L$. The low-regularity integrator constructed in this article for the KdV equation \eqref{model} reads: For given $u^n\in H^\gamma$, find $u^{n+1}\in H^\gamma$ by    
\begin{align}\label{numer-formula-u}
u^{n+1}=&\fe^{-\tau\partial_x^3}u^n+F[u^n] +H[u^n] \quad \mbox{for}\,\,\, n=0,1\ldots,L-1 ,  
 \end {align}
where 
\begin{align*}
F[u^n]=&\,\frac16 \P\big[\big(\fe^{-\tau\partial_x^3}\partial_x^{-1}u^n\big)^2\big]
-\frac16 \fe^{-\tau\partial_x^3} \P\big[\big(\partial_x^{-1}u^n\big)^2\big] , \\
H[u^n]
=&\,\frac13  \P \Big[ \big(\fe^{-\tau\partial_x^3}\partial_x^{-1}u^n\big)\, \partial_x^{-1} F[u^n]\Big]
\notag \\
&\, +  \frac{\tau}{9} \fe^{-\tau\partial_x^3} (\partial_x^{-1} u^n) \, \P_{0}[ (u^n)^2 ]  \\
&\,
- \frac{1}{54}\fe^{(s-\tau)\partial_x^3}\partial_x^{-1}\big[ (\fe^{-s\partial_x^3}\partial_x^{-1}u^n)^3 \big] \bigg|_{s=0}^{s=\tau} \\
&\,
-\frac1{27\tau}  \fe^{(s-\tau)\partial_x^3}\partial_x^{-2}\Big[ (\fe^{-(s-\tau)\partial_x^3}\partial_x^{-2}F[u^n]) \,
(\fe^{-s\partial_x^3}\partial_x^{-1}u^n)\Big]\Big|_{s=0}^{s=\tau}.
\end{align*}

 The convergence of the numerical solution to the solution of the KdV equation is guaranteed by the following theorem.

\begin{theorem}\label{main:thm1}
Let $\gamma\in (0,1]$ and $u\in C([0,T];H^{\gamma})$ with initial value satisfying $\int_\T u^0\,dx=0$. Then there exist positive constants $\tau_0\in(0,\frac12]$ and $C$ such that for $\tau\in(0,\tau_0]$ the numerical solution given by \eqref{numer-formula-u} has the following error bound:
\begin{equation}\label{est:error}
\max_{1\le n\le L}  \|u(t_n,\cdot)-u^{n}\|_{L^2}
  \le C \tau^\gamma\ln (1/\tau) ,
\end{equation} 
where the constants $\tau_0$ and $C$ depend only on $\|u^0\|_{H^\gamma}$, $\gamma$ and $T$. 
\end{theorem}

\begin{remark}\label{Remark-THM}
\upshape 
Without loss of generality, we can assume that $\int_\T u^0\,dx=0$, i.e., $\P_0u^0=0$ and $\P u^0=u^0$, where 
$$
\P_0u^0=\frac{1}{2\pi}\int_{\T}u^0\d x \quad\mbox{and}\quad \P u^0=u^0-\P_0u^0
$$ 
are the zero-mode and nonzero-mode projections of $u^0$, respectively. Otherwise we can consider the function 
$$
\tilde u(t,x) :=  u(t,x- t\,\P_0u^0) -\P_0u^0 ,
$$
which satisfies the KdV equation in \eqref{model} with initial value $\tilde u^0= \P u^0$, which satisfies $\int_\T \tilde u^0\d x=0$. 

\end{remark}

The rest of this article is organized as follows. Some basic notations and preliminary results are presented in Section \ref{sec:proof}. Several new tools for the construction and analysis of low-regularity integrators are presented in Section \ref{section:tools}, including a logarithmically growing trilinear estimate on $L^2$, the averaging approximation of exponential functions with imaginary powers, and new trilinear estimates associated to the KdV operator. The construction of the low-regularity integrator is presented in Section \ref{sec:numer-method}, and the reduction of the proposed numerical scheme to a continuously formulated perturbed KdV equation is presented in Section \ref{section:reduction}. 
The consistency estimates for the local and global errors are presented in Sections \ref{section:local-error} and \ref{section:global-error}, respectively. The stability estimates using low- and high-frequency decompositions are presented in Section \ref{sec:regularity}. The error estimates (i.e., proof of Theorem \ref{main:thm1}), which combine the consistency and stability estimates, are presented in Section \ref{sec:globalstability}. 
Numerical experiments and conclusions are presented in Sections \ref{sec:numerical} and \ref{sec:conclusion}, respectively.

\section{Notations and preliminary results}\label{sec:proof}

In this section we present the basic notations to be used in this article, as well as some preliminary results which were known in the literature and are frequently used in this article. 

\subsection{Baisc notations}\label{subsec1}
For convenience, we adopt the following notations which are widely used in harmonic analysis and partial differential equations: 
\begin{enumerate}

\item[(i)] 
For a function $f(t,x)$ which depends on $t$ and $x$, we simply denote $f(t)=f(t,\cdot)$. 

\item[(ii)]
We denote $\langle k\rangle=(1+|k|^2)^{\frac12}$ for $k\in\Z$. 

\item[(iii)]
We denote by $C$ a generic positive constant which may be different at different occurrences, possibly depending on $\|u\|_{C([0,T];H^\gamma)}$ and $T$, 
but is independent of the stepsize $\tau$ and time level $n$.  

\item[(iv)]
We denote by $A\lesssim B$ or $B\gtrsim A$ the statement ``$A\leq CB$ for some constant $C>0$''.

\item[(v)]
We denote by $A\sim B$ the statement ``$C^{-1}B\le A\leq CB$ for some constant $C>0$''. Namely, $A\sim B$ is equivalently to $A\lesssim B\lesssim A$. 

\item[(vi)] 
We denote by $A\ll B$ or $B\gg A$ the statement $A\le C^{-1}B$ for some sufficiently large constant $C$ (which is independent of $\tau$ and $n$). 

\item[(vii)] 
The notation $a+$ stands for $a+\epsilon$ with an arbitrary small $\epsilon>0$, 
and $a-$ stands for $a-\epsilon$ with an arbitrary small $\epsilon>0$.


\end{enumerate}

With the notations above, we often decompose a subset $E\subset \Z^2=\{(k_1,k_2):k_1,k_2\in\Z\}$ into two parts, i.e., $E=E_1\cup E_2$, with 
$$
E_1=\{(k_1,k_2)\in E: |k_1|\ll |k_2|\}
\quad\mbox{and}\quad
E_2=\{(k_1,k_2)\in E: |k_1|\gtrsim |k_2|\} . 
$$
This means that we consider the decomposition with
$$
E_1=\{(k_1,k_2)\in\Z^2: |k_1|< c|k_2|\}
\quad\mbox{and}\quad
E_2=\{(k_1,k_2)\in\Z^2: |k_1|\ge c |k_2|\} ,
$$
where $c$ is some sufficiently small constant (independent of $\tau$ and $n$) which can satisfy the requirement in our analysis.  

\subsection{Fourier transform}\label{subsec1}
The inner product and norm of $L^2(\T)$ is defined by 
$$
\langle f,g\rangle = \int_\bT f(x)\overline{g(x)}\,\d x
\quad\mbox{and}\quad
\|f\|_{L^2(\bT)}:=\sqrt{\langle f,f\rangle} .
$$
The Fourier transform of a function $f\in L^2(\T)$ is defined by
$$
\mathcal{F}_k[f] = \frac1{2\pi}\displaystyle\int_{\bT}
e^{- i kx}f( x)\,\d x . 
$$
For the simplicity of notation, we also denote $\hat{f}_k=\mathcal{F}_k[f]$ and $f=\mathcal{F}_k^{-1}[\hat f_k]$. The following standard properties of the Fourier transform are well known: 
\begin{align*}
f(x) &=\sum\limits_{k\in\Z}\hat{f}_k \fe^{i kx} && \mbox{(Fourier series expansion)} \\ 
\|f\|_{L^2(\bT)} 
&= \sqrt{2\pi}\Big(\sum\limits_{k\in\Z}|\hat f_k|^2 \Big)^\frac12 && \mbox{(Plancherel's identity)}  \\
\langle f,g\rangle  &  =2\pi \sum\limits_{k\in\Z} \hat f_k \overline{\hat g_k} && \mbox{(Parseval's identity)}   \\
\mathcal{F}_k[fg]  &=\sum\limits_{k_1+k_2=k}\hat f_{k_1}\hat g_{k_2} && 
\mbox{(Conversion of products to convolutions)} 
\end{align*}

The Sobolev space $H^s(\bT)$, with $s\in\R$, consists of generalized functions $f=\sum\limits_{k\in\Z}\hat{f}_k \fe^{i kx} $  such that $\|f\|_{H^s}<\infty$, where 
$$
\|f\|_{H^s} :=  \sqrt{2\pi}\bigg(\sum_{k\in\Z} (1+|k|^2)^{s} |\hat f_k|^2 \bigg)^{\frac12} . 
$$
The operator $J^s=(1-\partial_{x}^2)^\frac s2: H^{s_0}(\T)\rightarrow H^{s_0-s}(\T)$, with $s_0,s\in\R$, is defined as   
$$
J^s f = \sum_{k\in\Z} (1+|k|^2)^\frac s2 \hat f_k \fe^{ikx} \quad\forall\, f\in H^{s_0}(\T) , 
$$
which satisfies that $\|f\|_{H^s(\bT)} = \|J^s f \|_{L^2(\bT)} $.

\subsection{Projection operators}
\label{section:projection}

For any real number $N\ge 0$, we define the Littlewood--Paley projections $\P_{\le N}: L^2(\bT)\rightarrow L^2(\bT)$ and $\P_{> N}: L^2(\bT)\rightarrow L^2(\bT)$ as 
\begin{align*}
\P_{\le N} f := \mathcal{F}^{-1}_k\big( 1_{|k|\le N} \mathcal{F}_k[f] \big) = \sum_{|k|\le N}\hat f_k \fe^{ikx} ,\\
\P_{> N} f := \mathcal{F}^{-1}_k\big( 1_{|k|> N} \mathcal{F}_k[f] \big) = \sum_{|k|> N}\hat f_k \fe^{ikx} . 
\end{align*}
We denote $\P_0=\P_{\le 0}$ and $\P=\P_{>0}$, which are called zero-mode and nonzero-mode projections, respectively, satisfying the following identities:  
$$
\P_{0}f = \frac1{2\pi}\int_\T f\,\d x 
\quad\mbox{and}\quad
\P f(x)=f(x)-\frac1{2\pi}\int_\T f\,\d x . 
$$
The operator $\partial_x^{-1}: L^2(\bT)\rightarrow H^{1}(\bT)$ is defined by 
\begin{equation*} 
\mathcal{F}_k[\partial_x^{-1}f]
=\Bigg\{ \aligned
    &(ik)^{-1}\hat f_k  &&\mbox{for}\,\,\, k\ne 0,\\
    &0 &&\mbox{for}\,\,\, k= 0.
   \endaligned
\end{equation*}
This operator has a natural extension $\partial_x^{-1}: H^s(\bT)\rightarrow H^{s+1}(\bT)$ for all $s\in\R$. Moreover, the following relation holds: 
\begin{align*}
\partial_x^{-1} \partial_x f =  \partial_x \partial_x^{-1} f=\P f . 
\end{align*}

For functions restricted to low frequency or high frequency, the following Bernstein's inequalities hold for any real numbers $s\ge s_0$:
\begin{align}\label{Bernstein}
\begin{aligned}
\|\P_{\le N} f\|_{H^{s}} &\lesssim N^{s-s_0} \|\P_{\le N} f\|_{H^{s_0}} &&\forall\, f\in H^{s_0}(\T) , \\
\|\P_{> N} f\|_{H^{s_0}} &\lesssim N^{s_0-s} \|\P_{> N} f\|_{H^{s}} &&\forall\, f\in H^{s}(\T) . 
\end{aligned}
\end{align}

\subsection{The Kato--Ponce inequality}

The Kato--Ponce inequality will be frequently used in this paper. The result was originally proved in \cite{Kato-Ponce} and then extended to the endpoint case in \cite{BoLi-KatoPonce, Li-KatoPonce} recently.
\begin{lemma}[Kato--Ponce inequality] \label{lem:kato-Ponce} 
For $s>0$, $1<p\le \infty$, $1<p_1,p_3< \infty$ and $1<p_2,p_4 \le \infty$ satisfying $\frac1p=\frac1{p_1}+\frac1{p_2}$ and $\frac1p=\frac1{p_3}+\frac1{p_4}$, the following inequality holds:
\begin{align*}
\big\|J^s(fg)\big\|_{L^p}\le C\big( \|J^sf\|_{L^{p_1}}\|g\|_{L^{p_2}}+ \|J^sg\|_{L^{p_3}}\|f\|_{L^{p_4}}\big),
\end{align*}
where the constant $C>0$ depends on $s,p,p_1,p_2,p_3,p_4$. If $s>\frac1p$ then the following inequality holds: 
\begin{align*}
\big\|J^s(fg)\big\|_{L^p}\le C \|J^sf\|_{L^{p}}\|J^sg\|_{L^{p}},
\end{align*}
where the constant $C>0$ depends on $s$ and $p$.
\end{lemma}
\begin{remark}\upshape 
The Kato--Ponce inequality was originally established in whole space $\R$, but it also holds for periodic functions on $\bT$. This can be proved by using Stein's extension operator $E:L^1(\T)\rightarrow L^1(\R)$, which is bounded from $W^{s,p}(\T)$ to $W^{s,p}(\R)$ for all $s\ge 0$ and $1< p<\infty$. Therefore,
\begin{align*}
\big\|J^s(fg)\big\|_{L^{p}(\T)} \sim \big\|fg\big\|_{W^{s,p}(\T)}
&\lesssim 
\big\|Ef\,Eg\big\|_{W^{s,p}(\R)} \\
&\lesssim 
C\big( \|Ef\|_{W^{s,p_1}(\R)}\|Eg\|_{L^{p_2}(\R)}
+ \|Eg\|_{W^{s,p_3}(\R)}\|Ef\|_{L^{p_4}(\R)}\big) \\
&\lesssim 
C\big( \|f\|_{W^{s,p_1}(\T)}\|g\|_{L^{p_2}(\T)}
+ \|g\|_{W^{s,p_3}(\T)}\|f\|_{L^{p_4}(\T)}\big) \\
&\lesssim 
C\big( \|J^sf\|_{L^{p_1}(\T)}\|g\|_{L^{p_2}(\T)}
+ \|J^sg\|_{L^{p_3}(\T)}\|f\|_{L^{p_4}(\T)}\big). 
\end{align*}

\end{remark}

In addition to the Kato--Ponce inequality, we will also use the following basic inequality (as a result of the H\"older and Sobolev embedding inequalities): 
\begin{align}\label{lem:ga-aga}
\|fg\|_{L^2}\lesssim \|f\|_{H^\gamma}\|g\|_{H^{a(\gamma)}}
\quad\mbox{for}\,\,\, f\in H^\gamma\,\,\,\mbox{and}\,\,\, g\in H^{a(\gamma)} ,\,\,\,\mbox{with}\,\,\,
\gamma\in [0,1] , 
\end{align}
where 
\begin{equation}\label{def:a-gamma}
a(\gamma)=
\left\{ \aligned
    &\mbox{$\frac12$}+ &&\mbox{when } \gamma= 0,\\
    &\mbox{$\frac12$}-\gamma  &&\mbox{when } \gamma\in (0,\mbox{$\frac12$}),\\
    &0+  &&\mbox{when } \gamma=\frac12,\\
    &0  &&\mbox{when } \gamma\in (\mbox{$\frac12$},1].
   \endaligned
   \right.
\end{equation}

\subsection{Integration by parts}

The following integration-by-parts formula is closely related to the nonlinearity of the KdV equation, and therefore will be used frequently. A proof of this result can be found in \cite{WuZhao-BIT}.
\begin{lemma}[Integration by parts] \label{lem:1-form}
Let $s\ge s_0\ge 0 $ and consider the space-time functions $f(t,x)$ and $g(t,x)$ satisfying  
$\P_0f(t)=\P_0g(t)=0$ for $t\in[s_0,s]$. Then the following formula holds: 
\begin{align*}
   &\hspace{-10pt} \int_{s_0}^s \fe^{t\partial_x^3}\left(\fe^{-t\partial_x^3}f(t)\cdot \fe^{-t\partial_x^3}g(t)\right)\,\d t\\
  =&
  \frac13 \fe^{t\partial_x^3}\partial_x^{-1}\left(\fe^{-t\partial_x^3}\partial_x^{-1}f(t)\cdot \fe^{-t\partial_x^3}\partial_x^{-1}g(t)\right)\Big|_{t=s_0}^{t=s}
  + \frac{1}{2\pi} \int_{s_0}^s\!\!\int_\T f(t)\>g(t)\,\d x\d t\\
  &-\frac13 \int_{s_0}^s \fe^{t\partial_x^3}\partial_x^{-1}\Big(\fe^{-t\partial_x^3}\partial_x^{-1}\partial_tf(t)\cdot \fe^{-t\partial_x^3}\partial_x^{-1}g(t)+\fe^{-t\partial_x^3}\partial_x^{-1}f(t)\cdot \fe^{-t\partial_x^3}\partial_x^{-1}\partial_tg(t)\Big)\,\d t.
  \end{align*}
\end{lemma}

\section{New tools for the construction of low-regularity integrators}\label{section:tools}
In this section, we establish several new technical tools which can be used to construct and analyze low-regularity integrators with improved convergence orders. These technical tools are used in the following sections in estimating the local truncation errors and establishing the stability estimates.


\subsection{A logarithmically growing trilinear estimate on $L^2$} 

The following trilinear estimate will be used multiple times in the analysis of local truncation errors. 
\begin{lemma}\label{lem:ln-loss}
For any $f,g,h\in L^2$ we define $M(f,g,h)$ to be a function determined by its Fourier coefficients 
$$
\mathcal{F}_k[M(f,g,h)] = \sum_{k_1+k_2+k_3=k}m(k,k_1,k_2,k_3)
\hat f_{k_1}\>\hat g_{k_2}\>\hat h_{k_3}  ,
$$
where $m$ is a multiplier satisfying the following estimate {\rm(}for some constants $\theta_0>0$ and $A\ge 2${\rm):} 
\begin{align}\label{est:mk}
|m(k,k_1,k_2,k_3)|\le A^{\theta} \big[ \langle k\rangle^{-\frac12-\theta}\langle k_3\rangle^{-\frac12-\theta}+\langle k_2\rangle^{-\frac12-\theta} \langle k_3\rangle^{-\frac12-\theta}\big] 
\quad \forall\, \theta\in [0,\theta_0]. 
\end{align}
Then the multilinear operator $M:L^2\times L^2\times L^2\rightarrow L^2$ is well defined and satisfies the following estimate:
$$
\|M(f,g,h)\|_{L^2}\lesssim  (\ln A)\, \|f\|_{L^2}\|g\|_{L^2}\|h\|_{L^2}.
$$
\end{lemma}
\begin{proof}
By the duality between $L^2$ and itself, it suffices to prove $|\langle M(f,g,h), \varphi\rangle| \lesssim \ln A$ for any functions $f,g,h,\varphi\in L^2$ such that $ \|f\|_{L^2}=\|g\|_{L^2}=\|h\|_{L^2}=\|\varphi\|_{L^2}=1$. 
By the Parseval identity and \eqref{est:mk}, the following result holds: 
\begin{align}\label{T-varphi1}
|\langle M(f,g,h), \varphi\rangle| 
&= 2\pi \bigg| \sum\limits_{k\in\Z}  \mathcal{F}_k[M(f,g,h)]  \overline{\mathcal{F}_k[\varphi]}  \bigg| \notag\\ 
&\lesssim  \sum\limits_{k\in\Z} \sum_{k_1+k_2+k_3=k} 
A^{\theta} \big[ \langle k\rangle^{-\frac12-\theta}\langle k_3\rangle^{-\frac12-\theta}+\langle k_2\rangle^{-\frac12-\theta} \langle k_3\rangle^{-\frac12-\theta}\big]
|\hat f_{k_1}||\hat g_{k_2}||\hat h_{k_3}||\hat\varphi_k| .
\end{align}
Let $\tilde f$, $\tilde g$, $\tilde h$ and $\tilde\varphi$ be the functions with the nonnegative Fourier coefficients 
$|\hat f_{k}|$, $|\hat g_{k}|$, $|\hat h_{k}|$ and $|\hat\varphi_k|$, respectively. These functions satisfy that  
$$
\|\tilde f\|_{L^2}=\|\tilde g\|_{L^2}=\|\tilde h\|_{L^2}=\|\tilde \varphi\|_{L^2}=
\|f\|_{L^2}=\|g\|_{L^2}=\|h\|_{L^2}=\|\varphi\|_{L^2}=1 .
$$
If we define $\tilde m(D,\theta)$ to be the linear operator associated to the multiplier $\tilde m(k,\theta)=A^{\frac{\theta}2}\langle k\rangle^{-\frac12-\theta}$. Namely, 
$$
\mathcal{F}_k[\tilde m(D,\theta) v ]= A^{\frac{\theta}2}\langle k\rangle^{-\frac12-\theta} \hat v_k . 
$$
Then \eqref{T-varphi1} can be written as  
\begin{align}\label{T-varphi}
|\langle M(f,g,h), \varphi\rangle| 
&\lesssim \sum\limits_{k\in\Z} \sum_{k_1+k_2+k_3=k} 
\tilde m(k,\theta)\tilde m(k_3,\theta)|\hat f_{k_1}||\hat g_{k_2}||\hat h_{k_3}||\hat\varphi_k| \notag\\
&\quad\, 
+\sum\limits_{k\in\Z} \sum_{k_1+k_2+k_3=k}  \tilde m(k_2,\theta)\tilde m(k_3,\theta) 
|\hat f_{k_1}||\hat g_{k_2}||\hat h_{k_3}||\hat\varphi_k| \notag\\
&= (\tilde f\,\tilde g\,\tilde m(D,\theta)\tilde h , \tilde m(D,\theta)\tilde \varphi)
+ (\tilde f\,[\tilde m(D,\theta)\tilde g] [\tilde m(D,\theta)\tilde h], \tilde \varphi) \notag\\
&\lesssim \|\tilde f\|_{L^2} \|\tilde g\|_{L^2} \|\tilde m(D,\theta)\tilde h\|_{L^\infty} \|\tilde m(D,\theta)\tilde \varphi\|_{L^\infty} \notag\\
&\quad\, 
+\|\tilde f\|_{L^2} \|\tilde m(D,\theta)\tilde g\|_{L^\infty} \|\tilde m(D,\theta)\tilde h\|_{L^\infty} \|\tilde \varphi\|_{L^2} .
\end{align} 
It remains to prove the following result:
\begin{align}\label{T-tildem}
\|\tilde m(D,\theta)\tilde g\|_{L^\infty} + \|\tilde m(D,\theta)\tilde h\|_{L^\infty} +
\|\tilde m(D,\theta)\tilde \varphi\|_{L^\infty} \lesssim \sqrt{\ln A} .
\end{align} 
This can be proved as follows: 
\begin{align}\label{mgA}
\|\tilde m(D,\theta)\tilde g \|_{L^\infty}
&=
\Big\| \sum_{k\in\Z} A^{\frac{\theta}{2}} \langle k\rangle^{-\frac12-\theta}  |\hat g_k| \fe^{ikx} \Big\|_{L^\infty} \notag\\
&\lesssim 
A^{\frac{\theta}{2}}
\Big(\sum_{k\in\Z} \langle k\rangle^{-1-2\theta} \Big)^{\frac12}
\Big(\sum_{k\in\Z} |\hat g_k|^2\Big)^{\frac12} \notag\\
&\lesssim 
\frac{A^{\frac{\theta}{2}}}{\sqrt{\theta}} \|g\|_{L^2} .
\end{align} 
If $\frac{1}{\ln A} \le \theta_0$ then we can choose $\theta=1/\ln A$ so that $A^{\frac{\theta}{2}}\lesssim 1$ and $\frac{1}{\sqrt{\theta}} =\sqrt{\ln A}$. In this case, inequality \eqref{mgA} reduces to  
\begin{align*}
\| \tilde m(D,\theta)\tilde g\|_{L^\infty}
&\lesssim \sqrt{\ln A} . 
\end{align*}
If $\frac{1}{\ln A} \ge\theta_0$ then $2\le A\le e^{1/\theta_0}$ and therefore $A\sim 2$, which implies that $\sqrt{\ln A}\sim 1$. In this case, we can choose $\theta=\theta_0$ so that inequality \eqref{mgA} implies that   
$$
\| \tilde m(D,\theta)\tilde g \|_{L^\infty} \lesssim 1 \lesssim \sqrt{\ln A} . 
$$
This proves \eqref{T-tildem} and therefore completes the proof of Lemma \ref{lem:ln-loss}. 
\end{proof}

\subsection{Averaging approximation of exponential functions}
\label{section:average}

In this subsection, we establish some average estimates which play important roles in the analysis of the local truncation errors. We define the {\it average} and {\it oscillation} of a function $f$ in the interval $[0,\tau]$ by 
$$
M_\tau(f)=\frac1\tau \int_0^\tau f(t)\,\d t ,
$$
and 
$$
\|f\|_{{\rm osc}([0,\tau])}:= \max\limits_{\{t_1,t_2\}\subset [0,\tau]}\big|f(t_1)-f(t_2)\big| , 
$$
respectively. 
Then the following basic inequality holds and will be frequently used: 
\begin{align}\label{lem:average1}
\big| M_\tau(fg)-M_\tau(f)M_\tau(g) \big| 
\le \|f\|_{{\rm osc}([0,\tau])} \|g\|_{{\rm osc}([0,\tau])} . 
\end{align}
Indeed, 
\begin{align}\label{lem:average111}
\big| M_\tau(fg)-M_\tau(f)M_\tau(g) \big| 
&= \bigg| \frac1\tau \int_0^\tau fg\,\d s-\frac1\tau \int_0^\tau M_\tau(f) g\,\d s \bigg| \notag \\
&=\bigg| \frac1\tau \int_0^\tau \big(f- M_\tau(f)\big)g\,\d s \bigg| \notag\\
&=\bigg|  \frac1\tau \int_0^\tau \big(f- M_\tau(f)\big)\big(g-M_\tau(g)\big)\,\d s \bigg| \notag\\
&\le \|f\|_{{\rm osc}([0,\tau])} \|g\|_{{\rm osc}([0,\tau])} . 
\end{align}

In the following lemma we prove that, if $f$ and $g$ are exponential functions with imaginary powers, then much better estimates can be obtained. 
\begin{lemma}\label{lem:average2}
Let $\alpha,\beta\in \R$. If $\alpha, \beta\ne 0$ and $s\in[0,\tau]$, then 
\begin{align}\label{average2-1}
\big|M_\tau\big(\fe^{is(\alpha+\beta)}\big)-M_\tau\big(\fe^{is\alpha}\big)M_\tau\big(\fe^{is\beta}\big)\big|\lesssim  \min\left\{\left|\frac{\alpha}{\beta}\right|,\left|\frac{\beta}{\alpha}\right|,\tau|\alpha|,\tau|\beta|\right\}.  
\end{align} 
If $\alpha+\beta\ne 0$, then 
\begin{align}\label{average2-2}
\big|M_\tau\big(\fe^{is(\alpha+\beta)}\big)-M_\tau\big(\fe^{is\alpha}\big)M_\tau\big(\fe^{is\beta}\big)\big|\lesssim  \tau^{-1}|\alpha+\beta|^{-1}.  
\end{align} 
\end{lemma}
\begin{proof}
Since  
$\|\fe^{is\alpha}\|_{{\rm osc}([0,\tau])}\lesssim \min\{1,\tau|\alpha|\}$, 
 \eqref{lem:average1} implies that  
\begin{align}\label{AVE-1}
\big|M_\tau\big(\fe^{is(\alpha+\beta)}\big)-M_\tau\big(\fe^{is\alpha}\big)M_\tau\big(\fe^{is\beta}\big)\big|
\lesssim  \min\left\{\tau|\alpha|,\tau|\beta|\right\}.  
\end{align} 
Furthermore, similar as \eqref{lem:average111}, we have 
\begin{align*}
&\hspace{-10pt} M_\tau\big(\fe^{is(\alpha+\beta)}\big)-M_\tau\big(\fe^{is\alpha}\big)M_\tau\big(\fe^{is\beta}\big)\\
=&\frac1\tau\int_0^\tau \left(\fe^{is\alpha}-M_\tau\big(\fe^{is\alpha}\big)\right) \fe^{is\beta} \,\d s\\
=&\frac1\tau\int_0^\tau \left(\fe^{is\alpha}-M_\tau\big(\fe^{is\alpha}\big)\right)\big(\fe^{is\beta}-1\big)\,\d s\\
=&\frac1\tau\int_0^\tau \fe^{is\alpha}\big(\fe^{is\beta}-1\big)\,\d s
-\frac1\tau\int_0^\tau M_\tau\big(\fe^{is\alpha}\big)\big(\fe^{is\beta}-1\big)\,\d s\\
=&\frac1\tau\int_0^\tau \frac{1}{i\alpha} \partial_s\big(\fe^{is\alpha}\big)\big(\fe^{is\beta}-1\big)\,\d s
-\frac1\tau\int_0^\tau M_\tau\big(\fe^{is\alpha}\big)\big(\fe^{is\beta}-1\big)\,\d s.
\end{align*}
Then, using integration by parts, we obtain  
\begin{align}
&\hspace{-10pt} \big| M_\tau\big(\fe^{is(\alpha+\beta)}\big)-M_\tau\big(\fe^{is\alpha}\big)M_\tau\big(\fe^{is\beta}\big) \big| \notag\\
=&\bigg| \frac{1}{i\tau\alpha}\fe^{i\tau\alpha}\big(\fe^{i\tau\beta}-1\big)-\frac1\tau\int_0^\tau \frac{\beta}{\alpha}
\fe^{is\alpha+is\beta}\,\d s-M_\tau\big(\fe^{is\alpha}\big)\frac1\tau\int_0^\tau \big(\fe^{is\beta}-1\big)\,\d s \bigg| \label{eqs:M-fg-exp11} \\
\le&\frac{1}{\tau|\alpha|}\big|\fe^{i\tau\beta}-1\big|+\frac{|\beta|}{|\alpha|}+\big|M_\tau\big(\fe^{is\alpha}\big)\big| \frac1\tau\int_0^\tau \big|\fe^{is\beta}-1\big|\,\d s .
\label{eqs:M-fg-exp}
\end{align}
By substituting the following estimates into \eqref{eqs:M-fg-exp}: 
$$
\big|\fe^{is\beta}-1\big|\le \tau|\beta|
\quad\mbox{and}\quad \big|M_\tau\big(\fe^{is\alpha}\big)\big|=\Big|\frac1{i\tau\alpha} \big|\fe^{i\tau\alpha}-1\big| \Big|\lesssim \frac{1}{\tau|\alpha|},
$$
we obtain  
$$
\big|M_\tau\big(\fe^{is(\alpha+\beta)}\big)-M_\tau\big(\fe^{is\alpha}\big)M_\tau\big(\fe^{is\beta}\big)\big|
\lesssim \frac{|\beta|}{|\alpha|}.
$$
Based on the symmetry between $\alpha$ and $\beta$, the following result also holds:  
$$
\big|M_\tau\big(\fe^{is(\alpha+\beta)}\big)-M_\tau\big(\fe^{is\alpha}\big)M_\tau\big(\fe^{is\beta}\big)\big|
\lesssim \frac{|\alpha|}{|\beta|}.
$$
The two estimates above, together with \eqref{AVE-1}, imply the desired estimate in \eqref{average2-1}. 

In the case $|\alpha|\ge |\beta|$ we can obtain the following result directly from the expression in \eqref{eqs:M-fg-exp11}: 
\begin{align*}
&M_\tau\big(\fe^{is(\alpha+\beta)}\big)-M_\tau\big(\fe^{is\alpha}\big)M_\tau\big(\fe^{is\beta}\big)\\
&=\frac{1}{i\tau\alpha}\fe^{i\tau\alpha}\big(\fe^{i\tau\beta}-1\big)-\frac{1}{i\tau (\alpha+\beta)} \frac{\beta}{\alpha } 
\big(\fe^{i\tau(\alpha+\beta)}-1\big)-M_\tau\big(\fe^{is\alpha}\big)\frac1\tau\int_0^\tau \big(\fe^{is\beta}-1\big)\,\d s , 
\end{align*}
which implies that 
$$
\big|M_\tau\big(\fe^{is(\alpha+\beta)}\big)-M_\tau\big(\fe^{is\alpha}\big)M_\tau\big(\fe^{is\beta}\big)\big|
\lesssim \tau^{-1}\big(|\alpha|^{-1}+|\alpha+\beta|^{-1}\big) 
\lesssim \tau^{-1}|\alpha+\beta|^{-1} .
$$
Since the expression of $M_\tau\big(\fe^{is(\alpha+\beta)}\big)-M_\tau\big(\fe^{is\alpha}\big)M_\tau\big(\fe^{is\beta}\big)$ is symmetric about $\alpha$ and $\beta$, in the case $|\beta|\ge|\alpha|$ we can obtain the same result by switching the roles of $\alpha$ and $\beta$ in the argument above. This proves the desired estimate in \eqref{average2-2}.
\end{proof}

\subsection{Trilinear estimates associated to the KdV operator} \label{section:Gammak}

In this subsection, we establish new estimates for the phase function 
$$\phi:=k^3-k_1^3-k_2^3-k_3^3$$  
and use the results to prove two technical estimates for the following trilinear KdV operator: 
\begin{align}\label{C-KdV}
\mathcal C(v_1,v_2,v_3)=\int_{s_0}^s \fe^{t\partial_x^3}\P\left(\P\Big(\fe^{-t\partial_x^3}v_1(t)\cdot \fe^{-t\partial_x^3}v_2(t)\Big)\cdot \fe^{-t\partial_x^3}\partial_x^{-1}v_3(t)\right)\,\d t ,
\end{align}
where $s\ge s_0\ge 0$ are any two numbers such that $|s-s_0|\lesssim 1$. 
The trilinear estimates for the KdV operator established in this subsection will play a key role in the stability estimates for nonsmooth solutions in $C([0,T];H^\gamma)$ with $\gamma\in(0,1]$ possibly approaching zero. 

For the simplicity of notation, we decompose the set $\{ (k_1,k_2,k_3)\in \Z^3: k_1+k_2+k_3=k\}$ into the following two subsets: 
\begin{align*}
\Gamma_0(k):=&\{(k_1,k_2,k_3)\in \Z^3: k_1+k_2+k_3=k, k_1+k_2=0\,\,\mbox{or}\,\,k_1+k_3=0\,\,\mbox{or}\,\,k_2+k_3=0\} ,\\ 
\Gamma(k):=&\{(k_1,k_2,k_3)\in \Z^3: k_1+k_2+k_3=k, k_1+k_2\ne0,k_1+k_3\ne0,k_2+k_3\ne0\} ,
\end{align*}
and denote 
$$
|k_m|=\max\{|k|,|k_1|,|k_2|,|k_3|\}.
$$
We further decompose $\Gamma(k)$ into two subsets, i.e., $\Gamma(k)=\Gamma_1(k)\cup \Gamma_2(k)$, where  
\begin{align*}
&\Gamma_1(k):=\Big\{(k_1,k_2,k_3)\in \Gamma: |\phi| < \frac14 |k_m|^2 \Big\} ,\\
&\Gamma_2(k):=\Big\{(k_1,k_2,k_3)\in \Gamma:  |\phi|\ge \frac14 |k_m|^2 \Big\}.
\end{align*}

In the following lemma, we show that a good estimate exists for the phase function when $(k_1,k_2,k_3)\in\Gamma(k)$. Moreover, better estimates can be obtained for $(k_1,k_2,k_3)\in\Gamma_1(k)$ and $(k_1,k_2,k_3)\in\Gamma_2(k)$, respectively. These new estimates of the phase function can be used to analyze the trilinear KdV operator defined in \eqref{C-KdV}.

\begin{lemma}\label{lem:a4-est} 
Let $k\in\Z$. Then the following results hold. 
\begin{enumerate}
\item[(1) ]If $(k_1,k_2,k_3)\in \Gamma(k)$ then 
$$
|\phi|\gtrsim  |k_m|.
$$
\item[(2) ] If $(k_1,k_2,k_3)\in \Gamma_1(k)$ then 
$$
|k|\sim |k_1|\sim|k_2|\sim |k_3|.
$$
\item[(3) ] 
$\Gamma_2(k)$ can be further decomposed into $\Gamma_2(k)=\Gamma_{21}(k)\cup\Gamma_{22}(k)$, with 
\begin{align*}
\Gamma_{21}(k)&:=\{(k_1,k_2,k_3)\in \Gamma: \frac14|k_m|^2 \le  |\phi|\ll |k_m|^\frac{15}{7}\} , \\ 
\Gamma_{22}(k)&:=\{(k_1,k_2,k_3)\in \Gamma:  |\phi|\gtrsim  |k_m|^\frac{15}{7}\}.
\end{align*}
Moreover, for $(k_1,k_2,k_3)\in\Gamma_{21}(k)$ there exists $j,h\in \{1,2,3\}$ such that  
$|k_j+k_h|\ll |k_m|^{\frac57}$.
\end{enumerate}
\end{lemma}
\begin{proof}
For $(k_1,k_2,k_3)\in\Gamma(k)$, we denote $k_0=-k$ so that $k_0+k_1+k_2+k_3=0$ and 
\begin{align}\label{no-equal-cond}
k_j+k_h\ne0, \quad \mbox{for any}\,\, j,h\in \{0,1,2,3\}\,\,\mbox{such that}\,\, j\ne h.
\end{align}
By the symmetry among the indices $k_0,k_1,k_2,k_3$,  we may further assume the following relation: 
\begin{align}\label{k0k1k2k3}
|k_0|\ge |k_1|\ge |k_2|\ge |k_3| . 
\end{align}
In this case, the following results must hold:
\begin{align}\label{km=k0-k1}
|k_m|=|k_0|\sim |k_1|\,\,\,(\mbox{in particular}, |k_1|\le |k_0|\le 3|k_1|) \quad\mbox{and}\quad k_0\cdot k_1<0 .
\end{align}
In fact, if $k_0\cdot k_1\ge 0$ then the relation $k_0+k_1+k_2+k_3=0$ implies that 
$k_0=k_1=-k_2=-k_3$, which contradicts \eqref{no-equal-cond}. 
If $|k_0|>3|k_1|$ then $|k_0+k_1|>2|k_1|\ge |k_2+k_3|$, which contradicts the relation $k_0+k_1+k_2+k_3=0$. 
Therefore, $|k_0|\sim |k_1|$. This proves the results in \eqref{km=k0-k1}. 

Since $k_0\cdot k_1<0$ as shown in \eqref{km=k0-k1}, without loss of generality, we may assume that $k_0>0$ and $k_1<0$. 
\begin{enumerate}

\item[(1)] If $k_2\le 0$ then $|\phi|=|3(k_1+k_0)(k_1+k_2)(k_1+k_3)| \ge 3|k_1+k_2|\ge 3|k_1|\sim |k_m|$.

\noindent If $k_2>0$ then by the symmetry in the expression of $\phi=-k_0^3-k_1^3-k_2^3-k_3^3$, we have 
$|\phi|=|3(k_0+k_1)(k_0+k_2)(k_0+k_3)| \ge 3|k_0+k_2|\ge 3|k_0|\sim |k_m| . $

\item[(2)] 
In view of \eqref{k0k1k2k3}, we only need to prove $|k_0|\sim |k_3|$ for $(k_1,k_2,k_3)\in\Gamma_1(k)$. 
In fact, if $|k_0|\ge 6|k_3|$ then \eqref{km=k0-k1} implies that $|k_1|\ge \frac13|k_0|\ge 2|k_3|$, and therefore 
$$
|k_0+k_3|\ge \frac12|k_0| = \frac12|k_m|, \quad \mbox{and }\quad 
|k_0+k_2|=|k_1+k_3|\ge \frac12 |k_1|\ge \frac16|k_m| . 
$$
This implies that 
$$
|\phi| = |3(k_0+k_1)(k_0+k_2)(k_0+k_3)| \ge  3\times \frac12|k_m|\times \frac16|k_m| = \frac14 |k_m|^2
$$
which contradicts that $(k_1,k_2,k_3)\in\Gamma_1(k)$. This proves $|k_0|\sim |k_3|$ for $(k_1,k_2,k_3)\in\Gamma_1(k)$. 

\item[(3)] 
If $|\phi|\gtrsim  |k_m|^\frac{15}{7}$ is not true, then $|\phi|\ll  |k_0|^\frac{15}{7}$.
In this case, the following result must hold:
\begin{align}\label{k0+k1<<57}
|k_0+k_1| =|k_2+k_3|\ll |k_0|^\frac{5}{7} . 
\end{align} 
Otherwise $|k_0+k_1|\gtrsim |k_0|^\frac{5}{7}$, which together with \eqref{k0k1k2k3} implies that 
$$
|k_0+k_j| \ge |k_0+k_1| \gtrsim |k_0|^\frac{5}{7} \quad \mbox{for all } j\in \{1,2,3\}.
$$
This means that $|\phi|\gtrsim |k_0|^\frac{15}{7}$, which contradicts $|\phi|\ll |k_0|^\frac{15}{7}$. 
This proves \eqref{k0+k1<<57} and completes the proof of Lemma \ref{lem:a4-est} (3). 
\end{enumerate}

\end{proof}

The main result of this subsection is the following proposition, which contains new estimates of the trilinear KdV operator defined in \eqref{C-KdV} for low-regularity functions in $L^\infty(s_0,s;H^{\alpha})\times L^\infty(s_0,s;H^{\alpha})\times L^\infty(s_0,s;H^{\alpha})$, with $\alpha\in[0,1]$ possibly approaching zero. The results are proved by utilizing the new estimates for the phase function in Lemma \ref{lem:a4-est}. 

\begin{proposition}\label{lem:tri-linear}
Let $\alpha\in [0,1]$ and $s\ge s_0\ge 0$ with $|s-s_0|\lesssim 1$. 
Suppose that $v_j\in L^\infty(s_0,s;H^{\alpha})$ and $\partial_tv_j\in L^\infty(s_0,s;H^{-\frac{23}{14}})$ for $j=1,2,3$. 
Then the trilinear operator defined in \eqref{C-KdV} has the following properties{\rm:} 
\begin{itemize}
\item[(1) ] $\,$\vspace{-10pt}
\begin{align*}
&\big\|\mathcal C(v_1,v_2,v_3)\big\|_{L^2} \\ 
&\lesssim  |s-s_0|\prod\limits_{j=1}^3\|v_j\|_{X^0([s_0,s])}
+\|v_1\|_{L^\infty(s_0,s;H^{-\frac{23}{14}})}\|v_2\|_{L^\infty(s_0,s;L^2)}\|v_3\|_{L^\infty(s_0,s;L^2)} ,
\end{align*}
where 
\begin{align*}
\|v\|_{X^\alpha([s_0,s])}:=\|v\|_{L^\infty(s_0,s;H^{\alpha})}+ \big\|\partial_t v\big\|_{L^\infty(s_0,s;H^{-\frac{23}{14}})}.
\end{align*}
Moreover, the inequality still holds when $v_1,v_2,v_3$ are permuted on the right-hand side. 
\item[(2) ] If $v_j,j=1,2,3$ are time-independent, then 
\begin{align*}
\big\|\mathcal C(v_1,v_2,v_3)\big\|_{L^2}\lesssim |s-s_0|^\alpha \prod\limits_{j=1}^3\|v_j\|_{H^\alpha}.
\end{align*}
\end{itemize}
\end{proposition}
\begin{proof} 
Clearly, the trilinear operator defined in \eqref{C-KdV} does not have zero mode, i.e., 
$$ \mathcal{F}_0[\mathcal C(v_1,v_2,v_3)]=0 . $$
For $k\ne 0$, the Fourier transform of \eqref{C-KdV} can be written as  
\begin{align*}
\mathcal{F}_k\big[\mathcal C(v_1,v_2,v_3)\big]
=
\int_{s_0}^s \sum\limits_{\substack{k_1+k_2+k_3=k\\ k_1+k_2\ne0,k_3\ne0}}\fe^{-it\phi} (ik_3)^{-1} 
\hat v_{k_1}(t)\hat v_{k_2}(t)\hat v_{k_3}(t)\,\d t
\end{align*}
For the simplicity of notation, 
we assume that $\hat v_{j,k_j}\ge 0$ for $k_j\in\Z$ and $j=1,2,3$ (otherwise one can replace $\hat v_{j,k_j}$ by $|\hat v_{j,k_j}|$ in the following argument and consider the functions $\tilde v_j := \mathcal{F}_{k}^{-1} [\,|\hat v_{j,k}|\,]$ as in the proof of Lemma \ref{lem:ln-loss}). 

Since $\Gamma(k)=\Gamma_1(k)\cup \Gamma_2(k)$, we can further decompose $\mathcal{F}_k\big[\mathcal C(v_1,v_2,v_3)\big]$ into the following several parts:  
\begin{align*}
\mathcal{F}_k\big[\mathcal C(v_1,v_2,v_3)\big]
=
&
\int_{s_0}^s \sum\limits_{\substack{k_1+k_2+k_3=k,k_3\ne0\\ k_1+k_2\ne0, k_1+k_3=0}}\fe^{-it\phi} (ik_3)^{-1} 
\hat v_{1,k_1}(t)\hat v_{2,k_2}(t)\hat v_{3,k_3}(t)\,\d t\\
&+
\int_{s_0}^s \sum\limits_{\substack{k_1+k_2+k_3=k,k_3\ne0\\ k_1+k_2\ne0, k_2+k_3=0}}\fe^{-it\phi} (ik_3)^{-1} 
\hat v_{1,k_1}(t)\hat v_{2,k_2}(t)\hat v_{3,k_3}(t)\,\d t\\
&-
\int_{s_0}^s \sum\limits_{\substack{k_1+k_2+k_3=k,k_3\ne0\\ k_1+k_2\ne0\\
k_1+k_3=k_2+k_3=0}}\fe^{-it\phi} (ik_3)^{-1} 
\hat v_{1,k_1}(t)\hat v_{2,k_2}(t)\hat v_{3,k_3}(t)\,\d t\\
&+
\int_{s_0}^s \sum\limits_{\substack{(k_1,k_2,k_3)\in\Gamma_1(k)\\ k_3\ne0}} \fe^{-it\phi} (ik_3)^{-1} 
\hat v_{1,k_1}(t)\hat v_{2,k_2}(t)\hat v_{3,k_3}(t)\,\d t\\
&+
\int_{s_0}^s \sum\limits_{\substack{(k_1,k_2,k_3)\in\Gamma_2(k)\\ k_3\ne0}} \fe^{-it\phi} (ik_3)^{-1} 
\hat v_{1,k_1}(t)\hat v_{2,k_2}(t)\hat v_{3,k_3}(t)\,\d t\\
=
&\!: 
\sum\limits_{j=1}^5 \mathcal{F}_k\big[\mathcal C^*_j(v_1,v_2,v_3)\big].
\end{align*}
We present estimates for $\mathcal C^*_j(v_1,v_2,v_3)$, $j=1,2,3,4,5$, respectively. 

\emph{\it (i) Estimates for $\mathcal C^*_1(v_1,v_2,v_3)$, $\mathcal C^*_2(v_1,v_2,v_3)$ and $\mathcal C^*_3(v_1,v_2,v_3)$}: 
Since $k_3=-k_1\ne0$ and $k_2=k$ in the expression of $\mathcal C^*_1(v_1,v_2,v_3)$, it follows that (by the Cauchy--Schwarz  inequality) 
\begin{align*}
\big|\mathcal{F}_k\big[\mathcal C^*_1(v_1,v_2,v_3)\big]\big|
\lesssim &
\int_{s_0}^s  \sum\limits_{k_1\ne0}|k_1|^{-1}
\hat v_{1,k_1}(t)\> \hat v_{2,k}(t)\>\hat v_{3,-k_1}(t)\,\d t\\
\lesssim &
\int_{s_0}^s   \hat v_{2,k}(t) \|v_1(t)\|_{L^2}\|v_3(t)\|_{L^2}\,\d t.
\end{align*}
This implies that
\begin{align*}
\big\|\mathcal C^*_1(v_1,v_2,v_3)\big\|_{L^2} 
\lesssim 
|s-s_0|  \|v_1\|_{L^\infty(s_0,s;L^2)} \|v_2\|_{L^\infty(s_0,s;L^2)}\|v_3\|_{L^\infty(s_0,s;L^2)} .
\end{align*}
Since $\mathcal C^*_2(v_1,v_2,v_3)$ is similar as $\mathcal C^*_1(v_1,v_2,v_3)$, and the expression of $\mathcal C^*_3(v_1,v_2,v_3)$ consists of terms which are contained in $\mathcal C^*_1(v_1,v_2,v_3)$, 
the same estimates hold for $\mathcal C^*_2(v_1,v_2,v_3)$ and $\mathcal C^*_3(v_1,v_2,v_3)$, i.e., 
\begin{align}\label{Estimate-C123}
&\big\|\mathcal C^*_1(v_1,v_2,v_3)\big\|_{L^2} 
+\big\|\mathcal C^*_2(v_1,v_2,v_3)\big\|_{L^2} + \big\|\mathcal C^*_3(v_1,v_2,v_3)\big\|_{L^2} \notag\\ 
&\lesssim 
|s-s_0|  \|v_1\|_{L^\infty(s_0,s;L^2)} \|v_2\|_{L^\infty(s_0,s;L^2)}\|v_3\|_{L^\infty(s_0,s;L^2)} .
\end{align}

\emph{\it (ii) Estimates for $\mathcal C^*_4(v_1,v_2,v_3)$}:  
Since $|k_1|\sim|k_2|\sim|k_3|\sim |k|$ for $(k_1,k_2,k_3)\in \Gamma_1(k)$, it follows that 
\begin{align*}
\big|\mathcal{F}_k\big[\mathcal C^*_4(v_1,v_2,v_3)\big]\big|
\lesssim &
\int_{s_0}^s \sum\limits_{\substack{\Gamma_1(k)\\ k_3\ne0}} |k_3|^{-1} 
\hat v_{1,k_1}(t)\hat v_{2,k_2}(t)\hat v_{3,k_3}(t)\,\d t \\
\lesssim &
\int_{s_0}^s |k|^{-1}  \sum\limits_{\substack{k_1+k_2+k_3=k\\|k_1|\sim|k_2|\sim|k_3|\sim |k| }}
\hat v_{1,k_1}(t)\hat v_{2,k_2}(t)\hat v_{3,k_3}(t)\,\d t\\ 
\lesssim &
\int_{s_0}^s  |k|^{-1}  
\Big(\sum\limits_{\substack{|k_1|\sim |k| \\ |k_2| \sim |k|}}1 \Big)^\frac12
\Big(\sum\limits_{\substack{|k_1|\sim |k| \\ |k_2| \sim |k|}}
\big||\hat v_{1,k_1}(t)|^2|\hat v_{2,k_2}(t)|^2 \>|\hat{v}_{3,k-k_1- k_2}(t)|^2 \Big)^\frac12\,\d t \\
\lesssim &
\int_{s_0}^s \Big(\sum\limits_{k_1,k_2}
\big||\hat v_{1,k_1}(t)|^2|\hat v_{2,k_2}(t)|^2 \>|\hat{v}_{3,k-k_1- k_2}(t)|^2 \Big)^\frac12\,\d t .
\end{align*}
Therefore,  
\begin{align}\label{Estimate-C4}
\big\|\mathcal C^*_4(v_1,v_2,v_3)\big\|_{L^2} 
&\lesssim 
\Big(\sum_{k\in\Z} \big|\mathcal{F}_k\big[\mathcal C^*_4(v_1,v_2,v_3)\big]\big|^2 \Big)^{\frac12} \notag\\
&\lesssim
\int_{s_0}^s \Big(\sum\limits_{k} \sum\limits_{k_1,k_2}
|\hat v_{1,k_1}(t)|^2|\hat v_{2,k_2}(t)|^2 \>|\hat v_{3,k-k_1- k_2}(t)|^2 \Big)^\frac12\,\d t \notag\\ 
&\lesssim
|s-s_0|  \|v_1\|_{L^\infty(s_0,s;L^2)} \|v_2\|_{L^\infty(s_0,s;L^2)}\|v_3\|_{L^\infty(s_0,s;L^2)} .
\end{align}

We see that $\mathcal C^*_j(v_1,v_2,v_3)$, $j=1,2,3,4$ satisfy the estimates in both (1) and (2). It remains to show that $\mathcal C^*_5(v_1,v_2,v_3)$ also satisfies the estimates in (1) and (2).\bigskip

{\it Proof of {\rm(1)}.}
Via integration by parts, we can write $\mathcal{F}_k\big[\mathcal C^*_5(v_1,v_2,v_3)\big]$ as 
\begin{align}
\mathcal{F}_k\big[\mathcal C^*_5(v_1,v_2,v_3)\big]
=&\sum\limits_{\substack{\Gamma_{2}(k)\\ k_3\ne0}} \fe^{-it\phi} \frac1{k_3 \phi}
\hat v_{1,k_1}(t)\hat v_{2,k_2}(t)\hat v_{3,k_3}(t)\Big|_{s_0}^{s}\notag\\
& - \sum\limits_{\substack{\Gamma_{2}(k)\\ k_3\ne0}} \int_{s_0}^s\fe^{-it\phi} \frac1{k_3 \phi}
\partial_t\big(\hat v_{1,k_1}(t)\hat v_{2,k_2}(t)\hat v_{3,k_3}(t)\big)\,\d s . \label{C4-12}
\end{align}
We further decompose $\Gamma_2(k)$ into two parts, i.e., $\Gamma_2(k)=\Gamma_{21}(k)\cup\Gamma_{22}(k) $, where 
\begin{align*}
\Gamma_{21}(k)&:=\{(k_1,k_2,k_3)\in \Gamma: |k_m|^2 \lesssim  |\phi|\ll |k_m|^\frac{15}{7}\} , \\ 
\Gamma_{22}(k)&:=\{(k_1,k_2,k_3)\in \Gamma:  |\phi|\gtrsim  |k_m|^\frac{15}{7}\}.
\end{align*}
Then \eqref{C4-12} can be written as   
\begin{align*}
\mathcal{F}_k\big[\mathcal C^*_5(v_1,v_2,v_3)\big]
= &\sum\limits_{\substack{\Gamma_{21}(k)\\ k_3\ne0}} \fe^{-it\phi} \frac1{k_3 \phi}
\hat v_{1,k_1}(t)\hat v_{2,k_2}(t)\hat v_{3,k_3}(t)\Big|_{s_0}^{s} \notag\\
& - \sum\limits_{\substack{\Gamma_{22}(k)\\ k_3\ne0}}\fe^{-it\phi} \frac1{k_3 \phi}
\hat v_{1,k_1}(t)\hat v_{2,k_2}(t)\hat v_{3,k_3}(t)\Big|_{s_0}^{s}\notag\\
& - \sum\limits_{\substack{\Gamma_{21}(k)\\ k_3\ne0}}\int_{s_0}^s\fe^{it\phi} \frac1{k_3 \phi}
\partial_t\big(\hat v_{1,k_1}(t)\hat v_{2,k_2}(t)\hat v_{3,k_3}(t)\big)\,\d s \notag\\
& - \sum\limits_{\substack{\Gamma_{22}(k)\\ k_3\ne0}} \int_{s_0}^s\fe^{it\phi} \frac1{k_3 \phi}
\partial_t\big(\hat v_{1,k_1}(t)\hat v_{2,k_2}(t)\hat v_{3,k_3}(t)\big)\,\d s\\
= &\!: 
\sum\limits_{j=1}^4 \mathcal{F}_k\big[\mathcal C^*_{5j}(v_1,v_2,v_3)\big].
\end{align*}

\emph{Estimates for $\mathcal C^*_{51}(v_1,v_2,v_3)$}: 
In the expression of $\mathcal C^*_{51}(v_1,v_2,v_3)$, we have $ (k_1,k_2,k_3)\in \Gamma_{21}(k)$. 
According to Lemma \ref{lem:a4-est}, for $ (k_1,k_2,k_3)\in \Gamma_{21}(k)$ there exist $j,h\in \{1,2,3\}$, such that 
$$
|k_j+k_h|\ll |k_m|^{\frac57} . 
$$
Without loss of generality, we may assume that $|k_2+k_3|\ll |k_m|^{\frac57}$, as the other cases can be treated similarly. 
Since $|\phi|\gtrsim |k_m|^{2}$ on $\Gamma_{21}(k)$, by using a change of variables $\tilde k_2=k_2+k_3$ and the Cauchy--Schwarz inequality, we obtain 
\begin{align*}
&\hspace{-10pt} \big| \mathcal{F}_k\big[\mathcal C^*_{51}(v_1,v_2,v_3)\big]\big| \\
\lesssim  &\max\limits_{t\in \{s_0,s\}} \sum\limits_{\substack{k_1+k_2+k_3=k\\|k_2+k_3|\ll   |k_m|^{\frac57},|k_3|\ne0 }}|k_m|^{-2}|k_3|^{-1}
\hat v_{1,k_1}(t)\hat v_{2,k_2}(t)\hat v_{3,k_3}(t)\\
\lesssim &
   \sum\limits_{|k_3|\ne0} \sum\limits_{|\tilde k_2|\ll   |k_m|^{\frac57}}  |k_m|^{-2}|k_3|^{-1}
\mathcal{F}_{k-\tilde k_2}[v_1(t)]\>\mathcal{F}_{\tilde k_2-k_3}[v_2(t)] \mathcal{F}_{k_3}[v_3(t)] \\
\lesssim &
 \Big(\sum\limits_{|\tilde k_2|\ll   |k_m|^{\frac57} , k_3\ne0} |k_m|^{-\frac{5}{7}}|k_3|^{-2} \Big)^\frac12
 \Big(\sum\limits_{\tilde k_2, k_3} |k_m|^{-\frac{23}{7}} |\mathcal{F}_{k-\tilde k_2}[v_1(t)]|^2 |\mathcal{F}_{\tilde k_2-k_3}[v_2(t)] |^2 |\mathcal{F}_{k_3}[v_3(t)]|^2 \Big)^\frac12 \\
\lesssim &
 \Big(\sum\limits_{k_1+ k_2+ k_3=k} |k_m|^{-\frac{23}{7}}|\mathcal{F}_{k_1}[v_1(t)]|^2 |\mathcal{F}_{k_2}[v_2(t)] |^2 |\mathcal{F}_{k_3}[v_3(t)]|^2\Big)^\frac12 ,
\end{align*}
where we have changed the subscripts back to $k_1$, $k_2$ and $k_3$ in the last inequality. 
For $k\ne 0$ and $k_1+k_2+k_3=k$, it is easy to verify that $|k_m|=\max(|k_1|,|k_2|,|k_3|)\ge \langle k_j\rangle$ for $j=1,2,3$. 
By taking square of the inequality above and summing up the results for $k\in\Z$ such that $k\ne0$, using the property that $\mathcal{F}_0[ \mathcal C^*_{51}(v_1,v_2,v_3)]=0$, we obtain
\begin{align}\label{est:C41-1}
\big\| \mathcal C^*_{51}(v_1,v_2,v_3) \big\|_{H^\sigma}^2 
\lesssim &
\sum_{k\ne0}  |k|^{2\sigma} \sum\limits_{k_1+ k_2+ k_3=k} |k_m|^{-\frac{23}{7}}|\mathcal{F}_{k_1}[v_1(t)]|^2 |\mathcal{F}_{k_2}[v_2(t)] |^2 |\mathcal{F}_{k_3}[v_3(t)]|^2 \notag\\
\lesssim &
\sum\limits_{k_1,k_2,k_3}
\big|\langle k_1\rangle^{\sigma_1}\mathcal{F}_{k_1}[v_1(t)]\big|^2 \big|\langle k_2\rangle^{\sigma_2}\mathcal{F}_{k_2}[v_2(t)] \big|^2 \big|\langle k_3\rangle^{\sigma_3}\mathcal{F}_{k_3}[v_3(t)]\big|^2 \notag\\
\lesssim &
\|v_1\|_{L^\infty(s_0,s;H^{\sigma_1})}^2 \|v_2\|_{L^\infty(s_0,s;H^{\sigma_2})}^2 \|v_3\|_{L^\infty(s_0,s;H^{\sigma_3})}^2 ,
\end{align}
where $0\le \sigma\le \frac{23}{14}$ and $-\frac{23}{14} \le \sigma_j \le 0 $, $j=1,2,3$, are any numbers satisfying $\sigma-\sigma_1-\sigma_2-\sigma_3=\frac{23}{14}$.

\emph{Estimates for $\mathcal C^*_{52}(v_1,v_2,v_3)$}: 
Since $|\phi|\gtrsim |k_m|^{\frac{15}{7}}$ for $(k_1,k_2,k_3)\in\Gamma_{22}(k)$, it follows that, by the Cauchy--Schwarz inequality, 
\begin{align*}
&\hspace{-10pt} \big| \mathcal{F}_k\big(\mathcal C^*_{52}(v_1,v_2,v_3)\big)\big| \\ 
\lesssim  &\max\limits_{t\in \{s_0,s\}} \sum\limits_{\substack{k_1+k_2+k_3=k\\ |k_3|\ne0}}|k_m|^{-\frac{15}{7}}|k_3|^{-1}
\mathcal{F}_{k_1}[v_1(t)]\mathcal{F}_{k_2}[v_2(t)]\mathcal{F}_{k_3}[v_3(t)]\\
\lesssim &
\Big(\sum\limits_{|k_2|\le|k_m|, k_3\ne0} |k_m|^{-1}|k_3|^{-2} \Big)^\frac12 
 \Big(\sum\limits_{|k_2|\le|k_m|,k_3\ne0} |k_m|^{-\frac{23}{7}}
  \big|\mathcal{F}_{k-k_2-k_3}[v_1(t)]\big|^2 \big|\mathcal{F}_{k_2}[v_2(t)]\big|^2
\big|\mathcal{F}_{k_3}[v_1(t)]\big|^2\Big)^\frac12 \\
\lesssim &
 \Big(\sum\limits_{k_1+k_2+k_3=k} |k_m|^{-\frac{23}{7}}
  \big|\mathcal{F}_{k_1}[v_1(t)]\big|^2 \big|\mathcal{F}_{k_2}[v_2(t)]\big|^2
\big|\mathcal{F}_{k_3}[v_1(t)]\big|^2\Big)^\frac12 , 
\end{align*}
where we have changed the subscripts back to $k_1$, $k_2$ and $k_3$ in the last inequality. By taking square of the inequality above and summing up the results for $k\in\Z$ such that $k\ne0$, using the property that $\mathcal{F}_0[ \mathcal C^*_{52}(v_1,v_2,v_3)]=0$, we obtain 
\begin{align}\label{est:C42-1}
\big\|\mathcal C^*_{52}(v_1,v_2,v_3)\big\|_{H^\sigma}
\lesssim  &
\|v_1\|_{L^\infty(s_0,s;H^{\sigma_1})}  \|v_2\|_{L^\infty(s_0,s;H^{\sigma_2})}  \|v_3\|_{L^\infty(s_0,s;H^{\sigma_3})} ,
\end{align}
where $0\le \sigma\le \frac{23}{14}$ and $-\frac{23}{14} \le \sigma_j \le 0 $, $j=1,2,3$, are any numbers satisfying $\sigma-\sigma_1-\sigma_2-\sigma_3=\frac{23}{14}$. 

In particular, by choosing $\sigma=\sigma_2=\sigma_3=0$ and $\sigma_1=-\frac{23}{14}$ in \eqref{est:C41-1} and \eqref{est:C42-1}, we obtain the following result:  
\begin{align}\label{est:C412-1}
&\big\|\mathcal C^*_{51}(v_1,v_2,v_3)\big\|_{L^2} + \big\|\mathcal C^*_{52}(v_1,v_2,v_3)\big\|_{L^2} \notag\\ 
&\lesssim \|v_1\|_{L^\infty(s_0,s;H^{-\frac{23}{14}})}  \|v_2\|_{L^\infty(s_0,s;L^2)}  \|v_3\|_{L^\infty(s_0,s;L^2)} .
\end{align}

Since the constraint $\sigma-\sigma_1-\sigma_2-\sigma_3=\frac{23}{14}$ is symmetric about $\sigma_1$, $\sigma_2$ and $\sigma_3$, it follows that the supscripts $\sigma_1$, $\sigma_2$ and $\sigma_3$ can be permuted in the right-hand sides of  \eqref{est:C41-1} and \eqref{est:C42-1}. Therefore, the estimate in \eqref{est:C412-1} still holds when $v_1$, $v_2$ and $v_3$ are permuted. 

\emph{Estimates for $\mathcal C^*_{53}(v_1,v_2,v_3)$ and $\mathcal C^*_{54}(v_1,v_2,v_3)$}:  
Similar as $\mathcal C^*_{51}(v_1,v_2,v_3)$ and $\mathcal C^*_{52}(v_1,v_2,v_3)$, the following result holds:  
\begin{align*} 
&\|\mathcal C^*_{53}(v_1,v_2,v_3)\|_{H^{\sigma}} + \|\mathcal C^*_{54}(v_1,v_2,v_3)\|_{H^{\sigma}} \\ 
&\lesssim  
|s-s_0|\Big( \|\partial_t v_1\|_{L^\infty(s_0,s;H^{\sigma_1})}  \|v_2\|_{L^\infty(s_0,s;H^{\sigma_2})}  \|v_3\|_{L^\infty(s_0,s;H^{\sigma_3})} \\
&\qquad\qquad\quad\,\, + \| v_1\|_{L^\infty(s_0,s;H^{\sigma_2})}  \|\partial_tv_2\|_{L^\infty(s_0,s;H^{\sigma_1})}  \|v_3\|_{L^\infty(s_0,s;H^{\sigma_3})}  \\
&\qquad\qquad\quad\,\, +  \| v_1\|_{L^\infty(s_0,s;H^{\sigma_3})}  \|v_2\|_{L^\infty(s_0,s;H^{\sigma_2})}  \|\partial_tv_3\|_{L^\infty(s_0,s;H^{\sigma_1})}  \Big) ,
\end{align*}
where $0\le \sigma\le \frac{23}{14}$ and $-\frac{23}{14} \le \sigma_j \le 0 $, $j=1,2,3$, are any numbers satisfying $\sigma-\sigma_1-\sigma_2-\sigma_3=\frac{23}{14}$.
In particular, by choosing $\sigma_1=-\frac{23}{14}$ and $\sigma=\sigma_2=\sigma_3=0$, we obtain 
\begin{align}\label{est:C434-1}
&\|\mathcal C^*_{53}(v_1,v_2,v_3)\|_{L^2} + \|\mathcal C^*_{54}(v_1,v_2,v_3)\|_{L^2} \notag\\ 
&\lesssim |s-s_0|
\Big( \|\partial_t v_1\|_{L^\infty(s_0,s;H^{-\frac{23}{14}})}  \|v_2\|_{L^\infty(s_0,s;L^2)}  \|v_3\|_{L^\infty(s_0,s;L^2)} \notag \\
&\qquad\qquad\quad\,\,  + 
\|v_1\|_{L^\infty(s_0,s;L^2)}  \|\partial_t v_2\|_{L^\infty(s_0,s;H^{-\frac{23}{14}})}  \|v_3\|_{L^\infty(s_0,s;L^2)}  \notag \\
&\qquad\qquad\quad\,\,  + 
\|v_1\|_{L^\infty(s_0,s;L^2)}  \|v_2\|_{L^\infty(s_0,s;L^2)}  \|\partial_t v_3\|_{L^\infty(s_0,s;H^{-\frac{23}{14}})}
\Big). 
\end{align}


Overall, the estimates of $\mathcal C^*_j(v_1,v_2,v_3)$, $j=1,2,3,4$, in \eqref{Estimate-C123} and \eqref{Estimate-C4}, and the estimates of $\mathcal C^*_{5j}(v_1,v_2,v_3)$, $j=1,2,3,4$, in \eqref{est:C412-1} and \eqref{est:C434-1}, imply the first result of Proposition \ref{lem:tri-linear}.\medskip

{\it Proof of {\rm(2)}.}
We further decompose $\mathcal C^*_{5}(v_1,v_2,v_3)$ into the low-frequency and high-frequency parts as follows: 
\begin{align}\label{C5-decompose}
\mathcal C^*_{5}(v_1,v_2,v_3)=\P_{\le |s-s_0|^{-\alpha}}\mathcal C^*_{5}(v_1,v_2,v_3)+\P_{> |s-s_0|^{-\alpha}}\mathcal C^*_{5}(v_1,v_2,v_3) , 
\end{align} 
where we have used the projections $\P_{\le N}$ and $\P_{>N}$ defined in Section \ref{section:projection} with $N=|s-s_0|^{-\alpha}$. 

The first part in \eqref{C5-decompose} can be estimated by using Bernstein's inequality in \eqref{Bernstein}, which converts the $L^2$ norm to the $H^{-b(\alpha)}$ norm, i.e., 
\begin{align*} 
\big\|\P_{\le |s-s_0|^{-\alpha}}\mathcal C^*_{5}(v_1,v_2,v_3)\big\|_{L^2} 
\lesssim 
|s-s_0|^{-b(\alpha)\alpha}\big\|\mathcal C^*_{5}(v_1,v_2,v_3)\big\|_{H^{-b(\alpha)}} ,
\end{align*}
where $b(\alpha)$ is chosen in the following way:  
$$
b(\alpha)
= \left\{\begin{aligned}
&\frac12+ && \mbox{when } \alpha=0 \\
&\frac12-\alpha && \mbox{when } \alpha\in \Big(0,\frac12 \Big] \\ 
&0 && \mbox{when } \alpha\in \Big(\frac12,1 \Big] .
\end{aligned}\right. 
$$
For time-independent functions $v_j$, $j=1,2,3$, it is straightforward to verify (by the Cauchy--Schwartz inequality and Sobolev embedding inequalities) that this choice of $b(\alpha)$ guarantees the following inequality: 
\begin{align*}
\big\|\mathcal C^*_{5}(v_1,v_2,v_3)\big\|_{H^{-b(\alpha)}} 
= &\, \bigg\|\int_{s_0}^s \sum\limits_{\substack{\Gamma_2(k)\\ k_3\ne0}} \fe^{-it\phi} (ik_3)^{-1} 
\hat v_{1,k_1}\hat v_{2,k_2}\hat v_{3,k_3}\,\d t\bigg\|_{H^{-b(\alpha)}} \notag\\ 
\lesssim &\, 
|s-s_0|   \|v_1\|_{H^\alpha} \|v_2\|_{H^\alpha}\|v_3\|_{H^\alpha}.
\end{align*}
Combining the two inequalities above, we obtain 
\begin{align*} 
\big\|\P_{\le |s-s_0|^{-\alpha}}\mathcal C^*_5(v_1,v_2,v_3)\big\|_{L^2} 
\lesssim &
|s-s_0|^{1-b(\alpha)\alpha}  \|v_1\|_{H^\alpha} \|v_2\|_{H^\alpha}\|v_3\|_{H^\alpha}.
\end{align*}
Since $1-b(\alpha)\alpha\ge \alpha$, it follows that 
\begin{align}\label{PsmallC5}
\big\|\P_{\le |s-s_0|^{-\alpha}}\mathcal C^*_5(v_1,v_2,v_3)\big\|_{L^2} 
\lesssim 
|s-s_0|^{\alpha}  \|v_1\|_{H^\alpha} \|v_2\|_{H^\alpha}\|v_3\|_{H^\alpha}.
\end{align}

For time-independent functions $v_1$, $v_2$ and $v_3$, we have $\mathcal C^*_{53}(v_1,v_2,v_3)=\mathcal C^*_{54}(v_1,v_2,v_3)=0$ (as they contain the time derivatives of the functions $v_1$, $v_2$ and $v_3$). 
Therefore, the second part in \eqref{C5-decompose} can be estimated by using the decomposition 
$$C^*_{5}(v_1,v_2,v_3)=C^*_{51}(v_1,v_2,v_3)+C^*_{52}(v_1,v_2,v_3) . $$ 
According to Bernstein's inequality, as shown in \eqref{Bernstein}, we have 
\begin{align*}
\|\P_{> |s-s_0|^{-\alpha}}\mathcal C^*_{5j}(v_1,v_2,v_3)\|_{L^2} 
\lesssim &\, 
|s-s_0|^{\frac{23}{14}\alpha}\|\mathcal C^*_{5j}(v_1,v_2,v_3)\|_{H^{\frac{23}{14}}}\notag\\
\lesssim &\, 
|s-s_0|^{\alpha} \|v_1\|_{L^2} \|v_2\|_{L^2}\|v_3\|_{L^2}
\quad\mbox{for}\,\,\, j=1,2 , 
\end{align*} 
where the last inequality follows from \eqref{est:C41-1} and \eqref{est:C42-1} with $\sigma=\frac{23}{14}$ and $\sigma_1=\sigma_2=\sigma_3=0$. This implies that 
\begin{align}\label{PlargeC5}
\|\P_{> |s-s_0|^{-\alpha}}\mathcal C^*_{5}(v_1,v_2,v_3)\|_{L^2} 
\lesssim &\, 
|s-s_0|^{\alpha} \|v_1\|_{L^2} \|v_2\|_{L^2}\|v_3\|_{L^2} . 
\end{align} 
 
Combining \eqref{PsmallC5} and \eqref{PlargeC5}, we obtain the second result of Proposition \ref{lem:tri-linear}. 
\end{proof}

\vskip 1.5cm
\section{Construction of the low-regularity integrator}\label{sec:numer-method}
\vskip .5cm

For the simplicity of notation, we decompose the phase function $\phi=k^3-k_1^3-k_2^3-k_3^3$ into the following two parts: 
$$
\phi=\phi_{1}+\phi_{2},
$$
where
\begin{align} 
&\phi_{1}:=(k_1+k_2)^3-k_1^3-k_2^3=3k_1k_2(k_1+k_2) , \label{def:phi-1}\\
&\phi_{2}:=k^3-k_3^3-(k_1+k_2)^3=3kk_3(k_1+k_2). \label{def:phi-2}
\end{align}


Since $\P_0u^0=0$ (as assumed in Theorem \ref{main:thm1} and explained in Remark \ref{Remark-THM}), the conservation law $\int_\T u(t,x)\d x=\int_\T u^0(x)\d x$ of the KdV equation implies that 
$\P_0u(t,\cdot)=0$ for all $t\ge 0$. As a result, the twisted function $v(t,\cdot):= \fe^{t\partial_x^3}u(t,\cdot)$ also satisfies $\P_0v(t,\cdot)=0$ for all $t\ge 0$, and the KdV equation in \eqref{model} can be written as 
\begin{equation}\label{vt eq}
  \partial_t v(t,x)=\frac{1}{2}\fe^{t\partial_x^3}
  \partial_x\big[\fe^{-t\partial_x^3}v(t,x)\big]^2,\quad
  t\geq0\,\,\,\mbox{and}\,\,\,x\in\bT.
\end{equation}
We denote $v(t)=v(t,\cdot)$ for abbreviation. 

Let $t_n=n\tau$, $n=0,1,\dots,N=T/\tau$ be a partition of the time interval $[0,T]$ with stepsize $\tau$. 
Then the solution of \eqref{vt eq} can be expressed in terms of the Newton--Leibniz formula, i.e., 
\begin{align}
  v(s)=&\,
  v(t_n)+\frac{1}{2}\int_{t_n}^s\fe^{t\partial_x^3}
  \partial_x\left(\fe^{-t\partial_x^3}v(t)\right)^2 \d t \notag \\
  =&\, 
  v(t_n)+ F^n[s;v(t_n)] +r_n(s)
  \qquad\mbox{for}\,\,\, s\in[t_n,t_{n+1}] ,  \label{v-increm} 
\end{align}
where 
\begin{align}
F^n[s;v(t_n)]  
:\!&= \frac{1}{2}\int_{t_n}^s\fe^{t\partial_x^3}
  \partial_x\big(\fe^{-t\partial_x^3}v(t_n)\big)^2\d t \label{def-Fn-1}  \\
  &=\frac16  \fe^{s\partial_x^3} \P\Big[\Big(\fe^{-s\partial_x^3}\partial_x^{-1}v(t_n)\Big)^2\Big]
-\frac16 \fe^{t_n\partial_x^3} \P\Big[\Big(\fe^{-t_n\partial_x^3}\partial_x^{-1}v(t_n)\Big)^2\Big] ,
  \label{def-Fn} \\[5pt] 
r_n(s) :\!&=\frac{1}{2}\int_{t_n}^s\fe^{t\partial_x^3}
  \partial_x\left[\fe^{-t\partial_x^3}\big(v(t)-v(t_n)\big)\cdot\fe^{-t\partial_x^3}\big(v(t)+v(t_n)\big)\right]\d t ,
  \label{def:rn}
\end{align}
where the expression in \eqref{def-Fn} is given by Lemma \ref{lem:1-form}.
By using the Newton--Leibniz formula again, i.e., 
\begin{align}\label{est:v-n}
v(t_{n+1})=&\, v(t_n)+\frac{1}{2}\int_{t_n}^{t_{n+1}}\fe^{s\partial_x^3}
\partial_x\left(\fe^{-s\partial_x^3}v(s)\right)^2 \d s , 
\end{align}
and substituting expression \eqref{v-increm} into \eqref{est:v-n}, we obtain 
\begin{align}
  v(t_{n+1})=&\,
  v(t_n)+\frac{1}{2}\int_{t_n}^{t_{n+1}}\fe^{s\partial_x^3}
  \partial_x\left(\fe^{-s\partial_x^3}\big[v(t_n)+F^n[s;v(t_n)]+r_n(s)\big]\right)^2\d s \notag\\
    =&\,
    v(t_n)
    +F^n\big[t_{n+1};v(t_n)\big]
    +A^n[v(t_n)]
    +R^n_1[v],\label{v-formu-1}
\end{align}
where $A^n[v(t_n)]$ and the remainder $R^n_1[v]$ are defined by 
\begin{align}
A^n[v(t_n)]:=&\,
\int_{t_n}^{t_{n+1}}\fe^{s\partial_x^3} \partial_x\left(\fe^{-s\partial_x^3}v(t_n)\> \fe^{-s\partial_x^3}F^n[s;v(t_n)]\right)\d s,\label{def-An-1} \\
R^n_1[v]:=&\,
\frac{1}{2}\int_{t_n}^{t_{n+1}}\fe^{s\partial_x^3}
  \partial_x\left(\fe^{-s\partial_x^3}F^n[s;v(t_n)]\right)^2 \d s \notag\\
 & +\frac{1}{2}\int_{t_n}^{t_{n+1}}\fe^{s\partial_x^3}
  \partial_x\left(\fe^{-s\partial_x^3}r_n(s)\> 
  \fe^{-s\partial_x^3}\big[2v(t_n)+2F^n[s;v(t_n)]+r_n(s)\big]\right) \d s.
\end{align}

The third term on the right-hand side of \eqref{v-formu-1} can be calculated by using the integration-by-parts formula in Lemma \ref{lem:1-form}, i.e.,   
\begin{align}\label{expression-A}
A^n[v(t_n)]
&=\frac13 \fe^{t_{n+1}\partial_x^3}  \P \left(\fe^{-t_{n+1}\partial_x^3}\partial_x^{-1}v(t_n)\> \fe^{-t_{n+1}\partial_x^3}\partial_x^{-1} F^n\big[t_{n+1}; v(t_n)\big]\right)
\notag \\
&\quad\, - \frac13 \int_{t_n}^{t_{n+1}}  \fe^{s\partial_x^3}
  \P \Big(\fe^{-s\partial_x^3} \partial_x^{-1} v(t_n)\>\fe^{-s\partial_x^3} \partial_x^{-1} \partial_s F^n\big[s;v(t_n)\big]\>\Big)\d s \notag \\
&=\frac13 \fe^{t_{n+1}\partial_x^3}  \P \left(\fe^{-t_{n+1}\partial_x^3}\partial_x^{-1}v(t_n)\> \fe^{-t_{n+1}\partial_x^3}\partial_x^{-1} F^n\big[t_{n+1}; v(t_n)\big]\right)
\notag \\
&\quad\, - \frac16 \int_{t_n}^{t_{n+1}}  \fe^{s\partial_x^3}
  \P \Big(\fe^{-s\partial_x^3} \partial_x^{-1} v(t_n)\>  
\P \big[\big(\fe^{-s\partial_x^3}v(t_n)\big)^2\big]\>\Big)\d s \notag \\
&=\frac13 \fe^{t_{n+1}\partial_x^3}  \P \left(\fe^{-t_{n+1}\partial_x^3}\partial_x^{-1}v(t_n)\> \fe^{-t_{n+1}\partial_x^3}\partial_x^{-1} F^n\big[t_{n+1}; v(t_n)\big]\right)
\notag \\
&\quad\, +  \frac16 \tau \partial_x^{-1} v(t_n)\>   \P_{0}\big[v(t_n)^2\big]  + B^n[v(t_n)],
\end{align}
where 
$$
B^n[v(t_n)]
:=
- \P \int_{t_n}^{t_{n+1}} \frac16 \fe^{s\partial_x^3}
\Big(\fe^{-s\partial_x^3} \partial_x^{-1} v(t_n)\>  \fe^{-s\partial_x^3}v(t_n)\>\fe^{-s\partial_x^3}v(t_n)\>\Big) \d s .
$$
We approximate $B^n[v(t_n)]$ by considering its Fourier coefficient, i.e., 
$\mathcal{F}_0\big[ B^n[v(t_n)] \big]=0$ and for $k\ne0$ 
\begin{align*}
&\hspace{-10pt} \mathcal{F}_k\big[ B^n[v(t_n)] \big] \\ 
=&\,
-\sum_{\substack{k_1+k_2+k_3=k\\ k_1,k_2,k_3\ne0}} \frac16 \int_{t_n}^{t_{n+1}}  \fe^{-is\phi}
\frac{1}{ik_1} \hat v_{k_1}(t_n)\>\hat v_{k_2}(t_n)\>\hat v_{k_3}(t_n)\, \d s \\
=&\,
-\sum_{\substack{k_1+k_2+k_3=k\\ k_1,k_2,k_3\ne0}} \frac{1}{18i} \int_{t_n}^{t_{n+1}} \fe^{-is\phi}
\bigg( \frac{1}{k_1} + \frac{1}{k_2} + \frac{1}{k_3}\bigg) 
\hat v_{k_1}(t_n)\>\hat v_{k_2}(t_n)\>\hat v_{k_3}(t_n)\, \d s \\
=&\,
-\sum_{\substack{k_1+k_2+k_3=k\\ k_1,k_2,k_3\ne0}} \frac{1}{18i} \int_{t_n}^{t_{n+1}} \fe^{-is\phi}
\bigg( \frac{1}{k_1} + \frac{1}{k_2} + \frac{1}{k_3} - \frac{1}{k} \bigg) 
\hat v_{k_1}(t_n)\>\hat v_{k_2}(t_n)\>\hat v_{k_3}(t_n)\, \d s \\ 
&\, 
-\sum_{\substack{k_1+k_2+k_3=k\\ k_1,k_2,k_3\ne0}} \frac{1}{18ik} \int_{t_n}^{t_{n+1}} \fe^{-is\phi}
\hat v_{k_1}(t_n)\>\hat v_{k_2}(t_n)\>\hat v_{k_3}(t_n)\, \d s \\ 
=&\,
-\sum_{\substack{k_1+k_2+k_3=k\\ k_1,k_2,k_3\ne0}} \frac{1}{18i} \int_{t_n}^{t_{n+1}} \fe^{-is\phi}
\frac{\phi}{3kk_1k_2k_3} 
 \hat v_{k_1}(t_n)\>\hat v_{k_2}(t_n)\>\hat v_{k_3}(t_n)\, \d s \\ 
&\, 
-\sum_{\substack{k_1+k_2+k_3=k\\ k_1,k_2,k_3\ne0}} \frac{1}{18ik} \int_{t_n}^{t_{n+1}} \fe^{-is\phi}
\hat v_{k_1}(t_n)\>\hat v_{k_2}(t_n)\>\hat v_{k_3}(t_n)\, \d s , 
\end{align*}  
where we have used the following relation (which was discovered in \cite{WuZhao-BIT}):
$$
\frac{1}{k_1} + \frac{1}{k_2} + \frac{1}{k_3} - \frac{1}{k} 
= \frac{\phi}{3kk_1k_2k_3} .
$$
Therefore, by applying the inverse Fourier transform, we have  
\begin{align}\label{expression-B}
B^n[v(t_n)]  
=&\,
- \frac{1}{54} \fe^{s\partial_x^3}\partial_x^{-1}\Big[ \Big(\fe^{-s\partial_x^3}\partial_x^{-1}v(t_n)\Big)^3\Big] \Big|_{s=t_n}^{s=t_{n+1}} 
+ S^n[v(t_n)] ,
\end{align} 
with  
\begin{align*}
S^n[v(t_n)] 
&=- \mathcal{F}_k^{-1} \sum_{k_1+k_2+k_3=k}  \frac{1}{18ik}  \fe^{-it_n\phi} 
\hat v_{k_1}(t_n)\>\hat v_{k_2}(t_n)\>\hat v_{k_3}(t_n) \int_{0}^{\tau}\fe^{-is\phi} \d s . 
\end{align*}
In view of the expressions of $\phi_1$ and $\phi_2$ in \eqref{def:phi-1}--\eqref{def:phi-2}, if $k_1+k_2=0$ then $\phi=0$. Therefore, we can decompose the expression of $S^n[v(t_n)] $ into the following two parts (according to whether $k_1+k_2$ is zero or not): 
\begin{align*}
S^n[v(t_n)] 
&= - \mathcal{F}_k^{-1} \sum_{k_1+k_2=0} \frac{\tau}{18ik} 
\hat v_{k_1}(t_n)\>\hat v_{k_2}(t_n)\>\hat v_{k}(t_n) \\ 
&\qquad - \mathcal{F}_k^{-1} \sum_{\substack{k_1+k_2+k_3=k\\ k_1+k_2\ne0}} \frac{1}{18ik} \fe^{-it_n\phi}
\hat v_{k_1}(t_n)\>\hat v_{k_2}(t_n)\>\hat v_{k_3}(t_n) \int_{0}^{\tau} \fe^{-is\phi} \d s . 
\end{align*}
Now we use the following formula: 
\begin{align}\label{Mean-Value}
\int_{0}^{\tau} \fe^{-is\phi}\,\d s 
&=\tau M_\tau(\fe^{-is\phi_1}\fe^{-is\phi_2})
=\tau M_\tau\big(\fe^{-is\phi_1}\big)M_\tau\big(\fe^{-is\phi_2}\big) +\tau\eta(\tau,k,k_1,k_2,k_3),
\end{align}
where we have used the notation $M_\tau(f)=\tau^{-1}\int_0^\tau f(t)\d t$ defined in Section \ref{section:average}, with 
\begin{align}\label{def:eta}
\eta(\tau,k,k_1,k_2,k_3):=M_\tau(\fe^{-is\phi_1}\fe^{-is\phi_2}) -M_\tau\big(\fe^{-is\phi_1}\big)M_\tau\big(\fe^{-is\phi_2}\big).
\end{align}
Substituting \eqref{Mean-Value} into the expression of $S^n[v(t_n)]$, we obtain  
\begin{align}\label{express-C}
S^n[v(t_n)] &=- \frac1{18} \tau \partial_x^{-1} v(t_n)\>   \P_{0}\big[v(t_n)^2\big] \notag\\
&\quad\, - \mathcal{F}_k^{-1} \sum_{\substack{k_1+k_2+k_3=k\\ k_1+k_2\ne0}} \frac{\tau}{18ik} M_\tau\big(\fe^{-is\phi_1}\big)M_\tau\big(\fe^{-is\phi_2}\big) 
\fe^{-it_n\phi_1} \fe^{-it_n\phi_2}   
\hat v_{k_1}(t_n)\>\hat v_{k_2}(t_n)\>\hat v_{k_3}(t_n)  \notag\\
&\quad\, + R_2^n[v(t_n)], 
\end{align}
where  the remainder $R_2^n[v(t_n)]$ is given by 
\begin{align}\label{def-R2n}
R_2^n[v(t_n)]
&= -\mathcal{F}_k^{-1} 
\sum_{\substack{k_1+k_2+k_3=k\\ k_1+k_2\ne0}}
\frac{\tau}{18ik}  
\fe^{-it_n\phi} \eta(\tau,k,k_1,k_2,k_3)
\hat v_{k_1}(t_n)\hat v_{k_2}(t_n)\hat v_{k_3}(t_n) \d s ,
\end{align}
which will be dropped in the numerical scheme. The other terms on the right-hand side of \eqref{express-C} will be kept in the numerical scheme. 

By applying Fourier transform to \eqref{def-Fn} we can obtain $\mathcal F_{0}\big[\partial_x^{-1}F^n[t_{n+1}; v(t_n)]\big]=0$ and the following expression for $\tilde k\ne0$:  
\begin{align}\label{expr-FkFn}
\mathcal F_{\tilde k}\big[\partial_x^{-1}F^n[t_{n+1}; v(t_n)]\big]
=&\, \frac16 \sum\limits_{k_1+k_2=\tilde k}\frac{\fe^{-i\tau(\tilde k^3-k_1^3-k_2^3)}-1}{ik_1 ik_2 i\tilde k} \fe^{-it_n(\tilde k^3-k_1^3-k_2^3)}\hat v_{k_1}(t_n) \hat v_{k_2}(t_n) \notag\\
=&\, \frac12 \sum\limits_{k_1+k_2=\tilde k}\frac{\fe^{-i\tau 3 k_1k_2(k_1+k_2)}-1}{-i3k_1k_2(k_1+k_2)} \fe^{-it_n 3k_1k_2(k_1+k_2)}\hat v_{k_1}(t_n) \hat v_{k_2}(t_n) \notag\\
=&\, \frac12 \sum\limits_{k_1+k_2=\tilde k}\frac{\fe^{-i\tau \phi_1}-1}{-i\phi_1} \fe^{-it_n \phi_1}\hat v_{k_1}(t_n) \hat v_{k_2}(t_n) \notag\\
=&\, \frac{\tau}{2} \sum\limits_{k_1+k_2=\tilde k}M_\tau\big(\fe^{-is\phi_1}\big) \fe^{-it_n \phi_1}\hat v_{k_1}(t_n) \hat v_{k_2}(t_n)  ,
\end{align}
where we have used the notation $\phi_1=3 k_1k_2(k_1+k_2)$ and the following relation in the last equality: 
$$
M_\tau\big(\fe^{-is\phi_1}\big) = \tau^{-1}\int_0^\tau \fe^{-is\phi_1}\d s =\frac{\fe^{-i\tau \phi_1}-1}{-i\tau\phi_1} . 
$$
Then, substituting \eqref{expr-FkFn} into the right-hand side of \eqref{express-C} and using the notation $\phi_2=k^3-k_3^3-(k_1+k_2)^3=3kk_3(k_1+k_2)$, we obtain  
\begin{align}\label{expression-S}
S^n[v(t_n)] 
&= - \frac1{18} \tau \partial_x^{-1} v(t_n)\>   \P_{0}\big[v(t_n)^2\big] \notag\\
&\quad\, 
-  \mathcal{F}_k^{-1} \frac{1}{9ik} \sum_{\substack{\tilde k+k_3= k\\ \tilde k\ne0}}  
\frac{\tau}{2} \sum_{k_1+k_2=\tilde k} M_\tau\big(\fe^{-is\phi_1}\big) \fe^{-it_n \phi_1}\hat v_{k_1}(t_n) \hat v_{k_2}(t_n)  \>   \frac{\fe^{-i\tau \phi_{2}}-1}{-i\tau\phi_{2}}\fe^{-it_n\phi_2} \hat v_{k_3}(t_n) \notag \\ 
&\quad\,+ R_2^n[v(t_n)] \notag\\
&= - \frac1{18} \tau \partial_x^{-1} v(t_n)\>   \P_{0}\big[v(t_n)^2\big] \notag\\
&\quad\, 
- \mathcal{F}_k^{-1}  \frac{1}{9ik} \sum_{\tilde k+k_3= k}   
 \fe^{it_n\tilde k^3}\mathcal F_{\tilde k}\big[\partial_x^{-1}F^n[t_{n+1}; v(t_n)]\big] 
 \fe^{-it_nk^3}  \frac{\fe^{-i\tau (k^3-\tilde k^3-k_3^3)}-1}{-i\tau 3k\tilde kk_3}  \fe^{it_nk_3^3}\hat v_{k_3}(t_n) \notag\\ 
&\quad\,+ R_2^n[v(t_n)] \notag\\
&= - \frac1{18} \tau \partial_x^{-1} v(t_n)\>   \P_{0}\big[v(t_n)^2\big] \notag\\
 &\quad\, -\frac1{27\tau}  \fe^{s\partial_x^3}\partial_x^{-2}\Big[ \fe^{-s\partial_x^3}\partial_x^{-2}F^n[t_{n+1}; v(t_n)] \, \fe^{-s\partial_x^3}\partial_x^{-1}v(t_n)\Big]\Big|_{s=t_n}^{s=t_{n+1}}
 + R_2^n[v(t_n)]  .
\end{align}

Combining the expressions of $A^n[v(t_n)]$, $B^n[v(t_n)]$ and $S^n[v(t_n)]$ in \eqref{expression-A}, \eqref{expression-B} and \eqref{expression-S}, respectively, we obtain 
\begin{align}\label{A-G-R2}
A^n[v(t_n)] = H^n[v(t_n)] + R_2^n[v(t_n)] , 
\end{align}
with
\begin{align}
\begin{aligned}
H^n[v(t_n)]
&=\frac13 \fe^{t_{n+1}\partial_x^3}  \P \left(\fe^{-t_{n+1}\partial_x^3}\partial_x^{-1}v(t_n)\> \fe^{-t_{n+1}\partial_x^3}\partial_x^{-1} F^n\big[t_{n+1}; v(t_n)\big]\right)
\notag \\
&\quad\, +  \frac19 \tau \partial_x^{-1} v(t_n)\>   \P_{0}\big[v(t_n)^2\big]  \\
&\quad\,
- \frac{1}{54} \fe^{s\partial_x^3}\partial_x^{-1}\Big[ \Big(\fe^{-s\partial_x^3}\partial_x^{-1}v(t_n)\Big)^3\Big]  \Big|_{s=t_n}^{s=t_{n+1}} \\
&\quad\,
-\frac1{27\tau}  \fe^{s\partial_x^3}\partial_x^{-2}\Big[ \fe^{-s\partial_x^3}\partial_x^{-2}F^n\big[t_{n+1}; v(t_n)\big] \,  \fe^{-s\partial_x^3}\partial_x^{-1}v(t_n)\Big]\Big|_{s=t_n}^{s=t_{n+1}}.
\end{aligned}
\end{align}
Then, substituting \eqref{A-G-R2} into \eqref{v-formu-1}, we obtain 
\begin{align}\label{v-formu-2}
  v(t_{n+1})
    = 
  v(t_n)
    +F^n[t_{n+1}; v(t_n)]
    +H^n[v(t_n)]
    + R_1^n[v] + R_2^n[v(t_n)] .
\end{align}
By dropping the remainders $R_1^n[v]$ and $R_2^n[v(t_n)]$ in \eqref{v-formu-2}, we obtain the following time-stepping method: 
\begin{align}\label{numer-formula'}
  v^{n+1}
  = 
  v^n+F^n[t_{n+1}; v^n]  + H^n[v^n] , 
\end {align}
where $v^n$ denotes the numerical approximation to $v(t_n)$. After substituting $v^n=\fe^{t_n\partial_x^3}u^n$ into \eqref{numer-formula'}, we obtain 
\begin{align}\label{numer-method-u}
u^{n+1}
= 
\fe^{-\tau\partial_x^3}u^{n} +\fe^{-t_{n+1}\partial_x^3}F^n\big[t_{n+1}; \fe^{t_{n}\partial_x^3}u^{n} \big]  + \fe^{-t_{n+1}\partial_x^3}H^n[\fe^{t_{n}\partial_x^3}u^{n}] , 
\end {align}
which is equivalent to the numerical scheme in \eqref{numer-formula-u}, where 
$$
F[u^{n}]=\fe^{-t_{n+1}\partial_x^3}F^n\big[t_{n+1}; \fe^{t_{n}\partial_x^3}u^{n} \big]
\quad\mbox{and}\quad 
H[u^{n}]=\fe^{-t_{n+1}\partial_x^3}H^n[\fe^{t_{n}\partial_x^3}u^{n}] 
. 
$$

The rest of this article is devoted to the proof of Theorem \ref{main:thm1} on the convergence with order $\gamma$ (up to a logarithmic factor) of the proposed method for $H^\gamma$ initial data with $\gamma\in(0,1]$.



\section{Reduction to a perturbed KdV equation}\label{section:reduction}

By using the twisted function $v(t)=\fe^{t\partial_x^3}u(t)$, the KdV equation can be equivalently formulated into the following integral form: 
\begin{align}\label{est:v-nt}
v(t)= v^0 +\frac{1}{2}\int_{0}^{t}\fe^{s\partial_x^3}
\partial_x\left(\fe^{-s\partial_x^3}v(s)\right)^2 \d s \quad\mbox{for}\,\,\, t\in[0,T] . 
\end{align}
In order to establish stability estimates for the numerical scheme under low-regularity conditions below $H^1$ (especially below $H^{\frac12}$), we shall rewrite the numerical scheme in \eqref{numer-formula'} as a perturbation of the integral equation in \eqref{est:v-nt}. 

By using the relation $H^n[v^n] = A^n[v^n] - R_2^n[v^n]$, as shown in \eqref{A-G-R2}, we first rewrite the numerical scheme in \eqref{numer-formula'} as  
\begin{align}\label{numer-formula} 
  v^{n+1}
  = 
  v^n+F^n[t_{n+1}; v^n]  + A^n[v^n]-R^n_2[v^n].
 \end {align}
In view of the definitions of $F^n[t_{n+1}; v^n]$ and $A^n[v^n]$ in \eqref{def-Fn-1} and \eqref{def-An-1}, respectively, the following relation holds: 
\begin{align*}
F^n\big[t_{n+1};v^n\big]  + A^n[v^n]
=&\, \frac12 \int_{t_n}^{t_{n+1}} \fe^{s\partial_x^3}\partial_x\left(\fe^{-s\partial_x^3}v^n\right)^2\,\d s\notag\\
&+ \int_{t_n}^{t_{n+1}} \fe^{s\partial_x^3}\partial_x\left(\fe^{-s\partial_x^3}v^n\cdot \fe^{-s\partial_x^3}F^n[s;v^n]\right)\,\d s\notag\\
=&\, \frac12 \int_{t_n}^{t_{n+1}} \fe^{s\partial_x^3}\partial_x\left(\fe^{-s\partial_x^3}\Big(v^n+F^n[s;v^n]\Big)\right)^2\,\d s\notag\\
&\, -\frac12 \int_{t_n}^{t_{n+1}} \fe^{s\partial_x^3}\partial_x\left(\fe^{-s\partial_x^3}F^n[s;v^n]\right)^2\,\d s.
\end{align*}
Substituting this expression into \eqref{numer-formula}, we obtain 
\begin{align}\label{numer-formula-2}
  v^{n+1}
  = 
  v^n &+\frac12 \int_{t_n}^{t_{n+1}} \fe^{s\partial_x^3}\partial_x\left(\fe^{-s\partial_x^3}\Big(v^n+F^n[s;v^n]\Big)\right)^2\,\d s\notag\\
&-\frac12 \int_{t_n}^{t_{n+1}} \fe^{s\partial_x^3}\partial_x\left(\fe^{-s\partial_x^3}F^n[s;v^n]\right)^2\,ds-R^n_2[v^n].
 \end {align}
 
We define a continuous function $\mathscr V(t)$, $t\in[0,T]$, which has the following expression for $ t\in[t_n,t_{n+1}]$:
\begin{align}\label{def:v-s-num}
\mathscr V(t)
  = 
  v^n &+\frac12 \int_{t_n}^t \fe^{s\partial_x^3}\partial_x\left(\fe^{-s\partial_x^3}\Big(v^n+F^n[s;v^n]\Big)\right)^2\,ds\notag\\
&-\frac{t-t_n}{\tau}\left[\frac12 \int_{t_n}^{t_{n+1}} \fe^{s\partial_x^3}\partial_x\left(\fe^{-s\partial_x^3}F^n[s;v^n]\right)^2\,ds+R^n_2[v^n]\right] . 
\end{align}
In particular, $\mathscr V(t_n)=v^n$ for $n=0,1,\cdots, N$. For $ t\in[t_n,t_{n+1}]$ we further rewrite \eqref{def:v-s-num} as
\begin{align}\label{def:v-s-num-222}
  \mathscr V(t)
  =
  v^n &+\frac12 \int_{t_n}^t \fe^{s\partial_x^3}\partial_x\left(\fe^{-s\partial_x^3}\mathscr V(s)\right)^2\,\d s\notag\\
  &-\frac12 \int_{t_n}^t \fe^{s\partial_x^3}\partial_x\left(\fe^{-s\partial_x^3}\big(\mathscr V(s)-v^n-F^n[s;v^n]\big)\cdot \fe^{-s\partial_x^3}\big(\mathscr V(s)+v^n+F^n[s;v^n]\big)\right)\,\d s\notag\\
&-\frac{t-t_n}{\tau}\left[\frac12 \int_{t_n}^{t_{n+1}} \fe^{s\partial_x^3}\partial_x\left(\fe^{-s\partial_x^3}F^n[s;v^n]\right)^2\,\d s+R^n_2[v^n]\right],
\end{align} 
and then iterate this expression for $n=0,1,\dots$. This yields the following integral equation in the continuous form: 
\begin{align}\label{def:v-s-num-2}
\mathscr V(t)
=
  v^0 +\frac12 \int_{0}^t \fe^{s\partial_x^3}\partial_x\left(\fe^{-s\partial_x^3}\mathscr V(s)\right)^2\,\d s 
  +\mathcal R(t) \quad\mbox{for}\,\,\, t\in (t_n,t_{n+1}], 
\end{align} 
with a remainder 
\begin{align} \label{def-R(t)}
\mathcal R(t) 
=
&-\frac12 \int_{t_n}^t \fe^{s\partial_x^3}\partial_x\left(\fe^{-s\partial_x^3}\big(\mathscr V(s)-v^n-F^n[s;v^n]\big)\cdot \fe^{-s\partial_x^3}\big(\mathscr V(s)+v^n+F^n[s;v^n]\big)\right)\,\d s \notag\\
&-\frac{t-t_n}{\tau}\left[\frac12 \int_{t_n}^{t_{n+1}} \fe^{s\partial_x^3}\partial_x\left(\fe^{-s\partial_x^3}F^n[s;v^n]\right)^2\,\d s+R^n_2[v^n]\right] \notag\\
  &-\frac12 \sum\limits_{j=0}^{n-1} \int_{t_j}^{t_{j+1}} \fe^{s\partial_x^3}\partial_x\left(\fe^{-s\partial_x^3}\big(\mathscr V(s)-v^j-F^j\big[s;v^j\big]\big)\cdot \fe^{-s\partial_x^3}\big(\mathscr V(s)+v^j+F^j[s;v^j]\big)\right)\,\d s \notag\\
&-\sum\limits_{j=0}^{n-1}\left[\frac12 \int_{t_j}^{t_{j+1}} \fe^{s\partial_x^3}\partial_x\left(\fe^{-s\partial_x^3}F^j[s;v^j]\right)^2\,\d s+R^j_2[v^j]\right] \notag\\
=&\!: \mathcal R^*_1(t)+\mathcal R^*_2(t)+\mathcal R^*_3(t)+\mathcal R^*_4(t) \quad\mbox{for}\,\,\, t\in(t_n,t_{n+1}] . 
\end{align}
The integral equation in \eqref{def:v-s-num-2} can be viewed as a perturbation of \eqref{est:v-nt} by the remainder $\mathcal{R}(t)$. This continuous integral formulation of the numerical scheme allows us to apply low-frequency and high-frequency decomposition in estimating the stability with respect to the perturbation, which can significantly weaken the regularity conditions compared with the energy approach of stability estimates used in the literature for the numerical analysis of the KdV equation.


In order to analyze the error of the numerical approximations, we consider the following continuous and discrete error functions: 
$$
e(t) :=v(t)- \mathscr V(t) \,\,\,\mbox{for}\,\,\, t\in[0,T], \quad\mbox{and}\quad e^n :=v(t_n)-v^n ,
$$
which satisfy that $e(t_n)=e^n$. 
Then, by comparing \eqref{est:v-nt} and \eqref{def:v-s-num-2}, we obtain the following error equation: 
\begin{align}\label{en-itera}
e(t)= e^0 & +  \int_0^t\fe^{s\partial_x^3}
  \partial_x\left(\fe^{-s\partial_x^3}e(s)\>\fe^{-s\partial_x^3}\big(v(s)-\frac{1}{2}e(s)\big)\right)\,\d s - \mathcal R(t)
  \quad\mbox{for}\,\,\, t\in[0,T] .
\end{align}
To simplify the notation, we rewrite \eqref{en-itera} as 
 \begin{align}\label{def:es}
 e(t)= e^0 + \mathcal F(t)- \mathcal R(t) , 
 \end{align}
 with
\begin{align}\label{def:math-F}
 \mathcal F(t) :=\int_0^t\fe^{s\partial_x^3}
  \partial_x\Big(\fe^{-s\partial_x^3}e(s)\>\fe^{-s\partial_x^3}\Big(v(s)-\frac{1}{2}e(s)\Big)\Big)\,\d s .
\end{align}

In the next two sections, we present estimates for the local error $R_2^n[v^n]$ and the global remainder $\mathcal R(t)$. 

\section{Estimates for the local error ${R_2^n[v^n]}$}\label{section:local-error}

As an extended notation of the local error $R_2^n[v^n]$ defined in \eqref{def-R2n}, we introduce the following trilinear form:  
\begin{align}\label{def-R2n-v1v2v3}
R_2^n[v_1,v_2,v_3]
= -\mathcal{F}_k^{-1} 
\sum_{\substack{k_1+k_2+k_3=k\\ k_1+k_2\ne0}}
\frac{\tau}{18ik} 
\fe^{-it_n\phi} \eta(\tau,k,k_1,k_2,k_3)\hat v_{1,k_1}\>\hat v_{2,k_2}\>\hat v_{3,k_3} ,  
\end{align}
where the symbol $\eta(\tau,k,k_1,k_2,k_3)$ is defined in \eqref{def:eta}. 
The main result of this section is the following proposition, where the estimate in (2) is not sharp but sufficient for the  purpose of this article. 
\begin{proposition}\label{prop:est-R-n_2}
For the function $R^n_2[v_1,v_2,v_3]$ defined in \eqref{def-R2n-v1v2v3}, the following estimates hold: 
\begin{itemize}
\item[(1)] If $v_1, v_2,v_3\in H^\gamma$ with $\gamma\in [0,1]$, then  
$$
\big\|R^n_2[v_1,v_2,v_3]\big\|_{L^2}\lesssim \tau^{1+\gamma}\ln(1/\tau) \|v_1\|_{H^\gamma} \|v_2\|_{H^\gamma}\|v_3\|_{H^\gamma} .
$$
\item[(2)] If $v_1\in L^2, v_2\in H^\gamma,v_3\in  H^\gamma$ for some $\gamma>0$, then    
$$
\big\|R^n_2[v_1,v_2,v_3]\big\|_{L^2}\lesssim \tau \|v_1\|_{L^2} \|v_2\|_{H^\gamma}\|v_3\|_{H^\gamma}.
$$
Moreover, the same result holds when $v_1,v_2,v_3$ are permuted on the right-hand side. 
\end{itemize}
\end{proposition}
\begin{proof}
For the simplicity of notation, we simply write $R_2^n=R_2^n[v_1,v_2,v_3]$ throughout the proof of Proposition \ref{prop:est-R-n_2}. According to the definition in \eqref{def-R2n-v1v2v3}, the Fourier coefficients of $R_2^n$ have the following expressions: $\mathcal{F}_0[R_2^n ] =0$ and 
\begin{align*}
\mathcal{F}_k[ R_2^n ]
&= 
-\sum_{\substack{k_1+k_2+k_3=k\\ k_1+k_2\ne0}}
\frac{\tau}{18 i k}\fe^{-it_n\phi} \eta(\tau,k,k_1,k_2,k_3)
\hat v_{1,k_1}\>\hat v_{2,k_2}\>\hat v_{3,k_3} \quad\mbox{for\, $k\neq 0$} . 
\end{align*}
Similarly as the proof of Proposition \ref{lem:tri-linear}, we may assume that $\hat v_{j,k}\ge 0$ for $j=\{1,2,3\}$ and $k\in\Z$. Otherwise we can replace $\hat v_{j,k}$ by $|\hat v_{j,k}|$ in the following argument and consider the functions $\tilde v_j := \mathcal{F}_{k}^{-1} [\,|\hat v_{j,k}|\,]$, as in the proof of Lemma \ref{lem:ln-loss}. 

The proof of Proposition \ref{prop:est-R-n_2} relies on the following technical estimate for $\eta(k,k_1,k_2,k_3)$. 
\begin{lemma}\label{lem:eta}
Let $(k,k_1,k_2,k_3)\in \Z^4$ with $k_1+k_2+k_3=k$, and denote by $k_1^*,k_2^*,k_3^*$ a permutation of $k_1,k_2,k_3$ satisfying 
$$
|k_1^*|\ge |k_2^*|\ge |k_3^*|.
$$
Then the following estimate holds for $\tilde\gamma \in [0,1]${\rm:}
\begin{align*} 
\langle k\rangle^{-1}\big|\eta(\tau,k,k_1,k_2,k_3)\big|\lesssim 
\left\{\begin{aligned}
&\tau^{\tilde \gamma}  |k|^{-\frac12}|k_1^*|^{\tilde\gamma}|k_2^*|^{\tilde\gamma} |k_3^*|^{\tilde\gamma-\frac12}+  \tau^{\tilde \gamma} |k_1^*|^{\tilde\gamma} |k_2^*|^{\tilde\gamma-\frac12} |k_3^*|^{\tilde\gamma-\frac12} &&\mbox{if $kk_1k_2k_3\ne0$,} \\
&0 &&\mbox{if $kk_1k_2k_3=0$.}
\end{aligned}\right.
\end{align*}
\end{lemma}
\begin{proof}
In the case $kk_1k_2k_3=0$ either $\phi_1=0$ or $\phi_2=0$, which together with Lemma \ref{lem:average2} imply that $\eta(\tau,k,k_1,k_2,k_3)=0$. Therefore, we focus on the case $kk_1k_2k_3\ne0$ in the proof. 

According to Lemma \ref{lem:average2}, for any $\gamma_j\in [0,\tilde\gamma]$ such that $\gamma_1+\gamma_2=\tilde \gamma$, the following result holds: 
\begin{align*}
\big|\eta(\tau,k,k_1,k_2,k_3)\big|
\lesssim 
\big(\tau |\phi_1|)^{\gamma_1}\big(\tau |\phi_2|)^{\gamma_2}\> \frac{|\phi_2|^{1-\tilde\gamma}}{|\phi_1|^{1-\tilde\gamma}} 
= 
\tau^{\tilde \gamma} |\phi_1|^{\gamma_1+\tilde\gamma-1}|\phi_2|^{\gamma_2+1-\tilde\gamma} .
\end{align*}
In particular, setting $\gamma_1=0$ and $\gamma_2=\tilde\gamma$ yields
\begin{align*}
\big|\eta(\tau,k,k_1,k_2,k_3)\big|
\lesssim 
\tau^{\tilde \gamma} |\phi_1|^{\tilde\gamma-1}|\phi_2| .
\end{align*}
Since the estimate in Lemma \ref{lem:average2} is symmetric about $\alpha=\phi_1$ and $\beta=\phi_2$, 
it follows that we can switch $\phi_1$ and $\phi_2$ in the inequality above, i.e., 
\begin{align*}
\big|\eta(\tau,k,k_1,k_2,k_3)\big|
\lesssim 
\tau^{\tilde \gamma} |\phi_1||\phi_2|^{\tilde\gamma-1} . 
\end{align*}
By considering the geometric average of the two inequalities above, we obtain  
\begin{align*}
\big|\eta(\tau,k,k_1,k_2,k_3)\big|
\lesssim 
\tau^{\tilde \gamma} |\phi_1|^a |\phi_2|^b \quad\mbox{for all $a,b\in [\tilde \gamma-1, 1]$ such that $a+b=\tilde\gamma$}.
 \end{align*}
By the definitions of $\phi_1$ and $\phi_2$ in \eqref{def:phi-1}--\eqref{def:phi-2}, this implies that 
\begin{align}\label{est:eta-1}
|k|^{-1} \big|\eta(\tau,k,k_1,k_2,k_3)\big|
\lesssim 
\tau^{\tilde \gamma} |k|^{-1+b}|k_1|^a|k_2|^a|k_3|^b|k_1+k_2|^{\tilde\gamma} .
\end{align}
Since the right-hand side of \eqref{est:eta-1} is symmetric about $k_1$ and $k_2$, without loss of generality, we may assume that $|k_1|\ge|k_2|$ (the case $|k_2|\ge|k_1|$ can be considered similarly by switching $k_1$ and $k_2$ in the estimates below). 

We consider the following several cases regarding whether $|k_3|$ is larger or smaller than $|k_1|$. 

\begin{enumerate}[leftmargin=35pt]
\item[(i)] $|k_3|\gg |k_1|$: In this case, since we have already assumed that $|k_1|\ge|k_2|$, there must be $|k|\sim |k_3|\gtrsim |k_1|\ge |k_2|$. This is classified as Case 1 below. 

\item[(ii)] $|k_3|\sim |k_1|$: In this case, $|k| = |k_1+k_2+k_3| \lesssim |k_1|$. There are two subcases:

\noindent Case 1: $|k|\sim |k_3|\gtrsim |k_1|\ge |k_2|$.

\noindent Case 2: $|k_3|\sim |k_1|\gg |k|$. 

\item[(iii)] $|k_3|\ll |k_1|$: In this case, $|k| = |k_1+k_2+k_3| \lesssim |k_1|$, and there are two subcases:

\noindent Case 3: $|k|\sim |k_1|\gg |k_3|$.

\noindent Case 4: $|k_1|\gg |k|, |k_3|$. In this case, $|k_2|=|k_1+k_3-k| \sim |k_1|$.

\end{enumerate}
In the following, we estimate $|k|^{-1} \big|\eta(\tau,k,k_1,k_2,k_3)$ for the four different cases respectively.
 
Case 1: $|k|\sim |k_3|\gtrsim |k_1|\ge |k_2|$. In this case, $|k_1^*|\sim|k|\sim|k_3| $ and $|k_2^*|\sim |k_1|, |k_3^*|\sim |k_2|$.  By choosing $a=1$ and $b=\tilde\gamma-1$ in \eqref{est:eta-1} and using the relation $|k_1+k_2|\lesssim |k_1|$, we obtain  
\begin{align*}
|k|^{-1}\big|\eta(\tau,k,k_1,k_2,k_3)\big|
\lesssim &\, \tau^{\tilde \gamma} |k|^{-2+\tilde\gamma} |k_1|^{1+\tilde\gamma}|k_2| |k_3|^{\tilde\gamma-1}\\
= &\, \tau^{\tilde \gamma} |k|^{-\frac12}|k_3|^{\tilde\gamma} |k_1|^{\tilde\gamma} |k_2|^{\tilde\gamma-\frac12}
\cdot |k|^{-\frac32+\tilde\gamma}|k_3|^{-1} |k_1| |k_2|^{\frac32-\tilde\gamma}\\
\lesssim  &\, \tau^{\tilde \gamma}|k|^{-\frac12}|k_3|^{\tilde\gamma} |k_1|^{\tilde\gamma} |k_2|^{\tilde\gamma-\frac12}\\
\sim &\, \tau^{\tilde \gamma}|k|^{-\frac12}|k_1^*|^{\tilde\gamma}|k_2^*|^{\tilde\gamma} |k_3^*|^{\tilde\gamma-\frac12}.
\end{align*}

Case 2: $|k_3|\sim |k_1|\gg |k|$. In this case, $|k_1^*|\sim |k_2^*|\sim|k_1|\sim|k_3| $ and $|k_3^*|\sim |k_2|$. 
By choosing $a=\tilde\gamma-\frac12$ and $b=\frac12$ in \eqref{est:eta-1}  and using the relation $|k_1+k_2|\lesssim |k_1|$, we obtain   
\begin{align*} 
|k|^{-1}\big|\eta(\tau,k,k_1,k_2,k_3)\big| 
\lesssim \tau^{\tilde \gamma}|k|^{-\frac12}|k_1|^{2\tilde\gamma}|k_2|^{\tilde\gamma-\frac12}
\sim\tau^{\tilde \gamma} |k|^{-\frac12}|k_1^*|^{\tilde\gamma}|k_2^*|^{\tilde\gamma}|k_3^*|^{\tilde\gamma-\frac12}.
\end{align*}
%
%
%

Case 3: $|k|\sim |k_1|\gg |k_3|$. In this case, $|k_1^*|\sim|k|\sim|k_1| $ and $k_2^*=k_2$ or $k_3$.  We may assume that $k_2^*=k_2$ and $k_3^*=k_3$ as the other case can be treated in the same way. 
By choosing $a=\tilde\gamma-\frac12, b=\frac12$ in \eqref{est:eta-1} and using the relation $|k_1+k_2|\lesssim |k_1|$, we obtain  
\begin{align*}
|k|^{-1}\big|\eta(\tau,k,k_1,k_2,k_3)\big|
\lesssim &\, \tau^{\tilde \gamma} |k|^{-\frac12}|k_1|^{\tilde\gamma-\frac12}|k_2|^{\tilde\gamma-\frac12}
|k_3|^{\frac12}|k_1|^{\tilde\gamma} \\
\lesssim &\,\tau^{\tilde \gamma} |k_1|^{-1+2\tilde\gamma} |k_2|^{\tilde\gamma-\frac12} |k_3|^\frac12\\
\lesssim &\,\tau^{\tilde \gamma}|k_1|^{\tilde\gamma}|k_2|^{\tilde\gamma-\frac12}|k_3|^{\tilde\gamma-\frac12}
\sim \tau^{\tilde \gamma}|k_1^*|^{\tilde\gamma}|k_2^*|^{\tilde\gamma-\frac12}|k_3^*|^{\tilde\gamma-\frac12}.
\end{align*}

Case 4:  $|k_1|\sim |k_2|\gg |k|, |k_3|$. In this case, $|k_1^*|\sim|k_2^*|\sim |k_1|\sim|k_2| $ and $|k_3^*|\sim|k_3|$.  By choosing $a=\tilde\gamma-1$ and $b=1$ in \eqref{est:eta-1}, we obtain  
\begin{align*}
|k|^{-1}\big|\eta(\tau,k,k_1,k_2,k_3)\big|
\lesssim &\,\tau^{\tilde \gamma}|k_1|^{3\tilde\gamma-2}|k_3| 
\lesssim \tau^{\tilde \gamma}|k_1|^{\tilde\gamma}|k_2|^{\tilde\gamma-\frac12}|k_3|^{\tilde\gamma-\frac12}\cdot |k_1|^{\tilde\gamma-\frac32}|k_3|^{\frac32-\tilde\gamma}\\
\lesssim &\,\tau^{\tilde \gamma}|k_1|^{\tilde\gamma}|k_2|^{\tilde\gamma-\frac12}|k_3|^{\tilde\gamma-\frac12}
\sim \tau^{\tilde \gamma}|k_1^*|^{\tilde\gamma}|k_2^*|^{\tilde\gamma-\frac12}|k_3^*|^{\tilde\gamma-\frac12}.
\end{align*}

Finally, by collecting the estimates in the four cases above, we obtain the desired result.
\end{proof}

We continue with the proof of Proposition \ref{prop:est-R-n_2}. Without loss of generality, we may assume that $|k_1|\ge |k_2|\ge |k_3|$ in applying Lemma \ref{lem:eta}, which implies the following result for $\tilde\gamma\in [0,1]$: 
\begin{align}\label{est:R2n}
\big|\mathcal{F}_k[ R_2^n ]\big|
\lesssim 
\left\{\begin{aligned}
&\tau^{1+\tilde\gamma} \hspace{-5pt} \sum_{k_1+k_2+k_3=k} \hspace{-5pt}
\big( |k|^{-\frac12}|k_1|^{\tilde\gamma}|k_2|^{\tilde\gamma} |k_3|^{\tilde\gamma-\frac12}+|k_1|^{\tilde\gamma} |k_2|^{\tilde\gamma-\frac12} |k_3|^{\tilde\gamma-\frac12}\big)
\hat v_{1,k_1} \hat v_{2,k_2} \hat v_{3,k_3} \\[-5pt]
&\hspace{258pt} \mbox{if $kk_1k_2k_3\ne0$}, \\[5pt] 
&0\hspace{252pt} \mbox{if $kk_1k_2k_3=0$}. 
\end{aligned}\right.
\end{align}
In view of \eqref{est:R2n}, we only need to consider the case $kk_1k_2k_3\ne0$ when estimating $|\mathcal{F}_k[ R_2^n ]|$.\bigskip

\noindent{\it Proof of (1)}: In the case $\gamma\in (0,1]$, we choose $\tilde\gamma =\gamma - \theta $ in \eqref{est:R2n} with $\theta \in [0,\frac{\gamma}{2}]$, i.e.,
\begin{align*}
\big|\mathcal{F}_k[  R_2^n ]\big|
\lesssim 
&\, \tau^{1+\gamma} \tau^{-\theta} \hspace{-5pt} \sum_{k_1+k_2+k_3=k} 
 |k|^{-\frac12}|k_1|^{\gamma-\theta}|k_2|^{\gamma-\theta} |k_3|^{\gamma-\frac12-\theta} \hat v_{1,k_1} \hat v_{2,k_2} \hat v_{3,k_3} 
\\
&\, + \tau^{1+\gamma} \tau^{-\theta} \hspace{-5pt} \sum_{k_1+k_2+k_3=k}  
        |k_1|^{\gamma-\theta} |k_2|^{\gamma-\frac12-\theta} |k_3|^{\gamma-\frac12-\theta} \hat v_{1,k_1} \hat v_{2,k_2} \hat v_{3,k_3} \\
\lesssim 
&\, \tau^{1+\gamma} \tau^{-\theta} \hspace{-5pt} \sum_{k_1+k_2+k_3=k} 
 |k|^{-\frac12-\theta} |k_3|^{-\frac12-\theta} \, 
 |k_1|^{\gamma}\hat v_{1,k_1} |k_2|^{\gamma}\hat v_{2,k_2}|k_3|^{\gamma}\hat v_{3,k_3} 
\\
&\, + \tau^{1+\gamma} \tau^{-\theta} \hspace{-5pt} \sum_{k_1+k_2+k_3=k}  
        |k_2|^{-\frac12-\theta} |k_3|^{-\frac12-\theta} 
        \, 
 |k_1|^{\gamma}\hat v_{1,k_1} |k_2|^{\gamma}\hat v_{2,k_2}|k_3|^{\gamma}\hat v_{3,k_3} .
\end{align*}
Then, by applying Lemma \ref{lem:ln-loss} with $A=\tau^{-1}$ and $\theta_0=\frac\gamma2$, we obtain 
\begin{align}\label{est:R2n+0}
\big\|R^n_2[v_1,v_2,v_3]\big\|_{L^2}\lesssim \tau^{1+\gamma}\ln(1/\tau) \|v_1\|_{H^\gamma} \|v_2\|_{H^\gamma}\|v_3\|_{H^\gamma} .
\end{align}

In the case $\gamma=0$, we decompose $\mathcal{F}_k[R_2^n]$ into two parts according to whether $(k_1,k_2,k_3)\notin\Gamma(k)$ or $(k_1,k_2,k_3)\in\Gamma(k)$, where $\Gamma(k)$ is defined in Section \ref{section:Gammak}. Namely, 
\begin{subequations}\label{eq:5.6-12}
\begin{align}
\mathcal{F}_k[R_2^n]
= 
&\sum_{\substack{k_1+k_2+k_3=k\\ k_1+k_2\ne0\\ k_1+k_3=0\,\,\mbox{\footnotesize or}\,\,k_2+k_3=0}}
\frac{i\tau}{18k}\fe^{-it_n\phi} \eta(\tau,k,k_1,k_2,k_3)
\hat v_{1,k_1}\>\hat v_{2,k_2}\>\hat v_{3,k_3} \label{eq:5.6-1}\\
&\quad +\sum_{(k_1,k_2,k_3)\in\Gamma(k)}
\frac{i\tau}{18k}\fe^{-it_n\phi} \eta(\tau,k,k_1,k_2,k_3)
\hat v_{1,k_1}\>\hat v_{2,k_2}\>\hat v_{3,k_3} . \label{eq:5.6-2}
\end{align}
\end{subequations} 
The summation in \eqref{eq:5.6-1} can be estimated as follows, where we focus on the case $k_2+k_3=0$ (the case $k_1+k_3=0$ can be treated in the same way): 
If $k_2+k_3=0$ then $k_1=k$ and therefore 
\begin{align*}
|\eqref{eq:5.6-1}| 
& \lesssim \tau \sum_{k_2+k_3=0}
|k|^{-1} |\eta(\tau,k,k,k_2,k_3)|  \hat v_{1,k}\>\hat v_{2,k_2}\>\hat v_{3,k_3} \\
& \lesssim \tau 
 \F_0[v_2v_3] \, \hat v_{1,k} &&\mbox{(since $|\eta|\lesssim 1$ and $|k|^{-1}\lesssim 1$)} \\
 & \lesssim \tau  \|v_2v_3\|_{L^1} \, \hat v_{1,k} ,
\end{align*}
which implies that 
\begin{align}\label{R2n1}
\big\|\F_k^{-1}[\eqref{eq:5.6-1}]\big\|_{L^2}\lesssim \tau \|v_1\|_{L^2} \|v_2\|_{L^2}\|v_3\|_{L^2}. 
\end{align}

According to Lemma \ref{lem:a4-est} (2), $\Gamma(k)$ can be decomposed into two parts, i.e., $ \Gamma(k)= \Gamma_1(k)\cup \Gamma_2(k)$, with  
\begin{align*} 
\begin{aligned}
&|k|\sim |k_1|\sim|k_2|\sim |k_3| &&\mbox{for}\,\,\, (k_1,k_2,k_3)\in \Gamma_1(k) ,\\
&|\phi|\ge |k_m|^2 = |k_1|^2  &&\mbox{for}\,\,\, (k_1,k_2,k_3)\in \Gamma_2(k). 
\end{aligned}
\end{align*} 
In view of \eqref{eq:5.6-12} we can decompose $R_2^n$ into 
\begin{align}\label{decompose-R2n}
R_2^n = \F_k^{-1}[\eqref{eq:5.6-1}] + w_1 + w_2 ,
\end{align}
with 
$$
w_j=\F_k^{-1} \sum_{(k_1,k_2,k_3)\in\Gamma_j(k)}
\frac{i\tau}{18k}\fe^{-it_n\phi} \eta(\tau,k,k_1,k_2,k_3)
\hat v_{1,k_1}\>\hat v_{2,k_2}\>\hat v_{3,k_3} . 
$$ 
By choosing $\tilde\gamma=0$ in Lemma \ref{lem:eta}, we have 
\begin{align*}
|\F_k[w_1]| 
&\lesssim 
\tau \hspace{-5pt} \sum_{(k_1,k_2,k_3)\in\Gamma_1(k)} \hspace{-5pt}
\big( |k|^{-\frac12}|k_3|^{-\frac12}+ |k_2|^{-\frac12} |k_3|^{-\frac12}\big)
\hat v_{1,k_1} \hat v_{2,k_2} \hat v_{3,k_3} \\
&\lesssim 
\tau \hspace{-5pt} \sum_{k_1+k_2+k_3=k} \hspace{-5pt}
|k_1|^{-\frac13}|k_2|^{-\frac13}|k_3|^{-\frac13} \hat v_{1,k_1} \hat v_{2,k_2} \hat v_{3,k_3} 
\quad\mbox{if}\,\,\, kk_1k_2k_3\ne 0 , 
\end{align*}
where the last inequality uses the equivalence relation $|k|\sim |k_1|\sim|k_2|\sim |k_3| $ for $(k_1,k_2,k_3)\in\Gamma_1(k)$. By applying the Fourier inversion formula and using the Sobolev embedding $L^2(\T)\hookrightarrow W^{-\frac13,6}(\T)$, we have 
\begin{align}\label{R2n2}
\|w_1\|_{L^2} 
&\lesssim 
\tau \||\partial_x|^{-\frac13}v_1\|_{L^6}  \||\partial_x|^{-\frac13}v_2\|_{L^6}  \||\partial_x|^{-\frac13}v_3\|_{L^6} 
\lesssim 
\tau \|v_1\|_{L^2} \|v_2\|_{L^2} \|v_3\|_{L^2} .
\end{align}

For $(k_1,k_2,k_3)\in \Gamma_2(k)\subset\Gamma(k)$ we have $|\phi|=|3(k_1+k_2)(k_1+k_3)(k_2+k_3)|\ne 0$ and $|\phi|\ge |k_m|^2=|k_1|^2$. In this case, by choosing $\alpha=\phi_1$, $\beta=\phi_2$ and $\alpha+\beta=\phi$ in \eqref{average2-2}, we obtain
$$
\tau |k|^{-1} |\eta(\tau,k,k_1,k_2,k_3)|
\lesssim \tau |k|^{-1} \tau^{-1}|\phi|^{-1} 
\lesssim |k_1|^{-2} .
$$
By choosing $\tilde\gamma=0$ in Lemma \ref{lem:eta} we also have
$$
\tau |k|^{-1} |\eta(\tau,k,k_1,k_2,k_3)|
\lesssim \tau (|k|^{-\frac12} |k_3|^{-\frac12}+ |k_2|^{-\frac12} |k_3|^{-\frac12}) .
$$
The geometric average of the two inequalities above yields that 
\begin{align*}
\tau |k|^{-1} |\eta(\tau,k,k_1,k_2,k_3)|
&\lesssim \tau \tau^{-\theta}|k_1|^{-2\theta} ( |k|^{-\frac12+\frac{\theta}{2}} |k_3|^{-\frac12+\frac{\theta}{2}}+ |k_2|^{-\frac12+\frac{\theta}{2}} |k_3|^{-\frac12+\frac{\theta}{2}} ) \\
&\lesssim \tau \tau^{-\theta}( |k|^{-\frac12-\frac{\theta}{2}} |k_3|^{-\frac12-\frac{\theta}{2}}+ |k_2|^{-\frac12-\frac{\theta}{2}} |k_3|^{-\frac12-\frac{\theta}{2}} ) \quad\forall\,\theta\in[0,1] . 
\end{align*}
Then we can apply Lemma \ref{lem:ln-loss} with $A=\tau^{-2}$ and $\theta_0=\frac12$, which implies that 
\begin{align}\label{R2n3}
\|w_2\|_{L^2} 
&\lesssim 
\tau\ln(1/\tau)  \|v_1\|_{L^2} \|v_2\|_{L^2} \|v_3\|_{L^2}  . 
\end{align}
Finally, substituting estimates \eqref{R2n1}, \eqref{R2n2} and \eqref{R2n3} into \eqref{decompose-R2n}, we obtain 
$$
\|R^n_2\|_{L^2}\lesssim \tau\ln(1/\tau) \|v_1\|_{L^2} \|v_2\|_{L^2}\|v_3\|_{L^2} .
$$
This, together with \eqref{est:R2n+0} for the case $\gamma\in(0,1]$, gives the desired estimate in (1).\bigskip

\noindent{\it Proof of (2)}: 
By choosing $\tilde\gamma=0$ in \eqref{est:R2n}, we have  
\begin{align*}
|\mathcal{F}_k[ R_2^n ]|
&\lesssim 
\tau\sum_{k_1+k_2+k_3=k}
\big( |k|^{-\frac12}|k_3|^{-\frac12}+|k_2|^{-\frac12} |k_3|^{-\frac12}\big)
\hat v_{1,k_1}\hat v_{2,k_2}\hat v_{3,k_3}\\
&=\tau\mathcal{F}_k\Big( |\partial_x|^{-\frac12}\big(v_1\>v_2\>|\partial_x|^{-\frac12}v_3\big)+v_1\>|\partial_x|^{-\frac12}v_2\>|\partial_x|^{-\frac12}v_3\Big).
\end{align*}
Therefore, by the Plancherel identity and the Sobolev inequality, we obtain that 
\begin{align*}
\|R^n_2\|_{L^2}
\lesssim &\tau  \big\||\partial_x|^{-\frac12}\big(v_1\>v_2\>|\partial_x|^{-\frac12}v_3\big)\big\|_{L^2}
+\tau  \big\|v_1\>|\partial_x|^{-\frac12}v_2\>|\partial_x|^{-\frac12}v_3\big\|_{L^2}\\
\lesssim & \tau \big\|v_1\>v_2\>|\partial_x|^{-\frac12}v_3\big\|_{L^{1+}}+\tau \big\|v_1\>|\partial_x|^{-\frac12}v_2\>|\partial_x|^{-\frac12}v_3\big\|_{L^2}\\
\lesssim & \tau \big\|v_1\big\|_{L^2}\big\|v_2\big\|_{L^{2+}}\big\||\partial_x|^{-\frac12}v_3\big\|_{L^\infty}+\tau \big\|v_1\big\|_{L^2}\big\||\partial_x|^{-\frac12}v_2\big\|_{L^\infty}\big\||\partial_x|^{-\frac12}v_3\big\|_{L^\infty}\\
\lesssim & \tau \big\|v_1\big\|_{L^2}\big\|v_2\big\|_{H^{\gamma}}\big\|v_3\big\|_{H^{\gamma}}.
\end{align*}
This proves the desired result in (2). 
\end{proof}

By choosing $v_1=v_2=v_3=v^n$ in Proposition \ref{prop:est-R-n_2}, we obtain the following estimate for the local error $R_2^n[v^n]$. 

\begin{corollary}\label{cor:est-R_n_2} 
Under the assumptions of Theorem \ref{main:thm1}, the following estimate holds:  
$$
\|R_2^n[v^n]\|_{L^2}\lesssim \tau^{1+\gamma}\ln(1/\tau) + \tau\big(\|e^n\|_{L^2}+|\ln(1/\tau)|^2\|e^n\|_{L^2}^3\big).
$$
\end{corollary}
\begin{proof}
The assumptions of Theorem \ref{main:thm1} guarantees that $\|v(t_n)\|_{H^\gamma}\lesssim 1$. 
By substituting the expression $v^n=v(t_n)-e^n$ into the trilinear form $R_2^n[v^n]$, we obtain 
\begin{align*}
R^n_2[v^n]
= R^n_2[v(t_n)] & + R^n_2[e^n,v(t_n),v(t_n)] + R^n_2[v(t_n),e^n,v(t_n)] + R^n_2[v(t_n),v(t_n),e^n] \\
& +R^n_2[e^n,e^n,v(t_n)]+R^n_2[e^n,v(t_n),e^n] +R^n_2[v(t_n),e^n,e^n]
- R^n_2[e^n,e^n,e^n] .
\end{align*}
We apply Proposition \ref{prop:est-R-n_2} (1) to the first term and last four terms on the right-hand side. This yields 
\begin{align*}
\|R^n_2[v(t_n)]\|_{L^2} &\lesssim \tau^{1+\gamma}\ln(1/\tau) \|v(t_n)\|_{H^\gamma}^3  , \\
\|R^n_2[e^n,e^n,v(t_n)]\|_{L^2} &\lesssim \tau\ln(1/\tau) \|e^n\|_{L^2}^2 \|v(t_n)\|_{L^2} 
\lesssim \tau \|e^n\|_{L^2} + \tau |\ln(1/\tau)|^2 \|e^n\|_{L^2}^3  ,\\
\|R^n_2[e^n,v(t_n),e^n]\|_{L^2} &\lesssim \tau\ln(1/\tau) \|e^n\|_{L^2}^2 \|v(t_n)\|_{L^2} \lesssim \tau \|e^n\|_{L^2} + \tau |\ln(1/\tau)|^2 \|e^n\|_{L^2}^3  ,\\
\|R^n_2[v(t_n),e^n,e^n]\|_{L^2} &\lesssim \tau\ln(1/\tau) \|e^n\|_{L^2}^2 \|v(t_n)\|_{L^2} \lesssim \tau \|e^n\|_{L^2} + \tau |\ln(1/\tau)|^2 \|e^n\|_{L^2}^3  ,\\
\|R^n_2[e^n,e^n,e^n]\|_{L^2} &\lesssim \tau\ln(1/\tau) \|e^n\|_{L^2}^3  .
\end{align*}
Furthermore, we apply Proposition \ref{prop:est-R-n_2} (2) to the rest terms in the above expression of $R^n_2[v^n]$. Then we obtain the desired result in Corollary \ref{cor:est-R_n_2}. 
\end{proof}

\section{Estimates for the global remainder ${\mathcal{R}(t)}$}\label{section:global-error}

The main result of this section is the following proposition, where ${\mathcal{R}(t)}$ is defined in \eqref{def-R(t)}. 

\begin{proposition}\label{prop:est-R} 
Under the assumptions of Theorem \ref{main:thm1}, the following estimate holds: 
$$
\|\mathcal R(t) \|_{L^2}\lesssim t_{n+1}\tau^\gamma\ln(1/\tau)
+ t_{n+1} \max_{0\le j\le n} \big(\|e^j\|_{L^2}+|\ln(1/\tau)|^7 \|e^j\|_{L^2}^8\big)
\quad\mbox{for}\,\,\, t\in(t_n,t_{n+1}]. 
$$
\end{proposition}

\begin{proof}
For $t\in(t_n,t_{n+1}]$, we consider the expression of $\mathcal{R}(t)$ in \eqref{def-R(t)} and estimate $\mathcal R^*_1(t)$, $\mathcal R^*_2(t)$, $\mathcal R^*_3(t)$ and $\mathcal R^*_4(t)$ separately in the following subsections.

\subsection{Estimation of $\mathcal R^*_1(t)$}

In view of the definition in \eqref{def-R(t)}, we can decompose $\mathcal R^*_1(t)$ into the following two parts:  
\begin{align*}
\mathcal R^*_1(t)=
  &- \int_{t_n}^t \fe^{s\partial_x^3}\partial_x\left(\fe^{-s\partial_x^3}\big(\mathscr V(s)-v^n-F^n[s;v^n]\big)\cdot \fe^{-s\partial_x^3}\big(v^n+F^n[s;v^n]\big)\right)\,\d s\\
&-\frac12 \int_{t_n}^t \fe^{s\partial_x^3}\partial_x\left(\fe^{-s\partial_x^3}\big(\mathscr V(s)-v^n-F^n[s;v^n]\big)\right)^2\,\d s.
\end{align*}
Then we can apply the integration-by-parts formula in Lemma \ref{lem:1-form} with $\mathcal V(t_n)=v^n$ and $F^n[t_n;v^n]=0$. This yields the following expression of $\mathcal R^*_1(t)$: 
\begin{align}\label{def-R11234}
\mathcal R^*_1(t)
=&- \frac13 \fe^{t\partial_x^3}\P\left(e^{-t\partial_x^3}\partial_x^{-1}\big(\mathscr V(t)-v^n-F^n[t;v^n]\big)\cdot e^{-t\partial_x^3}\partial_x^{-1}\big(v^n+F^n[t;v^n]\big)\right) \notag\\
  &+\frac13\int_{t_n}^t \fe^{s\partial_x^3}\P\left(\fe^{-s\partial_x^3}\partial_x^{-1}\partial_s\big(\mathscr V(s)-v^n-F^n[s;v^n]\big)\cdot \fe^{-s\partial_x^3}\partial_x^{-1}\big(v^n+F^n[s;v^n]\big)\right)\,\d s \notag\\
   &+\frac13\int_{t_n}^t \fe^{s\partial_x^3}\P\left(\fe^{-s\partial_x^3}\partial_x^{-1}\big(\mathscr V(s)-v^n-F^n[s;v^n]\big)\cdot \fe^{-s\partial_x^3}\partial_x^{-1}\partial_sF^n[s;v^n]\right)\,\d s \notag\\
&-\frac12 \int_{t_n}^t \fe^{s\partial_x^3}\partial_x\left(\fe^{-s\partial_x^3}\big(\mathscr V(s)-v^n-F^n[s;v^n]\big)\right)^2\,\d s \notag\\
= &\!: \mathcal R^*_{11}(t)+\mathcal R^*_{12}(t)+\mathcal R^*_{13}(t)+\mathcal R^*_{14}(t).
\end{align}

As an extended notation of the function $F^n[s;v]$ defined in \eqref{def-Fn}, we consider the following bilinear form (for  time-independent functions $v_1,v_2$ such that $\P_0v_1=\P_0v_2=0$): 
\begin{align}
F^n[s;v_1,v_2]  
:\!&= \frac{1}{2}\int_{t_n}^s \fe^{t\partial_x^3}
  \partial_x\big(\fe^{-t\partial_x^3}v_1\>\fe^{-t\partial_x^3}v_2\big)\d t \label{def-Fn-1-2}  \\
  &=\frac16\fe^{t\partial_x^3}\Big(\fe^{-t\partial_x^3}\partial_x^{-1}v_1\>\fe^{-t\partial_x^3}\partial_x^{-1}v_2\Big)\Big|_{t=t_n}^{t=s}
  \label{def-Fn-22} \quad\mbox{for}\,\,\, s\in[t_n,t_{n+1}] , 
\end{align}
where the last equality is obtained by using the integration-by-parts formula in Lemma \ref{lem:1-form}. This extended notation satisfies that $F^n[s;v,v]  =F^n[s;v]$. The following bilinear estimate for $F^n[s;v_1,v_2]$ will be used.  

\begin{lemma}\label{lem:est-Fn}
For $v_1,v_2\in H^\gamma$ with $\gamma\in(0,1]$, and $s\in[t_n,t_{n+1}]$, the following result holds:
$$
\big\|F^n\big[s;v_1,v_2\big]\big\|_{H^\beta}
\lesssim \tau^{\frac{1+\gamma-\beta}{1+\gamma-\beta_0(\gamma)}} \|v_1\|_{H^\gamma} \|v_2\|_{H^\gamma}
\quad\mbox{for\, $\beta\in [\beta_0(\gamma), 1+\gamma]$} , 
$$
where
$$
\beta_0(\gamma)
= \left\{\begin{aligned}
&2\gamma-\frac32 && \mbox{when } \gamma\in \Big(0,\frac12\Big) \\
&-\frac12- && \mbox{when } \gamma=\frac12\\ 
&\gamma-1 && \mbox{when } \gamma\in \Big(\frac12,1 \Big] .
\end{aligned}\right. 
$$
\end{lemma}
\begin{proof}
On one hand, we consider the $ H^{\beta_0(\gamma)}$ norm of the expression in \eqref{def-Fn-1-2} and prove the following result: 
\begin{align}
\big\|F^n\big[t;v_1,v_2\big]\big\|_{H^{\beta_0(\gamma)}}
\lesssim \tau \|v_1\|_{H^\gamma} \|v_2\|_{H^\gamma}.
\end{align}
Indeed,  the expression in \eqref{def-Fn-1-2} gives us the following inequality:  
\begin{align*}
\| F^n[s;v_1,v_2] \|_{H^{\beta_0(\gamma)}}
\lesssim &\, \int_{t_n}^s \big\| \fe^{-t\partial_x^3}v_1\>\fe^{-t\partial_x^3}v_2\big\|_{H^{\beta_0(\gamma)+1}} \d t .
\end{align*}
If $\gamma\in(0,\frac14]$ then $2\gamma-\frac12\le 0$. In this case, we apply the Sobolev embedding $H^\gamma\hookrightarrow L^p$ and $L^{\frac{p}{2}}\hookrightarrow H^{-\frac12+2\gamma}$ with $p=2/(1-2\gamma)$, i.e.,  
\begin{align*}
\| F^n[s;v_1,v_2] \|_{H^{\beta_0(\gamma)}}
\lesssim &\, \int_{t_n}^s
\big\| \fe^{-t\partial_x^3}v_1\>\fe^{-t\partial_x^3}v_2\big\|_{L^{\frac{p}{2}}} \d t \\
\lesssim &\, \int_{t_n}^s
\big\| \fe^{-t\partial_x^3}v_1\big\|_{L^p} \big\|\fe^{-t\partial_x^3}v_2\big\|_{L^p} \d t \\
\lesssim &\, \int_{t_n}^s
\big\| \fe^{-t\partial_x^3}v_1\big\|_{H^\gamma} \big\|\fe^{-t\partial_x^3}v_2\big\|_{H^\gamma} \d t \\
\lesssim &\, \tau \|v_1\|_{H^\gamma}\|v_2\|_{H^\gamma} . 
\end{align*}
If $\gamma\in(\frac14,\frac12)$ then $0<-\frac12+2\gamma\le \gamma$. In this case, we apply the Kato--Ponce inequality in Lemma \ref{lem:kato-Ponce} and the Sobolev embeddings $H^\gamma\hookrightarrow W^{-\frac12+2\gamma,p_1}$, $H^\gamma\hookrightarrow L^{p_2}$ with $\frac12-\gamma=\frac12-\frac{1}{p_1}$ and $\frac{1}{p_1}+\frac{1}{p_2}=\frac12$, i.e., 
\begin{align*}
\| F^n[s;v_1,v_2] \|_{H^{\beta_0(\gamma)}}
\lesssim &\, \int_{t_n}^s\big\| \fe^{-t\partial_x^3}v_1\big\|_{W^{-\frac12+2\gamma,p_1}} \big\|\>\fe^{-t\partial_x^3}v_2\|_{L^{p_2}} \d t \\
&\,
+ \int_{t_n}^s \big\| \fe^{-t\partial_x^3}v_1\big\|_{L^{p_2}} \big\|\>\fe^{-t\partial_x^3}v_2\|_{W^{-\frac12+2\gamma,p_1}} \d t \\
\lesssim &\, \int_{t_n}^s \big\| \fe^{-t\partial_x^3}v_1\big\|_{H^\gamma} \big\|\>\fe^{-t\partial_x^3}v_2\|_{H^\gamma} \d t \\
\lesssim &\, \tau \|v_1\|_{H^\gamma}\|v_2\|_{H^\gamma} . 
\end{align*}
If $\gamma=\frac12$, then $\beta_0(\gamma)+1=\frac12-$ and  therefore, by the Kato--Ponce inequality in Lemma \ref{lem:kato-Ponce} and the Sobolev embedding $H^\frac12\hookrightarrow L^p$ for all $p\in[1,\infty)$, we have 
\begin{align*}
\| F^n[s;v_1,v_2] \|_{H^{\beta_0(\gamma)}}
\lesssim &\, \int_{t_n}^s \big\| \fe^{-t\partial_x^3}v_1\>\fe^{-t\partial_x^3}v_2\big\|_{H^{\frac12-}} \d t \\
\lesssim &\, \int_{t_n}^s \big\| \fe^{-t\partial_x^3}v_1\big\|_{H^{\frac12}}
\big\| \fe^{-t\partial_x^3}v_2\big\|_{H^{\frac12}} \d t \\
\lesssim &\, \tau \|v_1\|_{H^{\frac12}}\|v_2\|_{H^{\frac12}} . 
\end{align*}
If $\gamma\in(\frac12,1]$, then $\beta_0(\gamma)+1=\gamma$ and therefore, by the Kato--Ponce inequality in Lemma \ref{lem:kato-Ponce} and the Sobolev embedding $H^\gamma\hookrightarrow L^\infty$, we have 
\begin{align*}
\| F^n[s;v_1,v_2] \|_{H^{\beta_0(\gamma)}}
\lesssim &\, \int_{t_n}^s \big\| \fe^{-t\partial_x^3}v_1\>\fe^{-t\partial_x^3}v_2\big\|_{H^{\gamma}} \d t \\
\lesssim &\, \int_{t_n}^s \big\| \fe^{-t\partial_x^3}v_1\big\|_{H^{\gamma}}
\big\| \fe^{-t\partial_x^3}v_2\big\|_{H^{\gamma}} \d t \\
\lesssim &\, \tau \|v_1\|_{H^\gamma}\|v_2\|_{H^\gamma} . 
\end{align*}

On the other hand, we consider the $H^{1+\gamma}$ norm of the expression in \eqref{def-Fn-22}, which gives us the following inequality:  
$$
\|F^n[s;v_1,v_2]\|_{H^{1+\gamma}}
\lesssim 
\|v_1\|_{H^\gamma}\|v_2\|_{H^\gamma} .
$$

By considering the complex interpolation between the estimates for $\|F^n[s;v_1,v_2]\|_{H^{\beta_0(\gamma)}}$ and $ \|F^n[s;v_1,v_2]\|_{H^{1+\gamma}}$, we obtain the result of Lemma \ref{lem:est-Fn}. 
\end{proof}

By using the result of Lemma \ref{lem:est-Fn}, we manage to obtain the following several useful estimates for $\mathscr V(t)-v^n-F^n[t;v^n]$. 

\begin{lemma}\label{lem:V-increm-est} 
Under the assumptions of Theorem \ref{main:thm1}, the following estimates hold for $t\in [t_n,t_{n+1}]${\rm:} 
\begin{itemize}
\item[(1)] 
$
\big\|\partial_x^{-1}\big(\mathscr V(t)-v^n-F^n[t;v^n]\big)\big\|_{L^2}
\lesssim \tau^{1+\gamma}\ln(1/\tau)+ \tau\big(\|e^n\|_{L^2}+|\ln(1/\tau)|^3\|e^n\|_{L^2}^4\big).
$
\item[(2)] 
$
\big\|\partial_x^{-1}\partial_t\big(\mathscr V(t)-v^n-F^n[t;v^n]\big)\big\|_{L^2}\lesssim \tau^{\gamma}\ln(1/\tau)+ \big(\|e^n\|_{L^2}+|\ln(1/\tau)|^3\|e^n\|_{L^2}^4\big) .
$
\item[(3)] 
$
\big\|\mathscr V(t)-v^n-F^n[t;v^n]\big\|_{L^2}\lesssim 
\tau^{\gamma} + \|e^n\|_{L^2}+\|e^n\|_{L^2}^4 .
$

\end{itemize}
\end{lemma}

\begin{proof}
By comparing the expressions of $\mathscr V(t)$ and $F^n[t;v^n]$ in \eqref{def:v-s-num} and \eqref{def-Fn-1}, respectively, we can derive the following expression:  
\begin{subequations}\label{eq:6.3}
\begin{align}
\mathscr V(t)-v^n-F^n[t;v^n] 
=&\, 
 \int_{t_n}^t \fe^{s\partial_x^3}\partial_x\left(\fe^{-s\partial_x^3}v^n\cdot \fe^{-s\partial_x^3}F^n[s;v^n]\Big)\right)\,\d s\label{V-v-F-1}\\
&\, + \frac12 \int_{t_n}^t \fe^{s\partial_x^3}\partial_x\left(\fe^{-s\partial_x^3}F^n[s;v^n]\Big)\right)^2\,\d s \label{V-v-F-2} \\
&\, -\frac{t-t_n}{\tau}\left[\frac12 \int_{t_n}^{t_{n+1}} \fe^{s\partial_x^3}\partial_x\left(\fe^{-s\partial_x^3}F^n[s;v^n]\right)^2\,\d s+R^n_2[v^n]\right].\label{V-v-F-3}
\end{align}
\end{subequations}

{\it Proof of (1)}: 
Although the three terms in \eqref{V-v-F-1}, \eqref{V-v-F-2} and \eqref{V-v-F-3} should be estimated separately, 
we focus on the estimation of \eqref{V-v-F-1} as the other terms can be treated similarly. 

By applying $\partial_x^{-1}$ to \eqref{V-v-F-1} and substituting $v^n=v(t_n)-e^n$ into the result, we obtain 
\begin{subequations}\label{6.3a-12}
\begin{align}
\partial_x^{-1}\eqref{V-v-F-1}=& \int_{t_n}^t \fe^{s\partial_x^3}\P\left(\fe^{-s\partial_x^3}v(t_n)\cdot \fe^{-s\partial_x^3}F^n[s;v^n]\Big)\right)\,\d s\label{V-v-F-1-1}\\
&- \int_{t_n}^t \fe^{s\partial_x^3}\P\left(\fe^{-s\partial_x^3}e^n\cdot \fe^{-s\partial_x^3}F^n[s;v^n]\right)\,\d s\label{V-v-F-1-2}.
\end{align}
\end{subequations}
The expression in \eqref{V-v-F-1-1} can be estimated by applying inequality \eqref{lem:ga-aga}, i.e.,   
\begin{align}
\|\eqref{V-v-F-1-1}\|_{L^2}
&\lesssim \tau\|v(t_n)\|_{H^\gamma} \sup_{s\in[t_n,t_{n+1}]} \|F^n[s;v^n]\|_{H^{a(\gamma)}} ,  \label{vF-vF-F1}
\end{align}
where $\|F^n[s; v^n]\|_{H^{a(\gamma)}} $ can be decomposed into the following two parts using the triangle inequality:  
\begin{align} \label{L2-Fsvn-tri}
\|F^n[s; v^n]\|_{H^{a(\gamma)}} 
&\le \|F^n[s; v(t_n)]\|_{H^{a(\gamma)}}+\|F^n[s; v^n]-F^n[s; v(t_n)]\|_{H^{a(\gamma)}}. 
\end{align}

The first term on the right-hand side of \eqref{L2-Fsvn-tri} can be estimated by choosing $\beta=a(\gamma)$ in Lemma \ref{lem:est-Fn}, which implies that   
\begin{align}\label{est:F-a-alpha}
\|F^n[s;v(t_n)]\|_{H^{a(\gamma)}}
\lesssim \tau^{\frac{1+\gamma-a(\gamma)}{1+\gamma-\beta_0(\gamma)}} \|v(t_n)\|_{H^\gamma}^2
\lesssim \tau^\gamma \|v(t_n)\|_{H^\gamma}^2 ,
\end{align}
where the last inequality follows from the fact that $\frac{1+\gamma-a(\gamma)}{1+\gamma-\beta_0(\gamma)} \ge \gamma$ for $\gamma\in(0,1]$ for the expression of $a(\gamma)$ in \eqref{def:a-gamma}. 

The  second term on the right-hand side of \eqref{L2-Fsvn-tri} can be estimated as follows: 
\begin{align}\label{est:F-c-alpha}
\|F[s; v^n]-F[s; v(t_n)]\|_{H^{a(\gamma)}} 
&=\|F[s; e^n,2v(t_n)-e^n]\|_{H^{a(\gamma)}} \notag\\
&\lesssim \|F[s; e^n,2v(t_n)-e^n]\|_{H^{1}} \qquad\mbox{(since $a(\gamma)\le 1$)}\notag\\
&\lesssim \|e^n\|_{L^2}(\|v(t_n)\|_{L^2} + \|e^n\|_{L^2}) , 
\end{align}
where the last inequality can be obtained from expression \eqref{def-Fn-22} directly. 
Hence, by substituting \eqref{est:F-a-alpha}--\eqref{est:F-c-alpha} into \eqref{L2-Fsvn-tri}, we obtain that 
\begin{align}\label{H-al-Fsvn}
\|F^n[s; v^n]\|_{H^{a(\gamma)}} 
&\lesssim \tau^\gamma +\|e^n\|_{L^2} + \|e^n\|_{L^2}^2 
\quad\mbox{for}\,\,\, s\in[t_n,t_{n+1}] , 
\end{align}
where we have omitted the dependence on $\|v(t_n)\|_{H^\gamma}^2$. 

Substituting \eqref{H-al-Fsvn} into \eqref{vF-vF-F1} yields 
\begin{align}\label{estimate:V-v-F-1-1}
\|\eqref{V-v-F-1-1}\|_{L^2}\lesssim \tau^{1+\gamma} +\tau\big(\|e^n\|_{L^2}+ \|e^n\|_{L^2}^2 \big).
\end{align}

The expression in \eqref{V-v-F-1-2} can be estimated by 
\begin{align}\label{estimate:V-v-F-1-2'}
\|\eqref{V-v-F-1-2}\|_{L^2}
&\lesssim \tau \|e^n\|_{L^2} \sup_{s\in[t_n,t_{n+1}]} \|F^n[s;v^n]\|_{H^1} . 
\end{align}
From expression \eqref{def-Fn} we see that 
\begin{align}\label{H1-Fsvn}
 \|F^n[s;v^n]\|_{H^1}
&\lesssim \|v^n\|_{L^2}^2
= \|v(t_n)-e^n\|_{L^2}^2 
\lesssim 1+\|e^n\|_{L^2}^2 .
\end{align}
Inserting \eqref{H1-Fsvn} into \eqref{estimate:V-v-F-1-2'}, gives that 
\begin{align}\label{estimate:V-v-F-1-2}
\|\eqref{V-v-F-1-2}\|_{L^2}
&\lesssim \tau \big(\|e^n\|_{L^2}+\|e^n\|_{L^2}^3\big).
\end{align}

Then, substituting  \eqref{estimate:V-v-F-1-1}  and  \eqref{estimate:V-v-F-1-2} into \eqref{6.3a-12}, we obtain  
\begin{align*}
\|\partial_x^{-1}\eqref{V-v-F-1}\|_{L^2}
\lesssim \tau^{1+\gamma} + \tau\big(\|e^n\|_{L^2}+\|e^n\|_{L^2}^3 \big). 
\end{align*}

The estimation of $\|\partial_x^{-1}\eqref{V-v-F-2}\|_{L^2} $ and $\|\partial_x^{-1}\eqref{V-v-F-3}\|_{L^2} $ are easier than $\|\partial_x^{-1}\eqref{V-v-F-1}\|_{L^2} $. In fact, employing \eqref{H-al-Fsvn}, we have that
\begin{align}\label{L2-Fsvn}
\|F^n[s; v^n]\|_{L^2} \le \|F^n[s; v^n]\|_{H^{a(\gamma)}} 
\lesssim \tau^\gamma +\|e^n\|_{L^2} + \|e^n\|_{L^2}^2 
\quad\mbox{for}\,\,\, s\in[t_n,t_{n+1}],
\end{align}
Hence, by applying $\partial_x^{-1}$ to \eqref{V-v-F-2} and considering the $L^2$ norm of the result, and then by  \eqref{L2-Fsvn} and \eqref{H1-Fsvn},  we have 
\begin{align*}
\|\partial_x^{-1}\eqref{V-v-F-2}\|_{L^2} 
&\lesssim 
\int_{t_n}^t \|F^n[s;v^n] \|_{L^2} \|F^n[s;v^n] \|_{H^1} \d s \\
&\lesssim 
 \tau^{1+\gamma} +\tau\big(\|e^n\|_{L^2}+ \|e^n\|_{L^2}^4\big).
\end{align*} 
%
%

Since $\partial_x^{-1}\eqref{V-v-F-3}$ consists of $R_2^n[v^n]$ and a term similar as $\partial_x^{-1}\eqref{V-v-F-2}$, it follows that we can directly use the result in Corollary \ref{cor:est-R_n_2} and the above estimate for $\|\partial_x^{-1}\eqref{V-v-F-2}\|_{L^2} $. Then we obtain the following result: 
\begin{align*}
\|\partial_x^{-1}\eqref{V-v-F-3}\|_{L^2}
&\lesssim 
\tau^{1+\gamma}\ln(1/\tau)+ \tau\big(\|e^n\|_{L^2}+|\ln(1/\tau)|^2\|e^n\|_{L^2}^3\big)
+ \tau^{1+\gamma} +\tau\big(\|e^n\|_{L^2}+ \|e^n\|_{L^2}^4\big) \\
&\lesssim 
\tau^{1+\gamma}\ln(1/\tau)+ \tau\big(\|e^n\|_{L^2}+|\ln(1/\tau)|^3\|e^n\|_{L^2}^4\big) . 
\end{align*}
where the last inequality is obtained by considering the two cases $|\ln(1/\tau)|\|e^n\|_{L^2}\le 1$ and $|\ln(1/\tau)|\|e^n\|_{L^2}\ge 1$ separately. 
This proves the first result of Lemma \ref{lem:V-increm-est}. \medskip

\noindent{\it Proof of (2)}: 
By applying $\partial_x^{-1}\partial_t$ to the expression in \eqref{eq:6.3}, we have 
\begin{align}
\partial_x^{-1}\partial_t\big(\mathscr V(t)- v^n-F^n[t;v^n]\big) 
=&\,
 \fe^{t\partial_x^3}\P\left(e^{-t\partial_x^3}v^n\cdot e^{-t\partial_x^3}F^n[t;v^n]\Big)\right)\notag\\
&\,+\frac12 \fe^{t\partial_x^3}\P\left(e^{-t\partial_x^3}F^n[t;v^n]\Big)\right)^2  \label{V-vn-F-t}\\
&\,-\frac1{\tau}\left[\frac12 \int_{t_n}^{t_{n+1}} \fe^{s\partial_x^3}\P\left(\fe^{-s\partial_x^3}F^n[s;v^n]\right)^2\,\d s+\partial_x^{-1}R^n_2[v^n]\right] , \notag
\end{align} 
which can be treated in the same way as \eqref{6.3a-12}. This proves the second result of Lemma \ref{lem:V-increm-est}. \medskip

\noindent{\it Proof of (3)}: 
By using the integration-by-parts formula in Lemma \ref{lem:1-form} and the expression of $F[s;v^n]$ in \eqref{def-Fn-1}, we find that 
\begin{subequations}\label{6.3a-2-12}
\begin{align}
\eqref{V-v-F-1}=&\, \fe^{t\partial_x^3}\P\Big(e^{-t\partial_x^3}\partial_x^{-1}v^n\cdot e^{-t\partial_x^3}\partial_x^{-1}F^n[t; v^n]\Big) \label{6.3a-2-123-1}\\
&\,- \int_{t_n}^t \fe^{s\partial_x^3}\P\Big(\P\big[\big(\fe^{-s\partial_x^3}v^n\big)^2\big]\, \fe^{-s\partial_x^3}\partial_x^{-1}v^n\Big) .\label{6.3a-2-123-2}
\end{align}
\end{subequations}
The first term on the right-hand side of \eqref{6.3a-2-12} can be estimated by using the H\"older and Sobolev inequalities as follows: 
\begin{align*}
\|\eqref{6.3a-2-123-1}\|_{L^2}
&\lesssim \|e^{-t\partial_x^3}\partial_x^{-1}v^n\|_{L^2} \|e^{-t\partial_x^3}\partial_x^{-1}F^n[t; v^n]\|_{L^\infty} \\
&\lesssim \|v^n\|_{L^2} \|F^n[t; v^n]\|_{L^2}.
\end{align*}
Then, by substituting \eqref{L2-Fsvn} into the inequality above, we further obtain that 
\begin{align*}
\|\eqref{6.3a-2-123-1}\|_{L^2}
&\lesssim (\|v(t_n)\|_{L^2} + \|e^n\|_{L^2})
\big( \tau^\gamma +\|e^n\|_{L^2} + \|e^n\|_{L^2}^2  \big) \\ 
&\lesssim \tau^\gamma 
+ \|e^n\|_{L^2}+\|e^n\|_{L^2}^2 
+\|e^n\|_{L^2}^3 \\
&\lesssim \tau^\gamma 
+ \|e^n\|_{L^2}+\|e^n\|_{L^2}^3.
\end{align*}

We rewrite \eqref{6.3a-2-123-2} as 
\begin{align*}
\eqref{6.3a-2-123-2}=
&- \int_{t_n}^t \fe^{s\partial_x^3}\P\Big(\P\big[\big(\fe^{-s\partial_x^3}v(t_n)\big)^2\big]\cdot \fe^{-s\partial_x^3}\partial_x^{-1}v(t_n)\Big)\,\d s \\ 
&+\int_{t_n}^t \fe^{s\partial_x^3}\P\Big(\P\big[\big(\fe^{-s\partial_x^3}v(t_n)\big)^2\big]\cdot \fe^{-s\partial_x^3}\partial_x^{-1}v(t_n)-\P\big[\big(\fe^{-s\partial_x^3}v^n\big)^2\big]\cdot \fe^{-s\partial_x^3}\partial_x^{-1}v^n\Big)\,\d s. 
\end{align*}
and apply Proposition \ref{lem:tri-linear} (2) with $\alpha=\gamma$ and $\alpha=0$ for the first and second term, respectively. This yields the following result:
\begin{align*}
\|\eqref{6.3a-2-123-2}\|_{L^2}
\lesssim \tau^\gamma\big\|v(t_n)\big\|_{H^\gamma}^3+\big(\|e^n\|_{L^2}\|v(t_n)\|_{L^2}^2+\|e^n\|_{L^2}^2\|v(t_n)\|_{L^2}+\|e^n\|_{L^2}^3\big) .
\end{align*}
Substituting the estimates of $\|\eqref{6.3a-2-123-1}\|_{L^2}$ and $\|\eqref{6.3a-2-123-2}\|_{L^2}$ into \eqref{6.3a-2-12}, we obtain 
\begin{align*}
\|\eqref{V-v-F-1}\|_{L^2}
\lesssim \tau^\gamma+\big(\|e^n\|_{L^2}+\|e^n\|_{L^2}^3\big) . 
\end{align*}
The following result can be obtained similarly by using integration by parts as in \eqref{6.3a-2-12}: 
\begin{align*}
\|\eqref{V-v-F-2}\|_{L^2}+\|\eqref{V-v-F-3}\|_{L^2}
\lesssim \tau^\gamma+\big(\|e^n\|_{L^2}+\|e^n\|_{L^2}^4\big). 
\end{align*}
Finally, substituting the above estimates of $\|\eqref{V-v-F-1}\|_{L^2}$, $\|\eqref{V-v-F-2}\|_{L^2}$ and $\|\eqref{V-v-F-3}\|_{L^2}$ into \eqref{eq:6.3} yields the third result of Lemma \ref{lem:V-increm-est}. 
\end{proof}
 
By applying Lemma \ref{lem:V-increm-est} (1) to the expression of $\mathcal R_{11}^*(t)$ in \eqref{def-R11234} and using \eqref{L2-Fsvn}, we obtain 
\begin{align*}
\big\|\mathcal R_{11}^*(t)\big\|_{L^2}
&\lesssim \big\|\partial_x^{-1}\big(\mathscr V(t)-v^n-F^n[t;v^n]\big)\big\|_{L^2}
\big\|v^n+F^n[t;v^n]\big\|_{L^2} \\
&\lesssim \big\|\partial_x^{-1}\big(\mathscr V(t)-v^n-F^n[t;v^n]\big)\big\|_{L^2}
\big\|e^n+F^n[t;v^n] + v(t_n)\big\|_{L^2} \\
&\lesssim 
\big[\tau^{1+\gamma}\ln(1/\tau)+ \tau\big(\|e^n\|_{L^2}+ |\ln(1/\tau)|^3 \|e^n\|_{L^2}^4\big)\big]
(\tau^\gamma+\|e^n\|_{L^2}+\|e^n\|_{L^2}^2+1) \\
&\lesssim 
\tau^{1+\gamma}\ln(1/\tau)+ \tau\big(\|e^n\|_{L^2}+|\ln(1/\tau)|^5 \|e^n\|_{L^2}^6\big) \qquad\mbox{for}\,\,\, t\in(t_n,t_{n+1}]. 
\end{align*}
Similarly, by applying Lemma \ref{lem:V-increm-est} (2) to the expression of $\mathcal R_{12}^*(t)$ in \eqref{def-R11234}, we obtain 
\begin{align*}
\big\|\mathcal R_{12}^*(t)\big\|_{L^2}
\lesssim &\, \tau \big\|\partial_x^{-1}\partial_t\big(\mathscr V(t)-v^n-F^n[t;v^n]\big)\big\|_{L^\infty_tL^2_x}
\big\|v^n+F^n[t;v^n]\big\|_{L^\infty_tL^2_x}\\
\lesssim &\,
\tau^{1+\gamma}\ln(1/\tau)+ \tau\big(\|e^n\|_{L^2}+ |\ln(1/\tau)|^5 \|e^n\|_{L^2}^6\big) \qquad\mbox{for}\,\,\, t\in(t_n,t_{n+1}]. 
\end{align*}

We substitute expression \eqref{def-Fn-1} into the expression of $\mathcal R_{13}^*(t)$ in \eqref{def-R11234}, i.e., 
\begin{align*}
\mathcal R_{13}^*(t) =
   &\, \frac13\int_{t_n}^t \fe^{s\partial_x^3}\P\left(\P\Big(\fe^{-s\partial_x^3}v^n\Big)^2\cdot \fe^{-s\partial_x^3}\partial_x^{-1}\big(\mathscr V(s)-v^n-F^n[s;v^n]\big)\right)\,\d s . 
\end{align*}
Then we apply Proposition \ref{lem:tri-linear} (1) to the expression above. This yields the following result: 
\begin{align*}
\big\|\mathcal R_{13}^*(t) \big\|_{L^2}
\lesssim &\, \tau \|v^n\|_{L^2}^2 \big\|\mathscr V(s)-v^n-F^n[s;v^n]\big\|_{L^\infty_tL^2_x}\\
 &\, + \tau \|v^n\|_{L^2}^2 \big\|\partial_t\big(\mathscr V(s)-v^n-F^n[s;v^n]\big)\big\|_{L^\infty_tH^{-\frac{23}{14}}_x}\\
&\, +\big\|v^n\big\|_{L^2}^2\big\|\mathscr V(s)-v^n-F^n[s;v^n]\big\|_{L^\infty_tH^{-\frac{23}{14}}_x}\\
\lesssim &\,
\big[ \tau^{1+\gamma}\ln(1/\tau)+ \tau \big(\|e^n\|_{L^2}+ |\ln(1/\tau)|^3 \|e^n\|_{L^2}^4\big) \big] \|v^n\|_{L^2}^2 ,\end{align*} 
where the last inequality follows from Lemma \ref{lem:V-increm-est}. 
Since $\|v^n\|_{L^2}^2 \lesssim \|e^n\|_{L^2}^2 +\|v(t_n)\|_{L^2}^2 $, it follows that 
\begin{align*}
\big\|\mathcal R_{13}^*(t) \big\|_{L^2} 
&\lesssim \tau^{1+\gamma}\ln(1/\tau)+ \tau\big(\|e^n\|_{L^2}+|\ln(1/\tau)|^3 \|e^n\|_{L^2}^4+|\ln(1/\tau)|^3 \|e^n\|_{L^2}^6\big) \\
&\lesssim \tau^{1+\gamma}\ln(1/\tau)+ \tau\big(\|e^n\|_{L^2}+|\ln(1/\tau)|^5 \|e^n\|_{L^2}^6\big) .
\end{align*}

The expression of $\mathcal R_{14}^*(t)$ can be rewritten as follows, by using the integration-by-parts formula in Lemma \ref{lem:1-form}, i.e., 
\begin{align*}
\mathcal R_{14}^*(t) 
=&\,
-\frac16 \fe^{t\partial_x^3}\P\left(e^{-t\partial_x^3}\partial_x^{-1}(\mathscr V(t)-v^n-F^n[t;v^n]\big)\right)^2\\
&\,+\frac16 \int_{t_n}^t \fe^{s\partial_x^3}\P\left(\fe^{-s\partial_x^3}\partial_x^{-1}\big(\mathscr V(s)-v^n-F^n[s;v^n]\big)\cdot \fe^{-s\partial_x^3}\partial_x^{-1}\partial_s\big(\mathscr V(s)-v^n-F^n[s;v^n]\big)\right)\,\d s.
\end{align*}
Then, by the H\"older and Sobolev inequalities, we have that 
\begin{align*}
&\hspace{-15pt} \sup_{t\in(t_n,t_{n+1}]} \big\|\mathcal R_{14}^*(t) \big\|_{L^2} \\
\lesssim &\,\big\|\mathscr V(t)-v^n-F^n[t;v^n]\big\|_{L^\infty(t_n,t_{n+1};L^2)}\big\|\partial_x^{-1}\big(\mathscr V(t)-v^n-F^n[t;v^n]\big)\big\|_{L^\infty(t_n,t_{n+1};L^2)} \\
&\, +\tau  \big\|\mathscr V(t)-v^n-F^n[t;v^n]\big\|_{L^\infty(t_n,t_{n+1};L^2)}\big\|\partial_x^{-1}\partial_t\big(\mathscr V(t)-v^n-F^n[t;v^n]\big)\big\|_{L^\infty(t_n,t_{n+1};L^2)} \\ 
\lesssim &\, 
\big( \tau^{\gamma} +  \|e^n\|_{L^2}+ \|e^n\|_{L^2}^4 \big) 
\big[ \tau^{1+\gamma}\ln(1/\tau)+ \tau \big(\|e^n\|_{L^2}+ |\ln(1/\tau)|^3 \|e^n\|_{L^2}^4\big) \big] \\
\lesssim &\, 
\tau^{1+\gamma} + \tau\big(\|e^n\|_{L^2}+|\ln(1/\tau)|^7 \|e^n\|_{L^2}^8\big) . 
\end{align*}

Combining with the estimates of $\big\|\mathcal R_{1j}^*(t) \big\|_{L^2} $, $j=1,2,3,4$, 
\begin{align}\label{prop:est-R_n_1}
\big\|\mathcal R^*_1(t)\big\|_{L^2}\le  \tau^{1+\gamma}\ln(1/\tau)+ \tau\big(\|e^n\|_{L^2}+|\ln(1/\tau)|^7 \|e^n\|_{L^2}^8\big) .
\end{align}

\subsection{Estimation of  $\mathcal R^*_2(t)$, $\mathcal R^*_3(t)$ and $\mathcal R^*_4(t)$}

For the remainder $\mathcal R_2^*(t)$ defined in \eqref{def-R(t)}, i.e., 
$$
\mathcal R_2^*(t)=-\frac{t-t_n}{\tau}\Big[\frac12 \int_{t_n}^{t_{n+1}} \fe^{s\partial_x^3}\partial_x\Big(\fe^{-s\partial_x^3}F^n[s;v^n]\Big)^2\,\d s+R^n_2[v^n]\Big] , 
$$
we rewrite $ \int_{t_n}^{t_{n+1}} \fe^{s\partial_x^3}\partial_x\big(\fe^{-s\partial_x^3}F^n[s;v^n]\big)^2\,\d s$ as   
\begin{align*}
&\hspace{-10pt} \int_{t_n}^{t_{n+1}} \fe^{s\partial_x^3}\partial_x\left(\fe^{-s\partial_x^3}F^n[s;v^n]\right)^2\,\d s \\ 
=&\, \int_{t_n}^{t_{n+1}} \fe^{s\partial_x^3}\partial_x\left(\fe^{-s\partial_x^3}F^n[s;v(t_n)]\right)^2\,\d s\\
&\, +\int_{t_n}^{t_{n+1}} \fe^{s\partial_x^3}\partial_x\left(\fe^{-s\partial_x^3}\big(F^n[s;v^n]-F^n[s;v(t_n)]\big)\cdot \fe^{-s\partial_x^3}\big(F^n[s;v^n]+F^n[s;v(t_n)]\big)\right)\,\d s.
\end{align*}
and use inequality \eqref{lem:ga-aga}. Then we obtain 
\begin{align*}
\big\|\mathcal R^*_2(t)\big\|_{L^2}
\lesssim&\,\, \tau \big\|\partial_xF^n[s;v(t_n)]\big\|_{L^\infty(t_n,t_{n+1};H^\gamma)}\big\|F^n[s;v(t_n)]\big\|_{L^\infty(t_n,t_{n+1};H^{a(\gamma)})}\\
&\,+\tau \big\|F^n[s;v^n]-F^n[s;v(t_n)]\big\|_{L^\infty(t_n,t_{n+1};H^1)}\big\|F^n[s;v^n]+F^n[s;v(t_n)]\big\|_{L^\infty(t_n,t_{n+1};H^1)}\\
&\,+ \big\| R^n_2[v^n] \big\|_{L^2}.
\end{align*}
By using Lemma \ref{lem:est-Fn} with $\beta=1+\gamma$ and $\beta=a(\gamma)$, with 
$\frac{1+\gamma-a(\gamma)}{1+\gamma-\beta_0(\gamma)} \ge \gamma$ for $\gamma\in(0,1]$ for the expression of $a(\gamma)$ in \eqref{def:a-gamma}, 
as well as inequalities \eqref{cor:est-R_n_2} and \eqref{est:F-c-alpha}, we obtain 
\begin{align}\label{prop:est-R_n_2} 
\big\|\mathcal R^*_2(t)\big\|_{L^2} 
\lesssim \tau^{1+\gamma}\ln(1/\tau)+\tau\big(\|e^n\|_{L^2}+|\ln(1/\tau)|^3\|e^n\|_{L^2}^4\big).
\end{align}

According to the definitions in \eqref{def-R(t)}, $\mathcal R^*_3(t)$ and $\mathcal R^*_4(t)$ can be expressed in terms of $ \mathcal R_1^*(t_{j+1})$ and $ \mathcal R_2^*(t_{j+1})$ as follows: 
$$
\mathcal R^*_3(t)=\sum\limits_{j=0}^{n-1} \mathcal R_1^*(t_{j+1})
\quad\mbox{and}\quad  
\mathcal R^*_4(t)=\sum\limits_{j=0}^{n-1} \mathcal R_2^*(t_{j+1})
\quad\mbox{for}\,\,\, t\in(t_n,t_{n+1}].
$$
By substituting estimates \eqref{prop:est-R_n_1} and \eqref{prop:est-R_n_2} into relations above, we obtain  the following result:
\begin{align}\label{prop:est-R_n_34}
\big\|\mathcal R^*_3(t)\big\|_{L^2}+\big\|\mathcal R^*_4(t)\big\|_{L^2}
\lesssim 
t_{n+1}\tau^\gamma\ln(1/\tau)
+t_{n+1} \max_{0\le j\le n} \big(\|e^j\|_{L^2}+|\ln(1/\tau)|^7\|e^j\|_{L^2}^8\big). 
\end{align}

Then, by combining the estimates in \eqref{prop:est-R_n_1} and \eqref{prop:est-R_n_2} and \eqref{prop:est-R_n_34}, we obtain the result of Proposition \ref{prop:est-R}.
\end{proof}

\section{Stability estimates using low-high frequency decompositions} \label{sec:regularity}

The main result of this section is the following proposition, which concerns the estimation of the function $\mathcal F(t)$ defined in \eqref{def:math-F}. 

\begin{proposition}\label{prop:math-F}
Under the assumptions of Theorem \ref{main:thm1}, there exists a positive constant $C$, which may depend on $\|v\|_{L^\infty(0,T;H^\gamma)}$ but is independent of $\tau$, such that 
\begin{align*}
\big\|\mathcal F(t)\big\|_{L^2}
\le 
Ct\tau^\gamma 
&\, + CN^2t \|e\|_{L^\infty(0,t;L^2)}\big(1+\|e\|_{L^\infty(0,t;L^2)}\big) \\
&\, + C\big(t+N^{-\frac{23}{14}}\big)\|e\|_{L^\infty(0,t;L^2)}\big(1+|\ln(1/\tau)|^{96}\|e\|_{L^\infty(0,t;L^2)}^{96}\big) 
\quad\forall\, N\ge 1. 
\end{align*}
\end{proposition}

In the proof of Proposition \ref{prop:math-F} we need to use the two estimates in the following lemma. 
\begin{lemma}\label{lem:e-R-pt}
Under the assumptions of Theorem \ref{main:thm1}, there exists a positive constant $C$, which may depend on $\|v\|_{L^\infty(0,T;H^\gamma)}$ but is independent of $\tau$, such that 
\begin{itemize}
\item[{\rm(1)}] $\|\partial_t\mathcal R\|_{L^\infty(t_n,t_{n+1};H^{-\frac32-})}
\le C \tau^\gamma\ln(1/\tau)+C\big(\|e^n\|_{L^2}+|\ln(1/\tau)|^7 \|e^n\|_{L^2}^8 \big)$. 
\item[{\rm(2)}] $\|\partial_te\|_{L^\infty(t_n,t_{n+1};H^{-\frac32-})}
\le C \tau^\gamma\ln(1/\tau)+C\big(\|e^n\|_{L^2}+|\ln(1/\tau)|^7 \|e^n\|_{L^2}^8 \big)$.
\end{itemize}
\end{lemma}
\begin{proof}
By differentiating \eqref{def-R(t)} with respect to $t\in(t_n,t_{n+1}]$, we can find the following expression: 
\begin{align}
\partial_t \mathcal R(t)
=&\, -\frac12 \fe^{t\partial_x^3}\partial_x\left(e^{-t\partial_x^3}\big(\mathscr V(t)-v^n-F^n[t;v^n]\big)\cdot e^{-t\partial_x^3}\big(\mathscr V(t)+v^n+F^n[t;v^n]\big)\right)\notag\\
&\, -\frac1\tau \left[\frac12 \int_{t_n}^{t_{n+1}} \fe^{t\partial_x^3}\partial_x\left(e^{-t\partial_x^3}F^n[t;v^n]\right)^2\,\d t+R^n_2[v^n]\right]\notag\\
=&\, -\frac12 \fe^{t\partial_x^3}\partial_x\left(e^{-t\partial_x^3}\big(\mathscr V(t)-v^n-F^n[t;v^n]\big)\cdot e^{-t\partial_x^3}\big(\mathscr V(t)-v^n-F^n[t;v^n]\big)\right) \notag\\
&\, - \fe^{t\partial_x^3}\partial_x\left(e^{-t\partial_x^3}\big(\mathscr V(t)-v^n-F^n[t;v^n]\big)\cdot e^{-t\partial_x^3}\big(v^n+F^n[t;v^n]\big)\right) \notag\\
&\, - \tau^{-1} \mathcal R^*_2(t_{n+1}).\label{ps-math-R}
\end{align}
Then, by applying the Sobolev embedding inequality $L^1\hookrightarrow H^{-\frac12-}$ and the H\"older inequality $\|fg\|_{L^1}\le \|f\|_{L^2}\|g\|_{L^2}$, we have 
\begin{align*}
\big\|\partial_t\mathcal R\big\|_{L^\infty(t_n,t_{n+1};H^{-\frac32-})}
\lesssim &\, \big\|\mathscr V(t)-v^n-F^n[t;v^n]\big\|_{L^\infty(t_n,t_{n+1};L^2)}^2 \\ 
&\, + \big\|\mathscr V(t)-v^n-F^n[t;v^n]\big\|_{L^\infty(t_n,t_{n+1};L^2)}  \big\|v^n+F^n[t;v^n]\big\|_{L^\infty(t_n,t_{n+1};L^2)} \\
&\, 
+ \tau^{-1}\big\|\mathcal R^*_2(t_{n+1})\big\|_{L^2}\\
\le &\, \big( \tau^{\gamma} +\|e^n\|_{L^2}+\|e^n\|_{L^2}^4 \big)^2 \\
&\,+ \big(\tau^{\gamma} + \|e^n\|_{L^2}+\|e^n\|_{L^2}^4 \big)(1+\|e^n\|_{L^2}^2) \\
&\, +\tau^{\gamma}\ln(1/\tau)+ \|e^n\|_{L^2}+|\ln(1/\tau)|^3\|e^n\|_{L^2}^4 \\
\le &\, \tau^\gamma\ln(1/\tau) + \|e^n\|_{L^2}+|\ln(1/\tau)|^7\|e^n\|_{L^2}^8 , 
\end{align*}
where the second to last inequality follows from Lemma \ref{lem:V-increm-est} and \eqref{prop:est-R_n_2}, together with the following estimate of $\|v^n+F^n[t;v^n]\|_{L^\infty(t_n,t_{n+1};L^2)}$ by using the expression in \eqref{def-Fn}: 
\begin{align}\label{v+F}
\|v^n+F^n[t;v^n]\|_{L^\infty(t_n,t_{n+1};L^2)}
\lesssim 
\|v^n\|_{L^2}+\|v^n\|_{L^2}^2
\lesssim
1+\|e^n\|_{L^2}+\|e^n\|_{L^2}^2
\lesssim
1+\|e^n\|_{L^2}^2 . 
\end{align}
This proves the first result of Lemma \ref{lem:e-R-pt}. 

By differentiating \eqref{en-itera} with respect to $t\in(t_n,t_{n+1}]$, we can find the following expression: 
\begin{align*}
\partial_t e(t)=\partial_x\Big[\fe^{-t\partial_x^3}e(t)\>\fe^{-t\partial_x^3}\Big(v(t)-\frac{1}{2}e(t)\Big)\Big]+\partial_t \mathcal R(t) . 
\end{align*}
Then, by applying the Sobolev embedding $L^1\hookrightarrow H^{-\frac12-}$ and the H\"older  inequality $\|fg\|_{L^1}\le \|f\|_{L^2}\|g\|_{L^2}$, we have 
\begin{align*}
\big\|\partial_te(t)\big\|_{H^{-\frac32-}} 
\lesssim &\, \|e(t)\|_{L^2}\big(\|v(t)\|_{L^2}+\|e(t)\|_{L^2}\big)
+\big\|\partial_t\mathcal R(t)\big\|_{H^{-\frac32-}}\\
\le &\, \tau^\gamma\ln(1/\tau) +\big(\|e\|_{L^\infty(0,t;L^2)}+|\ln(1/\tau)|^3\|e\|_{L^\infty(0,t;L^2)}^6\big)  .
\end{align*}
This proves the second result of Lemma \ref{lem:e-R-pt}. 
\end{proof}

\begin{proof}[Proof of Proposition \ref{prop:math-F}]
We first consider the low-frequency part of $\mathcal F(t)$, i.e.,  
\begin{align}\label{est:math-F-low}
\big\|\P_{\le N}\mathcal F(t)\big\|_{L^2}
\lesssim &\,
\bigg\|\int_0^t\fe^{s\partial_x^3}
  \P_{\le N}\partial_x\Big[\fe^{-s\partial_x^3}e(s)\>\fe^{-s\partial_x^3}\Big(v(s)-\frac{1}{2}e(s)\Big)\Big]\,\d s \bigg\|_{L^2} \notag\\
  \lesssim &\,
\int_0^t 
  N^2 \Big\|\partial_x^{-1}\Big[\fe^{-s\partial_x^3}e(s)\>\fe^{-s\partial_x^3}\Big(v(s)-\frac{1}{2}e(s)\Big)\Big]\Big\|_{L^2}  \,\d s \notag\\
\lesssim &\, N^2t \|e\|_{L^\infty(0,t;L^2)}\big(\|v\|_{L^\infty(0,t;L^2)}+\|e\|_{L^\infty(0,t;L^2)}\big)\notag\\
\lesssim &\, N^2t\big(\|e\|_{L^\infty(0,t;L^2)}+\|e\|_{L^\infty(0,t;L^2)}^2\big).
\end{align}

We then consider the high-frequency part of $\mathcal F(t)$ by using the integration-by-parts formula in Lemma \ref{lem:1-form}, which implies that 
\begin{subequations}\label{math-F-123}
\begin{align}
\mathcal F(t)=&\, \frac13\fe^{t\partial_x^3}
  \P\Big[\fe^{-t\partial_x^3}\partial_x^{-1}e(t)\>\fe^{-t\partial_x^3}\partial_x^{-1}\Big(v(t)-\frac12e(t)\Big)\Big]\label{math-F-1}\\
   &\, -\frac{1}{3}\int_0^t\fe^{s\partial_x^3} \P\Big[\fe^{-s\partial_x^3}\partial_x^{-1}\partial_s e(s)\>\fe^{-s\partial_x^3}\partial_x^{-1}\Big(v(s)-e(s)\Big)\Big]\d s\label{math-F-2}\\
   &\, -\frac{1}{3}\int_0^t\fe^{s\partial_x^3} \P\Big[\fe^{-s\partial_x^3}\partial_x^{-1}e(s)\>\fe^{-s\partial_x^3}\partial_x^{-1}\partial_sv(s)\Big]\d s.\label{math-F-3}
\end{align}
\end{subequations}
The first term on the right-hand side of \eqref{math-F-123} can be estimated by using the Sobolev and Bernstein inequalities, i.e., 
\begin{align*}
\big\|\P_{\ge N}\eqref{math-F-1}\big\|_{L^2}
&\lesssim N^{-1}\Big\|\fe^{-t\partial_x^3}\partial_x^{-1}e(t)\>\fe^{-t\partial_x^3}\partial_x^{-1}\Big(v(t)-\frac12e(t)\Big)\Big\|_{H^1} \\
&\lesssim N^{-1} \|e\|_{L^2}\big(\|v\|_{L^2}+\|e\|_{L^2}\big) \\ 
&\lesssim N^{-1} \big(\|e\|_{L^2}+\|e\|_{L^2}^2\big).
\end{align*}
In view of the relation $e(t)= e^0 + \mathcal F(t)+ \mathcal R(t)$, as shown in \eqref{def:es}, 
the second term on the right-hand side of \eqref{math-F-123} can be decomposed into the following two parts: 
\begin{subequations}\label{math-F-2-123}
\begin{align}
\eqref{math-F-2}=&\, -\frac{1}{3}\int_0^t\fe^{s\partial_x^3} \P\left(\fe^{-s\partial_x^3}\partial_x^{-1}\partial_s \mathcal F(s)\>\fe^{-s\partial_x^3}\partial_x^{-1}\big(v(s)-e(s)\big)\right)\d s\label{math-F-2-1}\\
   &\,  -\frac{1}{3}\int_0^t\fe^{s\partial_x^3} \P\left(\fe^{-s\partial_x^3}\partial_x^{-1}\partial_s \mathcal R(s)\>\fe^{-s\partial_x^3}\partial_x^{-1}\big(v(s)-e(s)\big)\right)\d s\label{math-F-2-2}.
\end{align}
\end{subequations}
where \eqref{math-F-2-1} can be furthermore expressed as follows by using relation \eqref{def:math-F}: 
\begin{align*}
\eqref{math-F-2-1}=& -\frac{1}{3}\int_0^t\fe^{s\partial_x^3} \P\Big(\P\Big[\fe^{-s\partial_x^3}e(s)\>\fe^{-s\partial_x^3}\Big(v(s)-\frac{1}{2}e(s)\Big)\Big]\>\fe^{-s\partial_x^3}\partial_x^{-1}\big(v(s)-e(s)\big)\Big)\d s.
\end{align*}
By considering $\P_{\ge N} \eqref{math-F-2-1}$, the frequency of at least one of the three terms $e(s)$, $v(s)-\frac{1}{2}e(s)$ and $v(s)-e(s)$, should be greater than or equal to $N/3$. Then, by applying Proposition \ref{lem:tri-linear} (1) to the high-frequency part of the expression above and using Bernstein's inequality for high-frequency functions, i.e., the second inequality of \eqref{Bernstein} with $s_0=-\frac{23}{14}$ and $s=0$, we obtain  
\begin{align*}
\big\|\P_{\ge N} \eqref{math-F-2-1}\big\|_{L^2}
\lesssim &\, t \|e\|_{X^0([0,t])}\|v-\mbox{$\frac{1}{2}$}e\|_{X^0([0,t])}\|v-e\|_{X^0([0,t])}\\
&\, +N^{-\frac{23}{14}}\|e\|_{L^\infty(0,t;L^2)}\|v-\mbox{$\frac{1}{2}$}e\|_{L^\infty(0,t;L^2)}\|v-e\|_{L^\infty(0,t;L^2)}  \\ 
\lesssim &\, t \|e\|_{X^0([0,t])}\big( \|v\|_{X^0([0,t])}^2+\|e\|_{X^0([0,t])}^2 \big)\\
&\, +N^{-\frac{23}{14}} \|e\|_{L^\infty(0,t;L^2)}\big(\|v\|_{L^\infty(0,t;L^2)}^2+\|e\|_{L^\infty(0,t;L^2)}^2\big).
\end{align*}
From \eqref{vt eq} it is straightforward to derive that $\|v\|_{X^0([0,t])}\lesssim 1$, 
and from Lemma \ref{lem:e-R-pt} (2) we know that 
\begin{align}\label{est:e-v-X0}
\|e\|_{X^0([0,t])}\le \tau^\gamma\ln(1/\tau) + 
\big(\|e\|_{L^\infty(0,t;L^2)}+|\ln(1/\tau)|^7 \|e\|_{L^\infty(0,t;L^2)}^8\big) . 
\end{align}
Therefore, we have that 
\begin{align*}
\big\|\P_{\ge N} \eqref{math-F-2-1}\big\|_{L^2}
\lesssim t\tau^\gamma \ln(1/\tau) 
&\,  +
(t+N^{-\frac{23}{14}})\big(\|e\|_{L^\infty(0,t;L^2)}+|\ln(1/\tau)|^{23} \|e\|_{L^\infty(0,t;L^2)}^{24}\big) .
\end{align*}

We substitute expression \eqref{ps-math-R} into \eqref{math-F-2-2}. This yields  
\begin{align}\label{math-F-2-2-R5}
\eqref{math-F-2-2}= &\, \mathcal R_{5}^*(t) + \sum\limits_{j=0}^{n-1}\mathcal R_{5}^*(t_{j+1}) \quad\mbox{for}\,\,\, t\in(t_n,t_{n+1}] , 
\end{align} 
where
\begin{align*}
\mathcal R_{5}^*(t)= &\, \frac{1}{6}\int_{t_n}^t\fe^{t\partial_x^3} \P\Big(\P\Big(\fe^{-s\partial_x^3}\big(\mathscr V(s)-v^n-F^n[s;v^n]\big)\\
&\,\qquad\qquad\qquad\qquad \cdot \fe^{-s\partial_x^3}\big(\mathscr V(s)+v^n+F^n[s;v^n]\big)\Big)\cdot\fe^{-s\partial_x^3}\partial_x^{-1}\big(v(s)-e(s)\big)\Big)\,\d s \\
&\, +\frac13\tau^{-1} \int_{t_n}^t \P\Big(\fe^{-s\partial_x^3}\partial_x^{-1}\mathcal R_{2}^*(t_{n+1}) \cdot\fe^{-s\partial_x^3}\partial_x^{-1}\big(v(s)-e(s)\big)\Big)\,\d s
\quad\mbox{for}\,\,\, t\in(t_n,t_{n+1} ] . 
\end{align*}
which can be estimated by using Proposition \ref{lem:tri-linear} (1), i.e.,  
\begin{align*}
\big\|\mathcal R_{5}^*(t)\big\|_{L^2}\lesssim &\, \tau \|\mathscr V(s)-v^n-F^n[s;v^n]\|_{X^0([t_n, t_{n+1}])}\\
&\quad\cdot \|\mathscr V(t)+v^n+F^n[s;v^n]\|_{X^0([t_n, t_{n+1}])}\big(\|v\|_{X^0([0,t])}+\|e\|_{X^0([0,t])}\big)\\
&\,+\|\mathscr V(t)-v^n-F^n[s;v^n]\|_{L^\infty(t_n, t_{n+1};H^{-\frac{23}{14}})}\\
&\quad\cdot \|\mathscr V(t)+v^n+F^j[s;v^n]\|_{L^\infty(t_n, t_{n+1};L^2)}\big(\|v\|_{L^\infty(0,t;L^2)}+\|e\|_{L^\infty(0, t;L^2)}\big)\\ 
&\,+ \|\mathcal R_2^*(t_{n+1})\|_{L^2}\big(\|v\|_{L^\infty(0,t;L^2)}+\|e\|_{L^\infty(0,t;L^2)}\big).
\end{align*} 
Lemma \ref{lem:V-increm-est} says that 
\begin{align*}
\big\|\mathscr V(s)-v^n-F^n[s;v^n]\big\|_{X^0([t_n, t_{n+1}])}
&\lesssim  \tau^{\gamma}\ln(1/\tau)+ \big(\|e^n\|_{L^2}+|\ln(1/\tau)|^3\|e^n\|_{L^2}^4\big) , \\
\big\|\mathscr V(s)-v^n-F^n[s;v^n]\big\|_{L^\infty(t_n, t_{n+1};H^{-\frac{23}{14}})}
&\lesssim \tau^{1+\gamma}\ln(1/\tau) + \tau \big(\|e^n\|_{L^2}+|\ln(1/\tau)|^3\|e^n\|_{L^2}^4\big)  .
\end{align*} 
The above result and \eqref{v+F}, together with the triangle inequality, imply that  
\begin{align*}
&\big\|\mathscr V(s) +v^n+F^n[s;v^n]\big\|_{X^0([t_n, t_{n+1}])}\\
&\lesssim \big\|\mathscr V(s)-v^n-F^n[s;v^n]\|_{X^0([t_n, t_{n+1}])}+2\big\|v^n+F^n[s;v^n]\big\|_{X^0([t_n, t_{n+1}])}\\
&\lesssim \tau^{\gamma}\ln(1/\tau)+ \|e^n\|_{L^2}+|\ln(1/\tau)|^3 \|e^n\|_{L^2}^4 
+ 1 + \|e^n\|_{L^2}^2 \\
&\lesssim 1 + |\ln(1/\tau)|^4 \|e^n\|_{L^2}^4  . 
\end{align*}
From \eqref{est:e-v-X0} we also know that 
$$
\|v\|_{X^0([0,t])}+\|e\|_{X^0([0,t])}
\lesssim 1 + |\ln(1/\tau)|^8 \|e^n\|_{L^2}^8 . 
$$
These estimates, together with \eqref{est:e-v-X0} and Proposition \ref{prop:est-R_n_2}, give the the following result: 
\begin{align*}
\big\|\mathcal R_{5}^*(t)\big\|_{L^2}
\le \,& \big[\tau^{1+\gamma}\ln(1/\tau) + \tau \|e^n\|_{L^2}\big(1+|\ln(1/\tau)|^3\|e^n\|_{L^2}^3\big) \big] \\
&\,\, \cdot (1 + |\ln(1/\tau)|^{4}\|e\|_{L^\infty(0,t;L^2)}^{4}) (1 + |\ln(1/\tau)|^{8}\|e\|_{L^\infty(0,t;L^2)}^{8}) . 
\end{align*}
Substituting this into \eqref{math-F-2-2-R5} yields that 
\begin{align*}
\big\| \eqref{math-F-2-2} \big\|_{L^2} 
\lesssim 
t\tau^\gamma\ln(1/\tau) + t \|e\|_{L^\infty(0,t;L^2)}
\big(1+|\ln(1/\tau)|^{96}\|e\|_{L^\infty(0,t;L^2)}^{96}\big).
\end{align*}

By substituting the estimates of $\big\|\P_{\ge N} \eqref{math-F-2-1}\big\|_{L^2}$ and $\big\| \eqref{math-F-2-2} \big\|_{L^2} $ into \eqref{math-F-2-123}, we obtain 
\begin{align*}
\big\|\P_{\ge N} \eqref{math-F-2}\big\|_{L^2}\lesssim 
t\tau^\gamma\ln(1/\tau)  +\big(t+N^{-\frac{23}{14}}\big)\|e\|_{L^\infty(0,t;L^2)}
\big(1+|\ln(1/\tau)|^{96}\|e\|_{L^\infty(0,t;L^2)}^{96}\big).
\end{align*}

Finally, by using the expression of $\partial_tv(s)$ in \eqref{vt eq}, the last term on the right-hand side of \eqref{math-F-123} can be rewritten as follows:  
\begin{align*}
 \eqref{math-F-3}= -\frac{1}{6}\int_0^t\fe^{s\partial_x^3} \P\left(\P\left(\fe^{-s\partial_x^3}v(s)\right)^2\>\fe^{-s\partial_x^3}\partial_x^{-1}e(s)\right)\d s.
\end{align*}
Again, by considering $\P_{\ge N}  \eqref{math-F-3}$, the frequency of at least one of the three terms $v(s)$, $v(s)$ and  $e(s)$, should be greater than or equal to $N/3$. Then, by applying Proposition \ref{lem:tri-linear} (1) to the high-frequency part of the expression above together with \eqref{est:e-v-X0}, and using Bernstein's inequality for high-frequency functions, i.e., the second inequality of \eqref{Bernstein} with $s_0=-\frac{23}{14}$ and $s=0$, we obtain  
\begin{align*}
\big\|\P_{\ge N} \eqref{math-F-3}\big\|_{L^2}
\lesssim &\, t \|e\|_{X^0([0,t])}\|v\|_{X^0([0,t])}^2 + N^{-\frac{23}{14}} \|e\|_{L^\infty(0,t;L^2)}\|v\|_{L^\infty(0,t;L^2)}^2\\
\le &\, t\tau^\gamma\ln(1/\tau) + 
(t+N^{-\frac{23}{14}})\|e\|_{L^\infty(0,t;L^2)} \big(1+|\ln(1/\tau)|^7 \|e\|_{L^\infty(0,t;L^2)}^7\big).
\end{align*}

Combining the estimates of $\P_{\ge N}$\eqref{math-F-1}, \eqref{math-F-2} and $\P_{\ge N}$\eqref{math-F-3}, yields a desired estimate of $\big\|\P_{\ge N}\mathcal F(t)\big\|_{L^2}$, which together with the estimate of $\big\|\P_{\le N}\mathcal F(t)\big\|_{L^2}$ in \eqref{est:math-F-low} implies the result of Proposition \ref{prop:math-F}. 
\end{proof}

\section{Error estimates (Proof of Theorem \ref{main:thm1})} \label{sec:globalstability} 

By using the relation $ e(t)= e^0 + \mathcal F(t)- \mathcal R(t) $ in \eqref{def:es}, and the estimates of $\mathcal R(t)$ and $\mathcal F(t)$ in Propositions \ref{prop:est-R} and \ref{prop:math-F}, respectively, we obtain  
\begin{align*}
\|e(t)\|_{L^2}\le 
 \|e^0\|_{L^2}
&\, +C (t+\tau)\tau^\gamma  \ln(1/\tau)
+ CtN^2\|e\|_{L^\infty(0,t;L^2)}\big(1+\|e\|_{L^\infty(0,t;L^2)}\big)\\
&\, + C\big(t+\tau+N^{-\frac{23}{14}}\big)\|e\|_{L^\infty(0,t;L^2)}\big(1+|\ln(1/\tau)|^{96}\|e\|_{L^\infty(0,t;L^2)}^{96}\big) 
\quad\forall\, t\in [0,T]. 
\end{align*}
Choosing $N=(t+\tau)^{-\frac14}$ in the inequality above, we have  
\begin{align}\label{error_inequality}
\|e(t)\|_{L^2}\le \|e^0\|_{L^2}
+ C_1\tau^\gamma \ln(1/\tau) +C_2(t+\tau)^{\frac14} \|e\|_{L^\infty(0,t;L^2)}\big(1+|\ln(1/\tau)|^{96}\|e\|_{L^\infty(0,t;L^2)}^{96}\big)  . 
\end{align}
Since $e^0=0$ and $\|e\|_{L^\infty(0,t;L^2)}$ is a continuous function of $t\in[0,T]$, we may assume that $t_*\in(0,T]$ is the maximal time such that 
$$ \|e\|_{L^\infty(0,t_*;L^2)}\le \tau^{\frac{\gamma}{2}} .$$ 

If $t_*=T$ then we set $\delta=0$. 

If $t_*<T$ then there exists a positive constant $\delta>0$ such that $\|e\|_{L^\infty(0,t_*+\delta;L^2)}\le 2\tau^{\frac{\gamma}{2}}$ according to the continuity of $\|e\|_{L^\infty(0,t;L^2)}$ with respect to  $t\in[0,T]$. 

In either case, we can rewrite \eqref{error_inequality} as follows (regarding $s$ as the initial time): 
\begin{align*}
\|e\|_{L^\infty(s,t;L^2)} 
& \le \|e(s)\|_{L^2}
+ C_1\tau^\gamma \ln(1/\tau) + C_3(t-s+\tau)^{\frac14} \|e\|_{L^\infty(s,t;L^2)} \\
& \le \|e\|_{L^\infty(0,s;L^2)} 
+ C_1\tau^\gamma \ln(1/\tau) + C_3(t-s+\tau)^{\frac14} \|e\|_{L^\infty(s,t;L^2)}, \quad  0\le s\le t\le t_*+\delta , 
\end{align*}
which implies that 
\begin{align*}
\|e\|_{L^\infty(0,t;L^2)} 
& \le \|e\|_{L^\infty(0,s;L^2)} 
+ C_1\tau^\gamma \ln(1/\tau) + C_3(t-s+\tau)^{\frac14} \|e\|_{L^\infty(0,t;L^2)}, \quad 0\le s\le t\le t_*+\delta . 
\end{align*}

Let $T_0=\frac12(2C_3)^{-4}$, $\tau\in(0,T_0]$, and consider a sequence of intervals $[kT_0,(k+1)T_0]$, $k=0,\dots,m,$ such that $mT_0<T\le (m+1)T_0$. The maximal number of such intervals is bounded, i.e., $m\le T/T_0$, which is independent of the stepsize $\tau$. On every subinterval $[kT_0,(k+1)T_0]$ such that $[kT_0,(k+1)T_0]\cap[0, t_*+\delta] \ne\emptyset$, we have 
\begin{align*}
\|e\|_{L^\infty(0,t;L^2)} 
& \le \|e\|_{L^\infty(0,kT_0;L^2)} 
+ C_1\tau^\gamma \ln(1/\tau) + \frac12 \|e\|_{L^\infty(0,t;L^2)} \quad\forall\, t\in[kT_0,(k+1)T_0]\cap[0, t_*+\delta] , 
\end{align*}
which implies that
\begin{align*}
\|e\|_{L^\infty(0,t;L^2)} & \le 2 \|e\|_{L^\infty(0,kT_0;L^2)} 
+ 2C_1\tau^\gamma \ln(1/\tau)  \quad\forall\, t\in[kT_0,(k+1)T_0]\cap[0, t_*+\delta] . 
\end{align*}
Iterating this inequality for $k=0,1,\dots$, yields that 
\begin{align*}
\|e\|_{L^\infty(0,t_*+\delta;L^2)} & \le 2^{m+1} \|e^0\|_{L^2} 
+ 2^{m+2}C_1\tau^\gamma \ln(1/\tau) = 2^{m+2}C_1\tau^\gamma \ln(1/\tau)  . 
\end{align*}
Since $m$ and $C_1$ are independent of $\tau$, it follows that there exists a positive constant $\tau_0$ such that for $\tau\le \tau_0$ we have 
\begin{align*}
\|e\|_{L^\infty(0,t_*+\delta;L^2)} \le \tau^{\frac{\gamma}{2}} . 
\end{align*}
This contradicts the maximality of $t_*\in(0,T]$ unless $t_*=T$. Therefore, $t_*=T$, $\delta=0$, and  
\begin{align*}
\|e\|_{L^\infty(0,T;L^2)} 
\le 2^{m+2}C_1\tau^\gamma \ln(1/\tau)  . 
\end{align*}
This proves the error estimate in Theorem \ref{main:thm1}.
\qed

\section{Numerical experiments} \label{sec:numerical}

In this section, we present numerical experiments on the convergence of the proposed low-regularity integrator for the KdV equation with $H^\gamma$ initial data, for $\gamma=0.2$, $0.4$, $0.6$ and $0.8$, respectively. The computations are performed by Matlab with double precision. 

We consider the KdV equation
\begin{equation}
\left\{
\begin{array}{lll}
{\displaystyle \partial_{t}u(t,x)+\partial^{3}_{x}u(t,x)=\frac{1}{2}\partial_{x}(u(t,x))^{2}\quad \text{for }x\in\mathbb{T}\text{ and }t>0,} \\[2mm]
{\displaystyle u(0,x)=u^{0}(x)\quad \text{for }x\in\mathbb{T},}
\end{array}
\right.
\end{equation}
with the following initial value: 
\begin{equation}
{\displaystyle u^{0}(x)=\frac{1}{10}\sum_{0\neq k\in\mathbb{Z}}|k|^{-0.51-\gamma}\fe^{ikx}.}
\end{equation}
which is in $H^\gamma(\T)$ but not in $H^{\gamma+0.01}(\T)$. 
We present the errors of the numerical solutions at $T=1$ in Figure~\ref{fig:0-1}, with sufficiently large degrees of freedom (i.e., dof=$2^{12}$) in the spatial discretization by a spectral method. The reference solution is obtained by using the proposed low-regularity integrator with a much smaller time stepsize. From Figure~\ref{fig:0-1} we see that the convergence order of the proposed method with $H^\gamma$ initial value is $\gamma$, for $\gamma=0.2$, $0.4$, $0.6$ and $0.8$, respectively. This is consistent with the theoretical result proved in Theorem \ref{main:thm1}. 

\begin{figure}[h]
\begin{minipage}{0.48\linewidth}\centerline{{\tiny(a)}\includegraphics[width=8cm,height=6cm]{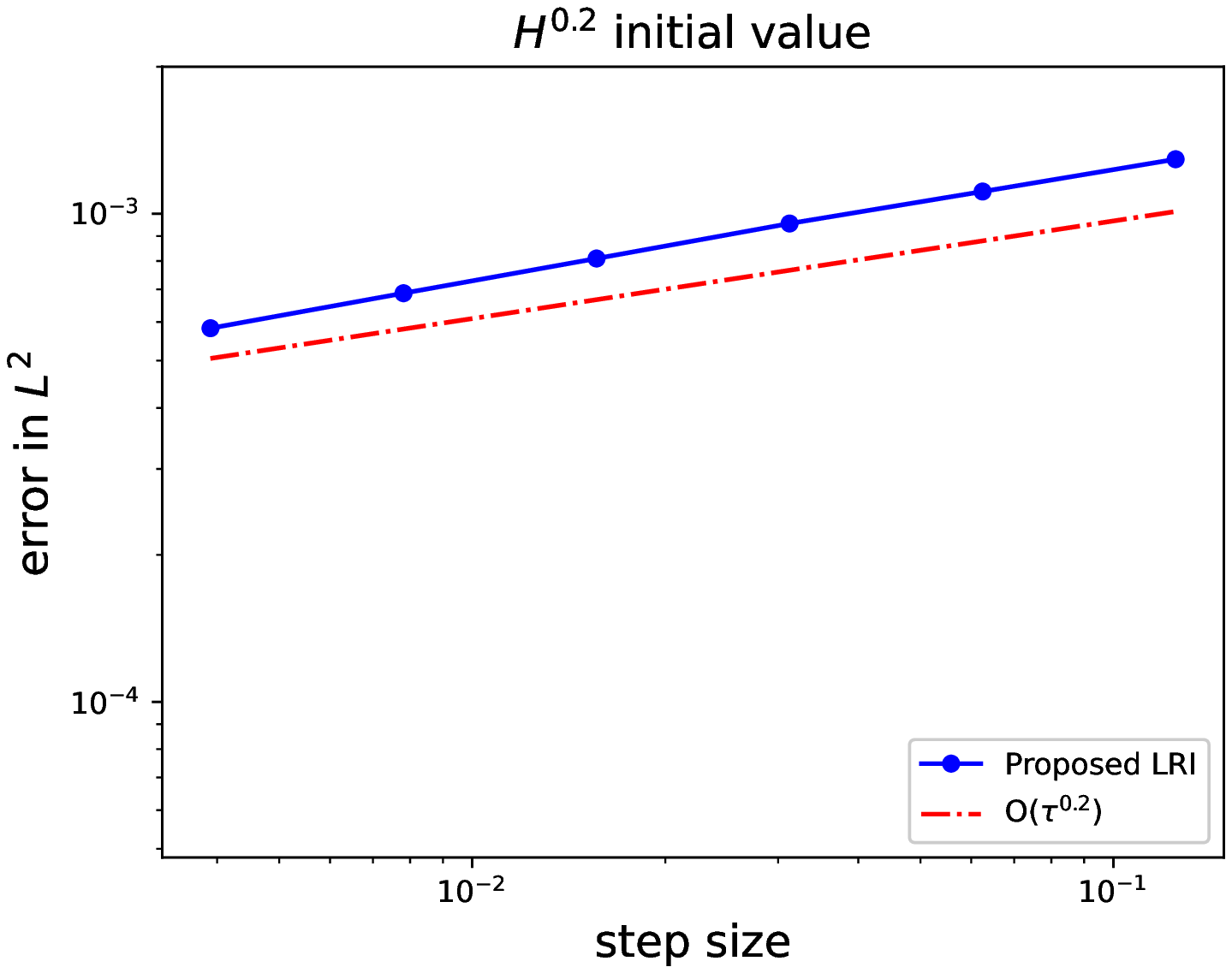}}
\end{minipage}
\hfill
\begin{minipage}{0.48\linewidth}\centerline{{\tiny(b)}\includegraphics[width=8cm,height=6cm]{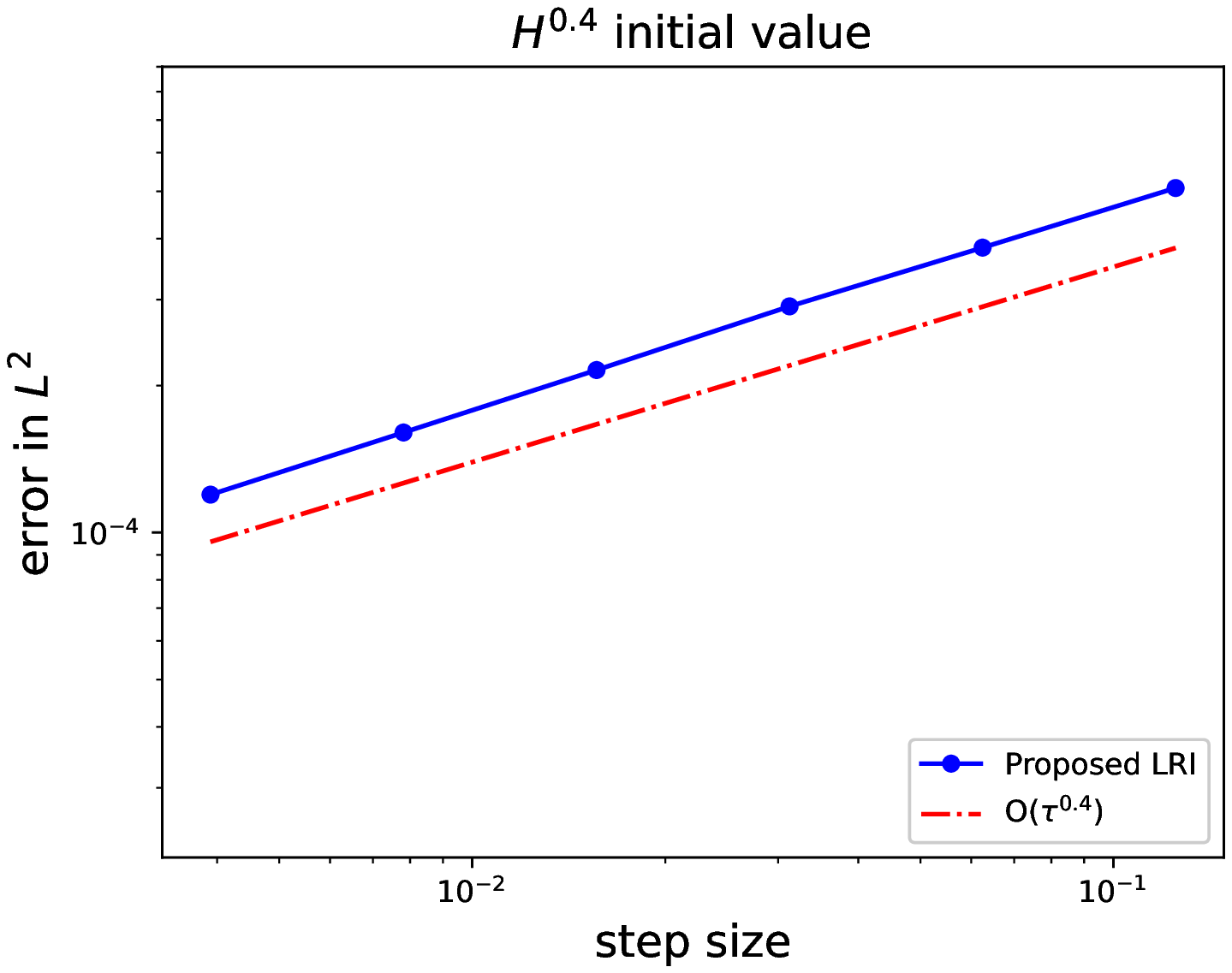}}
\end{minipage}
\vfill
\begin{minipage}{0.48\linewidth}\centerline{{\tiny(c)}\includegraphics[width=8cm,height=6cm]{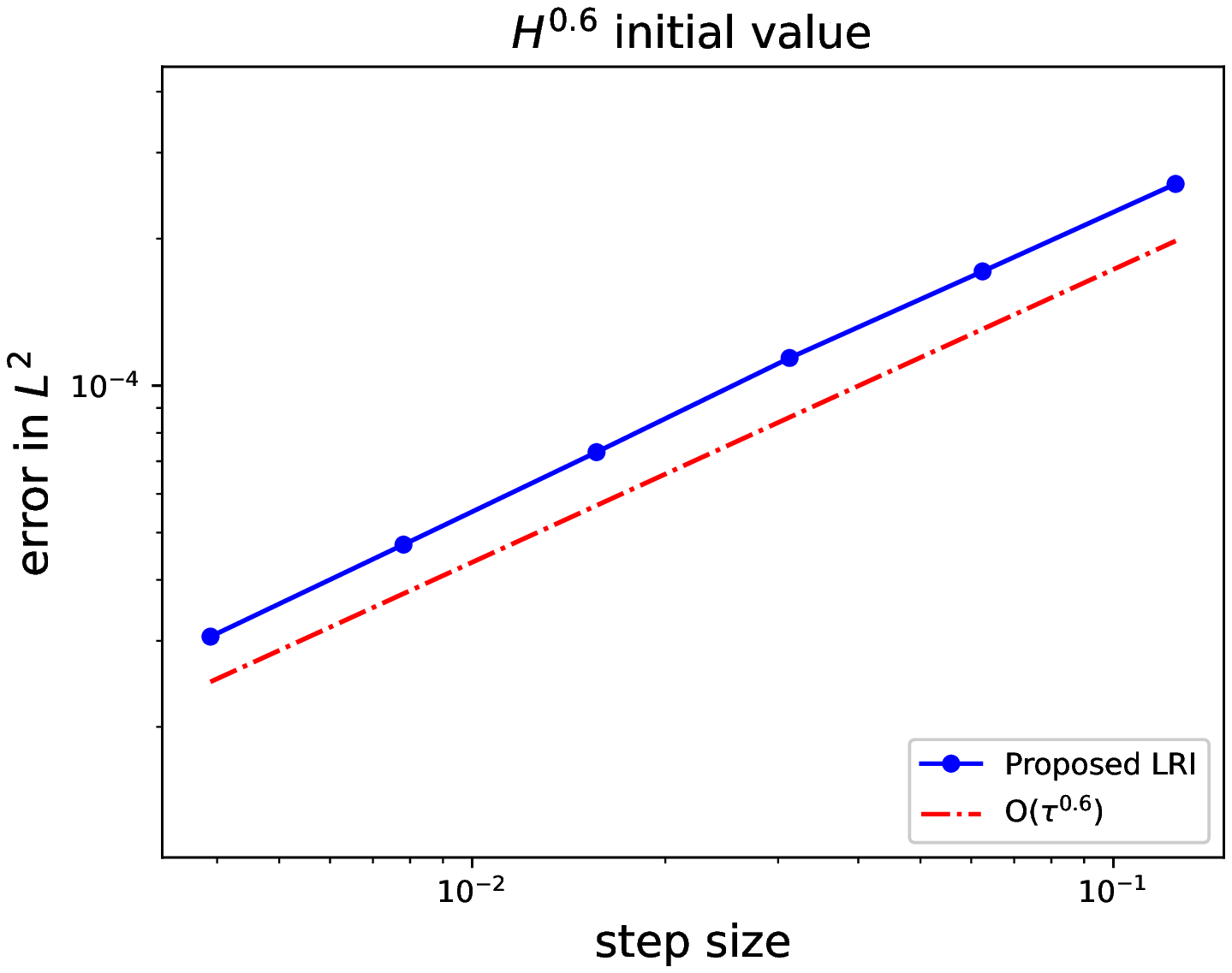}}
\end{minipage}
\hfill
\begin{minipage}{0.48\linewidth}\centerline{{\tiny(d)}\includegraphics[width=8cm,height=6cm]{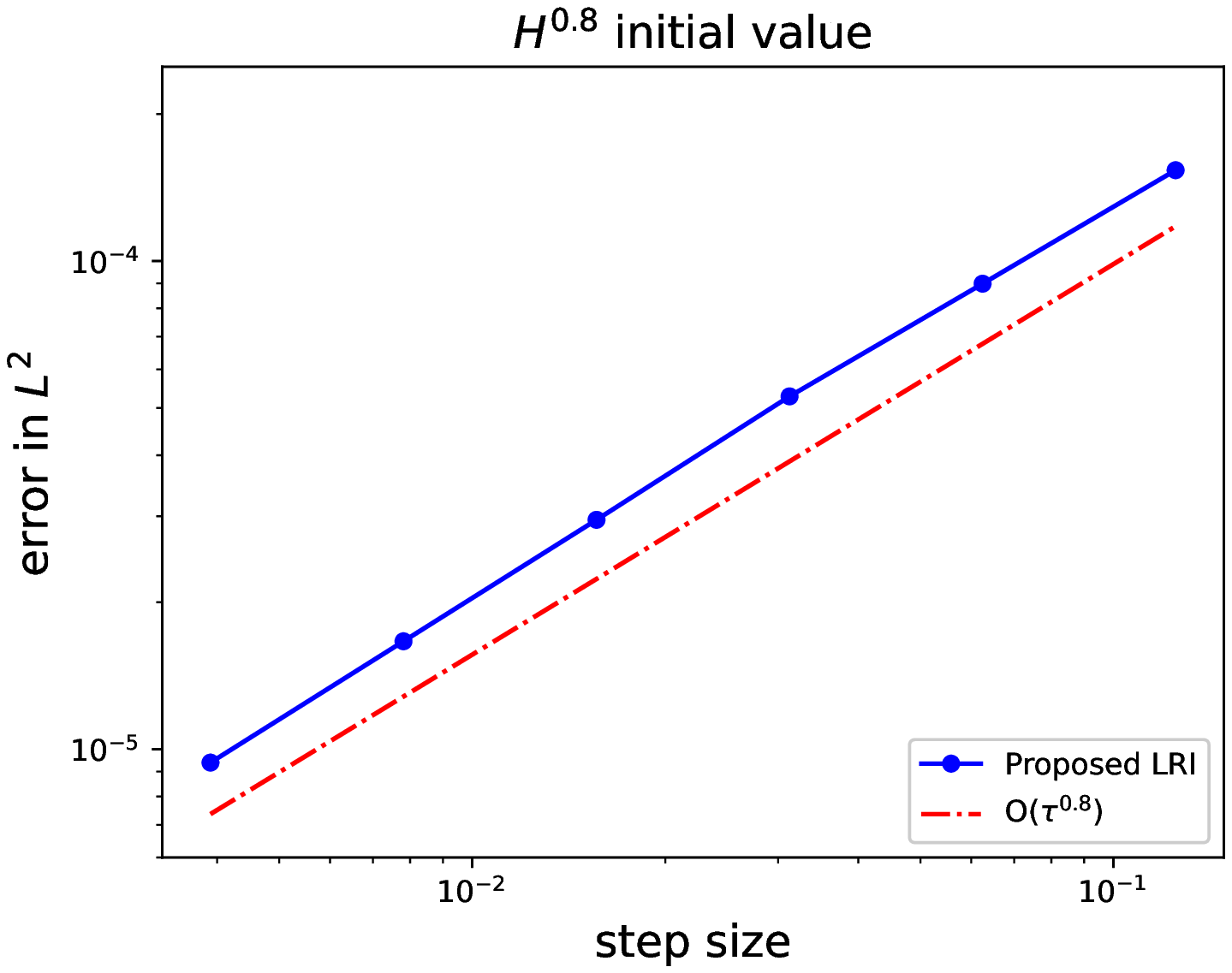}}
\end{minipage}
\caption{$L^2$ errors of the numerical solutions for $H^{\gamma}$ initial data, with $\gamma\in(0,1]$.}
\label{fig:0-1}
\end{figure}

 \vspace{-15pt}

\section{Conclusions} \label{sec:conclusion}

We have presented several new tools for the construction and analysis of low-regularity integrators for the KdV equations, including the averaging approximation technique for exponential functions with imaginary powers (Lemma \ref{lem:average2}), the new estimates for the symbol $\phi=k^3-k_1^3-k_2^3-k_3^3$ (Lemma \ref{lem:a4-est}), and new trilinear estimates associated to the KdV operator (Proposition \ref{lem:tri-linear}). These new techniques have played essential roles in analyzing the local error, global remainder, and the stability estimates. We have also introduced a new technique, which reformulates the numerical scheme into a perturbed integral formulation of the continuous KdV equation globally posed on the time interval $[0,T]$, instead of locally posed on $[t_n,t_{n+1}]$, for analyzing the stability of numerical approximations to solutions below $H^1$. By combining the several new techniques, we have constructed a new time discretization which is convergent with order $\gamma$ in $L^2$ (up to a logarithmic factor) under the regularity condition $u\in C([0,T];H^\gamma)$, with $\gamma\in(0,1]$ possibly approaching zero. The techniques and framework established in this article may also be applied and extended to the construction of unfiltered low-regularity integrators for other nonlinear dispersive equations, including the modified KdV equation, the generalized KdV equation, and the nonlinear Schr\"odinger equation. 

\section*{Acknowledgements}
The work of B. Li is partially supported by the Hong Kong Research Grants Council (General Research Fund, project no. PolyU15300519) and the internal grant at The Hong Kong Polytechnic University (PolyU Project ID: P0031035, Work Programme: ZZKQ). The work of Y. Wu is partially supported by NSFC 12171356.

\bibliographystyle{model1-num-names}

\end{document}